\documentclass[11pt]{amsart}
\usepackage{mathrsfs}
\usepackage{amsfonts}
\usepackage{amssymb}
\usepackage{amsxtra}
\usepackage{dsfont}
\usepackage{color}
\usepackage{graphicx}
\usepackage[english,polish]{babel}
\usepackage{enumitem}

\usepackage{newtxmath}
\usepackage{newtxtext}

\usepackage[margin=2.5cm, centering]{geometry}

\usepackage[colorlinks,citecolor=blue,urlcolor=blue,bookmarks=true]{hyperref}

\hypersetup{
pdfpagemode=UseNone,
pdfstartview=FitH,
pdfdisplaydoctitle=true,
pdfborder={0 0 0}, 
pdftitle={Variational estimates for discrete operators modeled on multi-dimensional polynomial subsets of primes},
pdfauthor={Bartosz Trojan},
pdflang=en-US
}

\newcommand{\CC}{\mathbb{C}}
\newcommand{\ZZ}{\mathbb{Z}}
\newcommand{\NN}{\mathbb{N}}
\newcommand{\RR}{\mathbb{R}}
\newcommand{\PP}{\mathbb{P}}

\newcommand{\TT}{\mathbb{T}}

\newcommand{\calC}{\mathcal{C}}

\newcommand{\calF}{\mathcal{F}}

\newcommand{\calO}{\mathcal{O}}

\newcommand{\calB}{\mathcal{B}}

\newcommand{\calA}{\mathcal{A}}

\newcommand{\calD}{\mathcal{D}}

\newcommand{\calP}{\mathcal{P}}
\newcommand{\calQ}{\mathcal{Q}}

\newcommand{\bfA}{\mathbf{A}}

\newcommand{\scrU}{\mathscr{U}}
\newcommand{\scrA}{\mathscr{A}}
\newcommand{\scrH}{\mathscr{H}}
\newcommand{\scrM}{\mathscr{M}}
\newcommand{\frkH}{\mathfrak{h}}
\newcommand{\frkM}{\mathfrak{m}}
\newcommand{\frkY}{\mathfrak{y}}

\newcommand{\lcm}{\operatorname{lcm}}

\newcommand{\pl}[1]{\foreignlanguage{polish}{#1}}

\newcommand{\norm}[1]{\lvert {#1} \rvert}
\newcommand{\abs}[1]{\left\lvert {#1} \right\rvert}
\newcommand{\sprod}[2]{{#1}\cdot {#2}}

\newcommand{\ind}[1]{{\mathds{1}_{{#1}}}}

\renewcommand{\atop}[2]{\genfrac{}{}{0pt}2{#1}{#2}}

\newtheorem{theorem}{Theorem}
\newtheorem{proposition}{Proposition}[section]
\newtheorem{lemma}{Lemma}

\newtheorem{claim}{Claim}
\newtheorem*{theorem*}{Theorem}

\numberwithin{equation}{section}

\newcounter{thm}

\newtheorem{main_theorem}[thm]{Theorem}

\theoremstyle{definition}

\title[Variational estimates for discete operators]
{Variational estimates for discrete operators modeled on multi-dimensional polynomial subsets of primes}

\author{Bartosz Trojan}
\address{
	\pl{
	Bartosz Trojan\\
	The Institute of Mathematics\\
	Polish Academy of Science\\
	ul. \'Sniadeckich 8\\
	00-696 Warszawa\\
	Poland}
}
\email{btrojan@impan.pl}

\begin{document}
\selectlanguage{english}

\begin{abstract}
	We prove the extensions of Birkhoff's and Cotlar's ergodic theorems to multi-dimensional polynomial
	subsets of prime numbers $\PP^k$. We deduce them from $\ell^p\big(\ZZ^d\big)$-boundedness of $r$-variational
	seminorms for the corresponding discrete operators of Radon type, where $p > 1$ and $r > 2$.
\end{abstract}

\maketitle

\section{Introduction}
Let $(X, \calB, \mu)$ be a $\sigma$-finite measure space with $d_0$ invertible commuting and measure
preserving transformations $T_1, \ldots, T_{d_0}: X \rightarrow X$. Let
$\calP = \big(\calP_1, \ldots, \calP_{d_0}\big):\RR^k \rightarrow \RR^{d_0}$ denote a polynomial mapping such that
each $\calP_j$ is a polynomial on $\RR^k$ having integer coefficients without a constant term. Let $B$ be an open
bounded convex subset in $\RR^k$ containing the origin such that for some $\iota > 0$ and all $N \in \NN$,
\begin{equation}
	\label{eq:151}
	[-\iota N, \iota N]^k \subseteq B_N \subseteq [-N, N]^k,
\end{equation}
where for $\lambda > 0$, we have set
\[
	B_\lambda = \left\{x \in \RR^k : \lambda^{-1} x \in B\right\}.
\]
In this paper we consider the following averages
\[
	\scrA_N^\calP f(x) = \frac{1}{\pi_B(N)} \sum_{n \in \NN^{k'}} \sum_{p \in \PP^{k''}}
	f\Big(T_1^{\calP_1(n, p)} \cdots T_{d_0}^{\calP_{d_0}(n, p)} x\Big) \ind{B_N}(n, p),
\]
where $k = k'+k''$, $\PP$ denotes the set of prime numbers, and
\[
	\pi_B(N) = \sum_{n \in \NN^{k'}} \sum_{p \in \PP^{k''}} \ind{B_N}(n, p).
\]
One of the results of this article establishes the following theorem.
\begin{main_theorem}
	\label{main_thm:1}
	Assume that $p \in (1, \infty)$. For every $f \in L^p(X, \mu)$ there exists $f^* \in L^p(X, \mu)$
	such that
	\[
		\lim_{N \to \infty} \scrA^\calP_N f(x) = f^*(x),
	\]
	for $\mu$-almost all $x \in X$.
\end{main_theorem}
Sums over prime numbers are irregular, thus it is more convenient to work with weighted averaging operators,
\[
	\scrM_N^\calP f(x) = \frac{1}{\vartheta_B(N)} 
	\sum_{n \in \NN^{k'}} \sum_{p \in \PP^{k''}} 
	f\left(T_1^{\calP_1(n, p)} \cdots T_{d_0}^{\calP_{d_0}(n, p)} x\right) \ind{B_N}(n, p)
	\bigg(\prod_{j = 1}^{k''} \log p_j\bigg),
\]
where
\[
	\vartheta_B(N) = \sum_{n \in \NN^{k'}} \sum_{p \in \PP^{k''}} \ind{B_N}(n, p) 
	\bigg(\prod_{j = 1}^{k''} \log p_j\bigg).
\]
Then the pointwise convergence of $(\scrA_N f : N \in \NN)$ can be deduced from the properties of 
$(\scrM_N f : N \in \NN)$, see Proposition \ref{prop:3} for details.

Next to the averaging operators we also study pointwise convergence of truncated discrete singular operators. To be more
precise, let $K \in \calC^1\big(\RR^k \setminus \{0\} \big)$ be a Calder\'on--Zygmund kernel satisfying the differential
inequality
\begin{equation}
	\label{eq:23}
	\norm{x}^k \abs{K(x)} + \norm{x}^{k+1} \norm{\nabla K(x)} \leq 1,
\end{equation}
for all $x \in \RR^k$ with $\norm{x} \geq 1$, and the cancellation condition
\begin{equation}
	\label{eq:24}
	\int_{B_{\lambda} \setminus B_{\lambda'}} K(x) {\: \rm d}x = 0,
\end{equation}
for every $0 < \lambda' \leq \lambda$. Then the truncated discrete singular operator $\scrH^\calP_N$ is defined as
\[
        \scrH^\calP_N f(x) = \sum_{n \in \ZZ^{k'}} \sum_{p \in (\pm \PP)^{k''}}
        f\Big(T_1^{\calP_1(n, p)} \cdots T_{d_0}^{\calP_{d_0}(n, p)} x\Big)
		K(n, p) \ind{B_N}(n, p) 
		\bigg(\prod_{j = 1}^{k''} \log \abs{p_j}\bigg). 
\] 
The logarithmic weights in $\scrM_N^\calP$ and $\scrH_N^\calP$ correspond to the density of prime numbers. 
In this article we prove the following theorem, which may be thought as an extension of Cotlar's ergodic theorem, 
see \cite{cot}. 
\begin{main_theorem}
	\label{main_thm:2}
	Assume that $p \in (1, \infty)$. For every $f \in L^p(X, \mu)$ there exists $f^* \in L^p(X, \mu)$ such that
	\[
		\lim_{N \to \infty} \scrH^\calP_N f(x) = f^*(x),
	\]
	for $\mu$-almost all $x \in X$.
\end{main_theorem}
The classical approach to the pointwise convergence in $L^p(X, \mu)$ proceeds in two steps. Namely, one needs
to show $L^p(X, \mu)$ boundedness of the corresponding maximal function reducing the problem to showing the convergence
on some dense class of $L^p(X, \mu)$ functions. However, finding such a class may be a difficult task. This is the case
of one dimensional averages along $(n^2 : n \in \NN)$ studied by Bourgain in \cite{bou}. To overcome this issue
Bourgain introduced the oscillation seminorm defined for a given lacunary sequence $(N_j : j \in \NN)$ and a sequence 
of complex numbers $(a_n : n \in \NN)$ as
\[
	\calO_J(a_n : n \in \NN)
	=
	\Big(\sum_{j = 1}^J \sup_{N_j \leq n \leq N_{j+1}} \abs{a_n - a_{N_j}}^2 \Big)^{1/2}.
\]
Then the pointwise convergence of $(\scrA_N f : N \in \NN)$ is reduced to showing that
\[
	\big\|\calO_J\big(\scrA_N f : N \in \NN\big)\big\|_{L^2} = o\big(J^{1/2}\big),
\]
while $J$ tends to infinity. In place of the oscillation seminorm, we investigate $r$-variational seminorm.
Let us recall that $r$-variational seminorm of a sequence $(a_n : n\in \NN)$ is defined by
\[
	V_r(a_n : n \in \NN) = \sup_{k_0 < k_2 < \ldots < k_J}
	\Big(\sum_{j = 1}^J \abs{a_{k_j} - a_{k_{j-1}}}^r\Big)^{1/r}.
\]
In fact, $r$-variational seminorm controls $\calO_J$ as well as the maximal function. Indeed, for any $r \geq 2$,
by H\"older's inequality we have
\[
	\calO_J(a_n : n \in \NN) \leq J^{\frac{1}{2}-\frac{1}{r}} V_r(a_n : n \in \NN).
\]
Moreover, for any $n_0 \in \NN$,
\[
	\sup_{n \in \NN} |a_n| \leq |a_{n_0}| + V_r(a_n : n \in \NN).
\]
Nevertheless, the main motivation to study $L^p(X, \mu)$ boundedness of $r$-variational seminorm is the following
observation: if $V_r(a_n : n \in \NN) < \infty$ for any $r \geq 1$ then the sequence $(a_n : n \in \NN)$ converges.
Therefore, we can deduce Theorem \ref{main_thm:1} and Theorem \ref{main_thm:2} from the following result.
\begin{main_theorem}
	\label{main_thm:4}
	For every $p \in (1, \infty)$ there is $C_p > 0$ such that for all $r \in (2, \infty)$ and all $f \in L^p(X, \mu)$,
	\begin{equation}
		\label{eq:141a}
		\big\|V_r\big(\scrM_N^\calP f : N \in \NN\big)\big\|_{L^p} \leq C_p \frac{r}{r-2} \|f\|_{L^p},
	\end{equation}
	and
	\begin{equation}
		\label{eq:142b}
		\big\|V_r\big(\scrH_N^\calP f : N \in \NN\big)\big\|_{L^p} \leq C_p \frac{r}{r-2} \|f\|_{L^p}.
	\end{equation}
	The constant $C_p$ is independent of the coefficients of the polynomial mapping $\calP$.
\end{main_theorem}
The variational estimates for discrete averaging operators have been the subject of many papers, see 
\cite{jkrw, K, mst2, msz-k, mt3, mtz, zk}. In \cite{K}, Krause studied the case $d_0 = k = k' = 1$ and has obtained
the inequality \eqref{eq:141a} for $p \in (1, \infty)$ and $r > \max\{p, p'\}$. On the other hand, Zorin-Kranich in 
\cite{zk} for the same case obtained \eqref{eq:141a} for all $r \in (2, \infty)$ but for $p$ in some vicinity of $2$.
Only recently in \cite{mst2} the variational estimates have been established in the full range of parameters, that is
$p \in (1, \infty)$ and $r \in (2, \infty)$, covering the case $k'' = 0$. In \cite{zk}, Zorin-Kranich has proved
\eqref{eq:141a} also for the averaging operators modeled on prime numbers, that is when $d_0 = k = k '' = 1$ with
a polynomial $P(n) = n$. It is worth mentioning that the variational estimates for discrete operators are based on
\emph{a priori} estimates for their continuous counterparts developed in \cite{jsw}, see also \cite[Appendix]{mst2}.

The variational estimates for discrete singular operators have been studied in \cite{cjrw, mst2, msz-k,mtz}. In \cite{mtz},
the authors obtained the inequality \eqref{eq:142b}, for the truncated Hilbert transform modeled on prime numbers, which
corresponds to $d_0 = k = k'' = 1$ and a polynomial $P(n) = n$. In fact, discrete singular operators of Radon type
required a new approach. An important milestone has been laid by Ionescu and Wainger in \cite{iw}. Ultimately,
the complete development of the discrete singular operators of Radon type has been obtained in \cite{mst2}.

Concerning pointwise ergodic theorems over prime numbers, there are some results using oscillation seminorms.
In \cite{bou-p}, Bourgain has shown pointwise convergence for the averages along prime numbers for functions from
$L^2(X, \mu)$. Then his result was extended to all $L^p(X, \mu)$, $p > 1$, by Wierdl in \cite{wrl}, see also
\cite[Section 9]{bou}. Not long afterwards, Nair in \cite{na1} has proved Theorem \ref{main_thm:1} for $L^2(X, \mu)$,
$d_0 = k = k'' = 1$, and any integer-valued polynomial. Nair also studied ergodic averages for functions in $L^p(X, \mu)$
for $p \neq 2$, however, \cite[Lemma 14]{na2} contains an error. In fact, the estimates on the
multipliers $W_N$ are insufficient to show that the sum considered at the end of the proof has bounds independent of
$\abs{\alpha-a/b}$. Lastly, the extension of Cotlar's ergodic theorem to prime numbers has been established in \cite{mt2},
see also \cite{mtz}.

In view of the Calder\'on transference principle, while proving Theorem \ref{main_thm:4}, we may work with
the model dynamical system, namely, $\ZZ^{d_0}$ with the counting measure and the shift operators. Let us denote by
$M_N^\calP$ and $H_N^\calP$, the corresponding operators, namely,
\begin{equation}
	\label{eq:131}
	M_N^\calP f(x) = 
	\frac{1}{\vartheta_B(N)}
	\sum_{n \in \NN^{k'}} \sum_{p \in \PP^{k''}}
	f\big(x - \calP(n, p)\big)
	\ind{B_N}(n, p) 
	\bigg(\prod_{j = 1}^{k''} \log p_j\bigg),
\end{equation}
and
\begin{equation}
	\label{eq:132}
	H^\calP_N f(x) = \sum_{n \in \ZZ^{k'}} \sum_{p \in (\pm \PP)^{k''}}
	f\big(x - \calP(n, p)\big) K(n, p) \ind{B_N}(n, p) \bigg(\prod_{j = 1}^{k''} \log \abs{p_j}\bigg).
\end{equation}
We now give some details about the method of the proof of Theorem \ref{main_thm:4} for the model dynamical system. 
To simplify the exposition we restrict attention to the averaging operators. Let us denote by $\frkM_N$ the discrete
Fourier multiplier corresponding to $M_N^\calP$. To deal with $r$-variational estimates we apply the method recently
used \cite{msz-k}, see also \cite{zk}. Namely, given $\rho \in (0, 1)$ we consider the set
$\calD_\rho = \{N_n : n \in \NN\}$, where $N_n = \big\lfloor 2^{n^\rho}\big\rfloor$. Then in view of \eqref{eq:90}
we can split the $r$-variation into two parts: long variations and short variations, and study them separately. 
For each $p \in (1, \infty)$ we can choose $\rho$ so that the estimate for $\ell^p$-norm of short variations is 
straightforward. Next, to control long variations we adopt the partition of unity constructed in \cite{mst2}, that is
\[
	1 =
	\sum_{s = 0}^{n-1} \Xi_{n, s}^{\beta}
	+ \Big(1-\sum_{s = 0}^{n-1} \Xi_{n, s}^{\beta}\Big),
\]
for some parameter $\beta \in \NN_0$. Each projector $\Xi_{n, s}^{\beta}$ is supported by a finite union of disjoint cubes
centered at rational points belonging to $\mathscr{R}_s^\beta$. In this way, we distinguish the part of the multiplier
where we can identify the asymptotic from the highly oscillating piece. The oscillating part is controlled by a
multi-dimensional version of Weyl--Vinogradov's inequality with a logarithmic loss together with $\ell^p\big(\ZZ^d\big)$
estimates for multipliers of Ionescu--Wainger type. By the triangle inequality, to control the first part it is enough to
show
\begin{equation}
	\label{eq:34}
    \big\|
    V_r\big(\calF^{-1}\big(\frkM_{N_n} \Xi_{n, s}^\beta \hat{f} \big) : n > s \big)
    \big\|_{\ell^p}
    \leq C_p (s+1)^{-2} \|f\|_{\ell^p}.
\end{equation}
First, by the circle method of Hardy and Littlewood, we find the asymptotic of the multiplier $\frkM_{N_n}$.
Here we encounter the main difference from \cite{mst2}. Namely, for $\xi$ sufficiently close to the rational point $a/q$
we have
\begin{equation}
	\label{eq:20}
	\frkM_{N_n}(\xi) = G(a/q) \Phi_{N_n}(\xi - a/q) + \calO\Big(\exp\big(-c\sqrt{\log N_n}\big)\Big),
\end{equation}
provided that $1 \leq q \leq (\log N_n)^{\beta'}$, where $G(a/q)$ is the Gaussian sum and $\Phi_{N_n}$ 
is an integral version of $\frkM_{N_n}$. The limitation on the size of the denominator is a consequence of the fact that
for a larger $q$ the Siegel--Walfisz theorem has an additional term due to the possible exceptional
zero of the exceptional quadratic character. The second issue is the slower decay of the error term in \eqref{eq:20}.
In particular, the later has its impact on the size of the cubes in the partition of unity. Both facts made the
analysis of the approximating multipliers $\nu_{N_n}^s$ harder. To overcome this we directly work with $\frkM_N$.
Moreover, we get completely unified approach to the variational estimates for the averaging operators and the truncated
discrete singular operators.

Going back to the sketch of the proof, in order to show \eqref{eq:34}, we divide the variation into two parts:
$s < n \leq 2^{\kappa_s}$ and $2^{\kappa_s} < n$, where $\kappa_s \simeq (s+1)^{\rho/10}$. For large scales
$2^{\kappa_s} < n$, we transfer \emph{a priori} estimates on $L^p$-norm for $r$-variation of the related continuous
multipliers. Since the Gaussian sums satisfies $\abs{G(a/q)} \lesssim q^{-\delta}$ for some $\delta > 0$, we gain
a decay $(s+1)^{-\delta\beta\rho}$ on $\ell^2$. Consequently, by interpolation the $\ell^p$ norm of $r$-variation for large
scales is bounded by $(s+1)^{-2}$ provided that $\beta$ is sufficiently large. In the case of small scales
$s < n \leq 2^{\kappa_s}$, the estimate on $\ell^2$ is obtained with a help of the numerical inequality \eqref{eq:32}.
We again show that $\ell^2$ norm is bounded by $(s+1)^{-\delta \beta \rho + 1}$. Because of the weaker asymptotic
\eqref{eq:20}, to obtain $\ell^p$ bounds for $r$-variations over small scales required a new approach. We further divide
the index set into dyadic blocks, then on each block we construct a good approximation to the multiplier giving bounds
on $\ell^p$ norm independent of the block. At the cost of additional factor of $\kappa_s^2$, we control $\ell^p$ norm of
$r$-variation. Again, by interpolation combined with a choice of $\beta$ large enough we can make the $\ell^p$ norm
bounded by $(s+1)^{-2}$.

Let us briefly describe the structure of the article. In Section \ref{sec:2.1} we collect basic properties of the
variational seminorm. In Section \ref{sec:3}, we show how to deduce Theorem \ref{main_thm:1} from $r$-variational 
estimates \eqref{eq:141a} and \eqref{eq:142b}. Then we present the lifting procedure, which allows us to replace any
polynomial mapping $\calP$ by a canonical one $\calQ$. In the next section, we describe multipliers of Ionescu--Wainger
type whose $\ell^p$ norm estimates are essential to our argument. In Section \ref{sec:4}, we show a multi-dimensional
version of Weyl--Vinogradov's inequality with a logarithmic loss. Moreover, we prove the estimate on the Gaussian sums
of a mixed type. Sections \ref{sec:8} and \ref{sec:9} are devoted to study the asymptotic behavior of multipliers
$M_N$ and $H_N$, respectively. Finally, to get completely unified approach to the variational estimates for the averaging
operators and truncated singular operators, at the beginning of Section \ref{sec:7}, we list the properties shared by
them which are sufficient to prove Theorem \ref{main_thm:4}. In the next two sections we show the estimates on long and
short variations.

\subsection*{Notation}
Throughout the whole article, we write $A \lesssim B$ ($A \gtrsim B$) if there is an absolute constant $C>0$ such that
$A\le CB$ ($A\ge CB$). Moreover, $C$ stand for a large positive constant
whose value may vary from occurrence to occurrence. If $A \lesssim B$ and $A\gtrsim B$ hold simultaneously then we write
$A \simeq B$. Lastly, we write $A \lesssim_{\delta} B$ ($A \gtrsim_{\delta} B$) to indicate that the constant $C$
depends on some $\delta > 0$. Let $\NN_0 = \NN \cup \{0\}$. For a vector $x \in \RR^d$, we set
$\norm{x}_\infty = \max\{|x_j| : 1 \leq j \leq d\}$. Given a subset $A \subseteq \ZZ$ and $x \in \RR$, we set 
$A_x = A \cap [0, x]$.

\section{Preliminaries}

\subsection{Variational norm}
\label{sec:2.1}
Let $r \in [1, \infty)$. For a sequence $(a_j : j \in A)$, $A \subseteq \ZZ$, we define $r$-variational seminorm by
\[
	V_r(a_j : j \in A) = \sup_{\atop{k_0 < \ldots < k_J}{k_j \in A}}
	\bigg(\sum_{j = 1}^J \abs{a_{k_j} - a_{k_{j-1}}}^r \bigg)^{\frac{1}{r}}.
\]
The function $r \mapsto V_r(a_j : j \in A)$ is non-decreasing, thus
\[
	V_r(a_j : j \in A) \leq V_1(a_j : j \in A),
\]
and by Minkowski's inequality
\[
	V_r(a_j : j \in A) \leq 2 \Big(\sum_{j \in A} \abs{a_j}^r\Big)^{\frac{1}{r}}.
\]
Moreover, for any $j_0 \in A$,
\begin{equation}
	\label{eq:94}
	\sup_{j \in A} \abs{a_j} \leq V_r(a_j : j \in A) + \abs{a_{j_0}},
\end{equation}
Finally, for any increasing sequence $(u_k : 0 \leq k \leq K)$, we have
\begin{equation}
	\label{eq:162}
	V_r(a_j : u_0 \leq j \leq u_K) \leq 
	K^{1-1/r} \Big( \sum_{k = 1}^K V_r(a_j : u_{k-1} \leq j \leq u_k)^r\Big)^\frac{1}{r}.
\end{equation}
The following lemma is essential in studying variational seminorms.
\begin{lemma}{\cite[Lemma 1]{mt3}}
	\label{lem:1}
	If $r \geq 2$ then for any sequence $(a_j : 0 \leq j \leq 2^s)$ of complex numbers
	\begin{equation}
		\label{eq:32}
		V_r(a_j : 0 \leq j \leq 2^s) \leq \sqrt{2} \sum_{i = 0}^s 
		\bigg(\sum_{j = 0}^{2^{s-i}-1} \big|a_{(j+1)2^i} - a_{j2^i}\big|^2\bigg)^{1/2}.
	\end{equation}
\end{lemma}

\subsection{Pointwise ergodic theorems}
\label{sec:3}
In this section we show how to deduce the pointwise ergodic theorem (Theorem \ref{main_thm:1}) from {\emph{a priori}}
$r$-variational estimates for $\scrM_N^\calP$.
\begin{proposition}
\label{prop:3}
	Let $p \in (1, \infty)$ and $r \in (2, \infty)$. Suppose that there is $C > 0$ such that for all $f \in L^p(X, \mu)$,
	\begin{equation}
		\label{eq:139}
		\big\|V_r\big(\scrM_N^\calP f : N \in \NN\big)\big\|_{\ell^p} \leq C \|f\|_{\ell^p}.
	\end{equation}
	Then there is $C > 0$ such that for all $f \in L^p(X, \mu)$,
	\[
		\big\|\sup_{N \in \NN} \big|\scrA_N f \big|\big\|_{\ell^p} \leq C \|f\|_{\ell^p},
	\]
	and the averages $\big(\scrA_N f(x) : N \in \NN\big)$ converges for $\mu$-almost all $x \in X$.
\end{proposition}
\begin{proof}
	Let us fix $N \in \NN$. For each $m \in \{1, \ldots, N\}$ and $s \in \{1, \ldots, k''\}$, we set
	\[
		S^{(s)}_{N, m} f(x) = 
		\sum_{n \in \NN^{k'}} 
		\sum_{\atop{p \in \PP^{k''}}{p_s \leq m}}
		f\Big(T_1^{\calP_1(n, p)} \cdots T_{d_0}^{\calP_{d_0}(n, p)} x \Big)
		\ind{B_N}(n, p)
		\bigg(\prod_{j = s+1}^{k''} \log p_j\bigg),
	\]
	and $S^{(0)}_{N, N} f = \vartheta_B(N) \scrM_N^\calP f$. For $0 \leq s < k''$, by the partial summation we obtain
	\begin{align*}
		S^{(s)}_{N, N} f& = 
		\sum_{m = 2}^{N}
		\left(S^{(s+1)}_{N, m}f - S^{(s+1)}_{N, m-1}f \right) \log m \\
		&=
		(\log N) S^{(s+1)}_{N, N}f
		+
		\sum_{m = 2}^{N-1}
		\big(\log m - \log (m+1)\big) S^{(s+1)}_{N, m} f.
	\end{align*}
	Hence,
	\begin{align}
		\nonumber
		\big\|
		S^{(s)}_{N, N} f - (\log N) S^{(s+1)}_{N, N} f
		\big\|_{L^p}
		&\leq
		\sum_{m = 2}^{N-1}
		\big\|S^{(s+1)}_{N, m} f \big\|_{L^p} m^{-1} \\
		\nonumber
		&\lesssim
		N^{k-1} (\log N)^{-s}
		\sum_{m = 2}^{N-1}
		\|f\|_{L^p} (\log m)^{-1}\\
		\label{eq:134}
		&\lesssim
		N^k (\log N)^{-s-1} \|f\|_{L^p},
	\end{align}
	where we have used the trivial estimate
	\[
		\big\|S^{(s+1)}_{N, m} f \big\|_{L^p} \lesssim N^{k-1} (\log N)^{-s} m (\log m)^{-1} \|f\|_{L^p},
	\]
	which is a consequence of \eqref{eq:151} and the prime number theorem. Observe that
	\[
		S^{(k'')}_{N, N} f= \pi_B(N) \scrA_N^\calP f,
	\]
	thus by repeated application of \eqref{eq:134}, we arrive at the conclusion that 
	\begin{equation}
		\label{eq:143}
		\Big\|
		\vartheta_B(N) \scrM_N^\calP f
		-
		(\log N)^{k''}
		\pi_B(N) \scrA_N^\calP f
		\Big\|_{L^p}
		\lesssim
		\vartheta_B(N) (\log N)^{-1} \|f\|_{L^p},
	\end{equation}
	because the prime number theorem implies that $\vartheta_B(N) \simeq N^k$. In particular,
	by taking $f = \ind{X}$ and $p = \infty$ in \eqref{eq:143} we get
	\[
		\pi_B(N) = \vartheta_B(N) (\log N)^{-k''} \Big(1 + \calO\big((\log N)^{-1}\big)\Big).
	\]
	Hence, for any $p \in [1, \infty]$ and $f \in L^p(X, \mu)$, 
	\begin{equation}
		\label{eq:138}
		\big\|\scrM_N^\calP f - \scrA_N^\calP f \big\|_{L^p} \lesssim (\log N)^{-1} \|f\|_{L^p}.
	\end{equation}
	Next, if $p > 1$ then we can write
	\begin{align*}
		\big\|
		\sup_{n \in \NN} \big|\scrA_{2^n}^\calP f \big|
		\big\|_{L^p}
		&\leq
		\big\|
		\sup_{n \in \NN} \big|\scrM_{2^n}^\calP f \big|
		\big\|_{L^p}
		+
		\big\|
		\sup_{n \in \NN} \big|\scrM_{2^n}^\calP f - \scrA_{2^n}^\calP f \big|
		\big\|_{L^p} \\
		&\lesssim
		\big\|
        \sup_{n \in \NN} \big|\scrM_{2^n}^\calP f \big|
        \big\|_{L^p}
        +
		\Big(
		\sum_{n \in \NN}
		n^{-p}
		\Big)^{1/p}
		\|f\|_{L^p}.
	\end{align*}
	In view of \eqref{eq:94}, {\emph{a priori}} estimate \eqref{eq:139} entails that
	\[
		\big\|
        \sup_{N \in \NN} \big|\scrA_N^\calP f \big|
        \big\|_{L^p}
		\lesssim
		\|f\|_{L^p}.
	\]
	Hence, while proving $\mu$-almost everywhere convergence of the averages $\big(\scrA_N f : N \in \NN\big)$
	for $f \in L^p(X, \mu)$, we may assume that the function $f$ is bounded. By \eqref{eq:138}, for $p = \infty$,
	we can write
	\[
		\big|
		\scrM_N^\calP f(x) - \scrA_N^\calP f(x) 
		\big|
		\leq
		\big\|
        \scrM_N^\calP f - \scrA_N^\calP f
        \big\|_{L^\infty}
		\lesssim
		(\log N)^{-1}
		\|f\|_{L^\infty}.
	\]
	Therefore, the convergence of $\big(\scrM_N^\calP f(x) : N \in \NN\big)$ implies the convergence of
	$\big(\scrA_N^\calP f(x) : N \in \NN\big)$ to the same limit.
\end{proof}
Thanks to the Calder\'on's transference principle we can restrict attention to the model dynamical system, that is,
$\ZZ^{d_0}$ with the counting measure and the shift operator. Hence, it suffices to study the operators
\eqref{eq:131} and \eqref{eq:132} on $\ell^p\big(\ZZ^{d_0}\big)$.

\subsection{Lifting lemma}
\label{sec:1}
For the polynomial mapping $\calP = \big(\calP_1, \ldots, \calP_{d_0}\big)$, let us define
\[
	\deg \calP = \max\big\{\deg \calP_j : 1 \leq j \leq d_0\big\}.
\]
It is convenient to work with the set
\[
	\Gamma = \big\{\gamma \in \ZZ^k \setminus \{0\} : 0 \leq \gamma_j \leq \deg \calP, \text{ for each } 
	j = 1, \ldots, k\big\}
\]
equipped with the lexicographic order. Then each $\calP_j$ can be expressed as
\[
	\calP_j(x) = \sum_{\gamma \in \Gamma} c_{j, \gamma} x^\gamma,
\]
for some $c_{j, \gamma} \in \ZZ$. The cardinality of the set $\Gamma$ is denoted by $d$. We identify $\RR^d$ with
$\RR^\Gamma$. Let $A$ be a diagonal $d \times d$ matrix such that for all $\gamma \in \Gamma$ and $v \in \RR^d$,
\begin{equation}
	\label{eq:21}
	(A v)_\gamma = \abs{\gamma} v_\gamma.
\end{equation}
For $t > 0$, we set
\[
	t^A v = \big(t^{\abs{\gamma}} v_\gamma : \gamma \in \Gamma\big).
\]
Finally, we introduce the \emph{canonical} polynomial mapping,
\[
	\calQ = (\calQ_\gamma : \gamma \in \Gamma): \RR^k \rightarrow \RR^d,
\]
by setting $\calQ_\gamma(x) = x^\gamma$. Now, if we define $L: \RR^d \rightarrow \RR^{d_0}$ to be the linear 
transformation such that for $v \in \RR^d$,
\[
	(L v)_j = \sum_{\gamma \in \Gamma} c_{j, \gamma} v_\gamma,
\]
then $L \calQ = \calP$. The following lemma allows us to reduce the problems to studying the canonical polynomial
mappings. 
\begin{lemma}{\cite[Lemma 2.1]{mst1}}
	Let $R^\calP_N$ be any of the operators $M_N^\calP$ or $H_N^\calP$. Suppose that for some $p \in (1, \infty)$
	and $r \in (2, \infty)$,
	\[
		\big\|V_r\big(R_N^\calQ f : N \in \NN\big)\big\|_{\ell^p(\ZZ^d)} 
		\leq C_{p, r} \|f\|_{\ell^p(\ZZ^d)},
	\]
	then 
	\[
		\big\|V_r\big(R_N^\calP f : N \in \NN\big)\big\|_{\ell^p(\ZZ^{d_0})} 
		\leq C_{p, r} \|f\|_{\ell^p(\ZZ^{d_0})}.
	\]
\end{lemma}
In the rest of the article by $M_N$ and $H_N$ we denote the averaging and the truncated discrete singular operator for the
canonical polynomial mapping $\calQ$, that is $M_N = M_N^\calQ$ and $H_N = H_N^\calQ$.

\subsection{Ionescu--Wainger type multipliers}
\label{sec:2}
Let $\calF$ denote the Fourier transform on $\RR^d$, that is for any $f \in L^1\big(\RR^d\big)$,
\[
	\calF f(\xi) = \int_{\RR^d} f(x) e^{2\pi i \sprod{\xi}{x}} {\: \rm d}x.
\]
If $f \in \ell^1\big(\ZZ^d\big)$, then we set
\[
	\hat{f}(\xi) = \sum_{x \in \ZZ^d} f(x) e^{2 \pi i \sprod{\xi}{x}}.
\]
To simplify the notation, by $\calF^{-1}$ we denote the inverse Fourier transform on $\RR^d$ as well as the inverse
Fourier transform on the $d$-dimensional torus identified with $(0, 1]^d$. We also fix $\eta : \RR^d \rightarrow \RR$, 
a smooth function such that $0 \leq \eta \leq 1$, and
\[
	\eta(x) = 
	\begin{cases}
		1 & \text{if } \norm{x}_\infty \leq \tfrac{1}{32 d}, \\
		0 & \text{if } \norm{x}_\infty \geq \tfrac{1}{16 d}. 
	\end{cases}
\]
We additionally assume that $\eta$ is a convolution of two non-negative smooth functions with
supports contained inside $\left[-\tfrac{1}{8d}, \tfrac{1}{8d}\right]^d$.

Next, let us recall necessary notation to define auxiliary multipliers of Ionescu--Wainger type. For details we refer to
\cite{mst1}. The following construction depends on a parameter $\beta \in \NN$.

For $n \in \NN$, we set $n_0 = \lfloor n^{1/20}\rfloor$ and $Q_0 = (n_0!)^D$ where
$D = 20 \beta + 1$. We define
\[
	\Pi = \bigcup_{k = 1}^D \Pi_k,
\]
wherein for $k \in \{1, \ldots, D\}$ we have set
\[
	\Pi_k = \Big\{p_1^{\gamma_1} \cdots p_k^{\gamma_k} : \gamma_j \in \NN_D \text{ and } 
	p_j \in \PP \cap \big(n_0, n^\beta\big]
	\text{ are distint for all } 1 \leq j \leq k \Big\}.
\]
Let
\[
	P_n = \big\{Q \cdot w : Q \mid Q_0 \text{ and } w \in \Pi \cup \{1\} \big\}.
\]
Notice that $\NN_{n^\beta} \subseteq P_n \subseteq \NN_{e^{n^{1/10}}}$. For $q \in \NN$, let us define
\[
	A_q = \big\{a \in \NN_q : (a, q) = 1 \big\},
\]
and
\[
	\bfA_q = \Big\{a \in \NN_q^k : \gcd\big(q, a_1, \ldots, a_k\big) = 1\Big\}.
\]
Lastly, we set
\begin{equation}
	\label{eq:91}
	\scrU_n^\beta = \Big\{a/q : a \in \bfA_q \text{ and } q \in P_n \Big\}.
\end{equation}
Given $(\Theta_j : j \in \ZZ)$ a sequence of multipliers on $\RR^d$ such that for each $r \in (1, \infty)$ there is 
$A_r > 0$ such that for all $f \in L^2\big(\RR^d\big) \cap L^r\big(\RR^d\big)$,
\[
	\Big\|
	\Big(
	\sum_{j \in \ZZ} \big|\calF^{-1}\big(\Theta_j \: \calF f \big)\big|^2\Big)^{1/2} \Big\|_{L^r} 
	\leq A_r 
	\|f\|_{L^r},
\]
its discrete counterpart is given by the formula
\[
	\Theta_j^{\beta}(\xi) = \sum_{a/q \in \mathscr{U}_n^\beta} 
	\eta\big(\mathcal{E}_n^{-1} (\xi - a/q)\big) \Theta_j(\xi - a/q),
\]
where $\mathcal{E}_n$ being a diagonal $d \times d$ matrix with positive entries
$(\epsilon_{n, \gamma} : \gamma \in \Gamma)$ such that $\epsilon_{n, \gamma} \leq \exp\big(-n^{1/5}\big)$. Then by 
\cite[Theorem 2.1]{msz-k}, for each $p \in (1, \infty)$ and any finitely supported function $f: \ZZ^d \rightarrow \CC$,
\begin{equation}
	\label{eq:10}
	\Big\| 
	\Big(\sum_{j \in \ZZ} \big| \calF^{-1}\big(\Theta_j^{\beta} \hat{f}\big) \big|^2\Big)^{1/2} 
	\Big\|_{\ell^p} 
	\lesssim_{\beta, p} 
	\log (n+1) A_{2r}
	\|f \|_{\ell^p},
\end{equation}
where $r = \max\big\{\lceil p/2 \rceil, \lceil p'/2 \rceil \big\}$. The scalar-valued version of \eqref{eq:10}
was proved in \cite{iw}, see also \cite{mst1}. The vector-valued extension was recently observed in \cite{msz-k}.
Essentially its proof follows the same line as scalar-valued except that in place of Marcinkiewicz--Zygmund inequality
one uses Kahane's vector-valued extension of Khinchine's inequality, see \cite[Theorem 2.1]{msz-k} for details.

\section{Trigonometric sums}
\label{sec:4}
\subsection{Weyl--Vinogradov sum}
We say that a subset of integers $\calA$ is \emph{polynomially regular}, if for all $\alpha, \alpha_1 > 0$, there are
$\beta_0 > 0$ and a constant $C > 0$ so that for any integer
$1 \leq Q \leq (\log N)^{\alpha_1}$, $\beta > \beta_0$ and any polynomial $P$ of a form
\[
	P(x) = \frac{a}{q} x^d + \ldots + \xi_1 x,
\]
for some coprime integers $a$ and $q$, such that $1 \leq a \leq q$, and
\[
	(\log N)^\beta \leq q \leq N^d (\log N)^{-\beta},
\]
we have
\begin{equation}
	\label{eq:8}
	\bigg|
	\sum_{\atop{n \in \calA}{n \equiv r \bmod Q}}
	 e^{2\pi i P(n)} \ind{[-N, N]}(n) 
	\bigg| \leq C Q^{-1} N (\log N)^{-\alpha},
\end{equation}
for all $r \in\{1, \ldots, Q\}$ and $N \in \NN$. 

Let us check that $\ZZ$ is polynomially regular. We write
\begin{equation}
	\label{eq:15}
	\sum_{\atop{n \in \ZZ}{n \equiv r \bmod Q}} 
	e^{2\pi i P(n)} \ind{[1, N]}(n)
	= 
	\sum_{m = 1}^{\lfloor N/Q \rfloor} e^{2\pi i \tilde{P}(m)} + \calO(Q),
\end{equation}
where
\[
	\tilde{P}(m) = P(Q m + r) = \frac{a}{q} Q^d m^d + \text{lower powers of m}.
\]
Set $M = \lfloor N/Q \rfloor$ and $a'/q' = Q^d a/q$ with $(a', q') = 1$. Then
\[
	(\log M)^{\beta-d \alpha_1} \leq q Q^{-d} \leq q' \leq q \leq M^d (\log M)^{-\beta+d \alpha_1},
\]
and hence, by Weyl estimates with logarithmic loss (see e.g. \cite[Remark after Theorem 1.5]{w0}),
\begin{align*}
	\Big|\sum_{m = 1}^M e^{2\pi i \tilde{P}(m)} \Big| 
	&\leq C M (\log M)
	\bigg(\frac{1}{q'} + \frac{1}{M} + \frac{q'}{M^d}\bigg)^{\frac{1}{2 d^2 - 2d +1}}.
\end{align*}
Therefore, for $\beta > \beta_0 = (1+\alpha)(2 d^2 - 2d +1) +d \alpha_1$, by \eqref{eq:15}, we conclude that
\[
	\Big|\sum_{\atop{n \in \ZZ}{n \equiv r \bmod Q}}
    e^{2\pi i P(n)} \ind{[1, N]}(n)
	\Big|
	\leq
	C Q^{-1} N (\log N)^{-\alpha},
\]
proving the claim. Another example of polynomially regular sets is the set of prime numbers. This is a consequence of 
\cite[Theorem 10]{hua}.

Our aim is to understand exponential sums over Cartesian products of polynomially regular sets. Let us fix
a function $\phi: \RR^k \rightarrow \CC$ satisfying
\begin{equation}
	\label{eq:7}
	\abs{\phi(x)} \leq C, \qquad \abs{\nabla \phi(x)} \leq (1+\norm{x})^{-1}.
\end{equation}
The main result of this section is the following theorem.
\begin{theorem}
	\label{thm:1}
	Let $\calA_1, \ldots \calA_k$ be polynomially regular subsets of $\ZZ$. For all $\alpha > 0$ there are
	$\beta_0 > 0$ and a constant $C > 0$ so that for all $\beta > \beta_0$ and any polynomial $P$ of a form
	\[
		P(x) = \sum_{0 < \abs{\gamma} \leq d} \xi_\gamma x^\gamma,
	\]
	wherein for some $0 < \abs{\gamma_0} \leq d$, 
	\[
		\bigg|\xi_{\gamma_0} - \frac{a}{q}\bigg| \leq \frac{1}{q^2},
	\]
	for some coprime integers $a$ and $q$ such that $1 \leq a \leq q$, and
    \[
        (\log N)^\beta \leq q \leq N^{\abs{\gamma_0}} (\log N)^{-\beta},
    \]
	we have
	\[
		\sup_{\atop{\Omega \subseteq [-N, N]^k}{\Omega \text{ convex}}}
		\left|\sum_{n \in \calA_1 \times \ldots \times \calA_k} e^{2\pi i P(n)} \ind{\Omega}(n) \phi(n)\right|
		\leq
		C N^k (\log N)^{-\alpha}.
	\]
	The constant $C$ depends on $\alpha$, $d$ and a constant in \eqref{eq:7}.
\end{theorem}
\begin{proof}
	Let us first assume that $\phi \equiv 1$. The proof consists of three steps. 

	\noindent
	{\bf Step 1.} We consider the case when $k = 1$ and $\abs{\gamma_0} = d$. Take $\alpha > 0$ and $\alpha_1 > 0$,
	and let $\beta > \beta_0 = 3 \beta_1 + 3 d \alpha$, where $\beta_1$ is the value of $\beta_0$ determined by
	$\calA_1$ for $\alpha$ and $\alpha_1$. Suppose that $a$ and $q$ are coprime integers such that
	$1 \leq a \leq q$, and
	\[
		\bigg|\xi_d - \frac{a}{q}\bigg| \leq \frac{1}{q^2},
	\]
	with $(\log N)^\beta \leq q \leq N^d (\log N)^{-\beta}$. By Dirichlet's principle, there are coprime integers
	$a'$ and $q'$ such that $1 \leq a' \leq q' \leq N^d (\log N)^{-\frac{1}{3}\beta}$, and
	\[
		\bigg|\xi_d - \frac{a'}{q'}\bigg| \leq \frac{1}{q'} N^{-d} (\log N)^{\frac{1}{3}\beta}.
	\] 
	If $a'/q' \neq a/q$ then
	\[
		\frac{1}{q q'} \leq \bigg|\frac{a}{q} - \frac{a'}{q'} \bigg| \leq
		\bigg|\xi_d - \frac{a}{q}\bigg| + \bigg|\xi_d - \frac{a'}{q'}\bigg|
		\leq
		\frac{1}{q^2} + N^{-d} (\log N)^{\frac{1}{3}\beta}.
	\]
	Hence, we obtain
	\[
		\frac{1}{q'} \leq \frac{1}{q} + q N^{-d} (\log N)^{\frac{1}{3}\beta} 
		\leq (\log N)^{-\beta} + (\log N)^{-\frac{2}{3}\beta}.
	\]
	Thus
	\[
		(\log N)^{\frac{1}{3}\beta} \leq q' \leq N^d (\log N)^{-\frac{1}{3}\beta}.
	\]
	Observe that the last estimate is also valid if $q' = q$. Let $Q$ be an integer such that
	$1 \leq Q \leq (\log N)^{\alpha_1}$. Given $r \in \{1, \ldots, Q\}$, we set
	\[
		S_{r, N} = \sum_{\atop{n \in \calA_1}{n \equiv r \bmod Q}} e^{2\pi i P(n)} \ind{[1, N]}(n),
		\qquad\text{and}\qquad
		\tilde{S}_{r, N} = \sum_{\atop{n \in \calA_1}{n \equiv r \bmod Q}} e^{2\pi i \tilde{P}(n)} \ind{[1, N]}(n),
	\]
	where
	\[
		\tilde{P}(x) = \frac{a'}{q'} x^d + \ldots + \xi_1 x.
	\]
	We first show that
	\begin{equation}
		\label{eq:9}
		\sup_{1 \leq N' \leq N} \abs{\tilde{S}_{r, N'}} \leq C Q^{-1} N (\log N)^{-\alpha},
	\end{equation}
	for all $\beta > \beta_0$. If $1 \leq N' \leq N (\log N)^{-\alpha}$, then there is nothing to be proven. For
	$N (\log N)^{-\alpha} \leq N' \leq N$, we have
	\[
		(\log N')^{\frac{1}{3} \beta} \leq q' 
		\leq N^d (\log N)^{-\frac{1}{3} \beta} \leq (N')^d (\log N)^{d\alpha - \frac{1}{3} \beta}
		\leq
		(N')^d (\log N')^{d\alpha-\frac{1}{3} \beta},
	\]
	thus, by \eqref{eq:8}, we obtain 
	\[
		\abs{\tilde{S}_{r, N'}} \leq C Q^{-1} N' (\log N')^{-\alpha} \leq C' Q^{-1} N (\log N)^{-\alpha},
	\]
	proving \eqref{eq:9}. We now set $\theta = \xi_d - a'/q'$ and apply the partial summation to get
	\begin{align*}
		\abs{S_{r, N'}} &= \Big|\sum_{n = 1}^{N'} (\tilde{S}_{r, n} - \tilde{S}_{r, n-1}) 
		e^{2\pi i \theta n^d} \Big|
		\leq
		\abs{\tilde{S}_{r, N'}} + \abs{\tilde{S}_{r, 0}} +
		\sum_{n=1}^{N'-1} \abs{\tilde{S}_{r, n}} 
		\Big|e^{2\pi i \theta n^d} - e^{2 \pi i \theta (n+1)^d}\Big|.
	\end{align*}
	Since
	\[
		\abs{\theta} \leq \frac{1}{q'} N^{-d} (\log N)^{\frac{1}{3}\beta} \lesssim
		N^{-d},
	\]
	by \eqref{eq:9}, we obtain
	\begin{align*}
		\sup_{1 \leq N' \leq N} \abs{S_{r, N'}} &\lesssim
		Q^{-1} N (\log N)^{-\alpha} \sum_{n = 1}^{N-1} \abs{\theta} n^{d-1}\\
		&\lesssim
		Q^{-1} N(\log N)^{-\alpha},
	\end{align*}
	which finishes the proof of Step 1.

	\noindent
	{\bf Step 2.} We next consider $k \geq 2$ and $\gamma_0 \neq (0, \ldots, 0 , \ell, 0, \ldots, 0)$ for any
	$\ell \leq d$. Without loss of generality we may assume that $\gamma_0(1) \geq 1$. By the triangle inequality 
	followed by Cauchy--Schwarz inequality we get
	\begin{align}
		\nonumber
		\Big| \sum_{n \in \calA_1 \times \ldots \times \calA_k} e^{2 \pi i P(n)} \ind{\Omega}(n)\Big|
        &\leq
		\sum_{\tilde{n} \in \calA_2 \times \ldots \times \calA_k}
        \Big|\sum_{n_1 \in \calA_1} e^{2 \pi i P(n_1; \tilde{n})} \ind{\Omega}(n_1, \tilde{n})
        \Big| \\
		\label{eq:11}
		&\lesssim
		N^{(k-1)/2} 
		\bigg(
		\sum_{\tilde{n} \in \calA_2 \times \ldots \times \calA_k}
		\Big|\sum_{n_1 \in \calA_1} e^{2 \pi i P(n_1; \tilde{n})} \ind{\Omega}(n_1, \tilde{n})
        \Big|^2	
		\bigg)^{\frac{1}{2}}.
	\end{align}
	Next, we have
	\begin{align}
		\nonumber
		&\sum_{\tilde{n} \in \calA_2 \times \ldots \times \calA_k}
        \Big|\sum_{n_1 \in \calA_1} e^{2 \pi i P(n_1; \tilde{n})} \ind{\Omega}(n_1, \tilde{n})
        \Big|^2
		\leq
		\sum_{\tilde{n} \in \ZZ^{k-1}}
        \Big|\sum_{n_1 \in \calA_1} e^{2 \pi i P(n_1; \tilde{n})} \ind{\Omega}(n_1, \tilde{n})
        \Big|^2 \\
		\nonumber
		&\qquad\qquad\leq
		\sum_{n_1, n_1' \in \calA_1}
		\Big|
		\sum_{\tilde{n} \in \ZZ^{k-1}}
		e^{2 \pi i (P(n_1; \tilde{n}) - P(n_1'; \tilde{n}))} \ind{\Omega}(n_1, \tilde{n})
		\ind{\Omega}(n_1', \tilde{n})
		\Big| \\
		\label{eq:12}
		&\qquad\qquad
		\leq
		\sum_{n_1, n_1' \in \ZZ}
		\Big|
        \sum_{\tilde{n} \in \ZZ^{k-1}}
        e^{2 \pi i (P(n_1; \tilde{n}) - P(n_1'; \tilde{n}))} \ind{\Omega}(n_1, \tilde{n})
        \ind{\Omega}(n_1', \tilde{n})
        \Big|,
	\end{align}
	which, by another application of Cauchy--Schwarz inequality, is bounded by
	\begin{align*}
		N
		\bigg(
		\sum_{n_1, n_1' \in \ZZ}
        \Big|
        \sum_{\tilde{n} \in \ZZ^{k-1}}
        e^{2 \pi i (P(n_1; \tilde{n}) - P(n_1'; \tilde{n}))} \ind{\Omega}(n_1, \tilde{n})
        \ind{\Omega}(n_1', \tilde{n}) 
        \Big|^2
		\bigg)^{\frac{1}{2}}.
	\end{align*}
	Finally,
	\begin{equation}
		\label{eq:144}
		\begin{aligned}
		&
		\sum_{n_1, n_1' \in \ZZ}
        \Big|
        \sum_{\tilde{n} \in \ZZ^{k-1}}
        e^{2 \pi i (P(n_1; \tilde{n}) - P(n_1'; \tilde{n}))} 
		\ind{\Omega}(n_1, \tilde{n})
        \ind{\Omega}(n_1', \tilde{n}) 
        \Big|^2 \\
		&\qquad\qquad\qquad=
		\sum_{n_1, n_1' \in \ZZ} \sum_{\tilde{n}, \tilde{n}' \in \ZZ^{k-1}}
		e^{2\pi i Q(n_1, \tilde{n}, n_1', \tilde{n}')}
		\ind{\Theta}(n_1, \tilde{n}, n_1', \tilde{n}'),
		\end{aligned}
	\end{equation}
	where
	\[
		\Theta = \Big\{(x_1, \tilde{x}, x_1', \tilde{x}') \in \Omega \times \Omega
		: (x_1, \tilde{x}'), (x_1', \tilde{x}) \in \Omega \Big\},
	\]
	and
	\[
		Q(x_1, \tilde{x}, x_1', \tilde{x}') = P(x_1; \tilde{x}) - P(x_1'; \tilde{x}) 
		- P(x_1; \tilde{x}') 
		+ P(x_1'; \tilde{x}').
	\]
	Notice that the set $\Theta$ is a convex subset of a cube $[-N, N]^{2k}$. 
	Moreover, the polynomial $Q(x, x')$ has degree at least $\abs{\gamma_0}$ having a coefficient $\xi_{\gamma_0}$ in
	front of the monomial $x^{\gamma_0}$. Therefore, by \cite[Theorem 3.1]{mst1}, there are $\beta_0 > 0$ and $C > 0$
	such that
	\[
		\Big|
		\sum_{n_1, n_1' \in \ZZ} \sum_{\tilde{n}, \tilde{n}' \in \ZZ^{k-1}}
        e^{2\pi i Q(n_1, \tilde{n}, n_1', \tilde{n}')}
        \ind{\Theta}(n_1, \tilde{n}, n_1', \tilde{n}')
		\Big|
		\leq
		C N^{2 k} (\log N)^{-4 \alpha},
	\]
	provided that $\beta > \beta_0$.
	Hence, by \eqref{eq:11}, \eqref{eq:12} and \eqref{eq:144} we obtain
	\[
		\Big| \sum_{n \in \calA_1 \times \ldots \times \calA_k} e^{2 \pi i P(n)} \ind{\Omega}(n)\Big|
        \lesssim
		N^k (\log N)^{-\alpha}.
	\]

	\noindent
	{\bf Step 3.} Suppose that $k \geq 1$ and $\gamma_0 = (0, \ldots, 0, \ell, 0, \ldots, 0)$ for $1 \leq \ell \leq d$. 
	Without loss of generality we may assume that $\gamma_0 = (\ell, \ldots, 0)$. 
	The proof is by a backward induction over $\ell \in \{1, \ldots, d\}$. We write
	\begin{equation}
		\label{eq:19}
		\Big| \sum_{n \in \calA_1 \times \ldots \times \calA_k} e^{2 \pi i P(n)} \ind{\Omega}(n) \Big|
		\leq
		\sum_{\tilde{n} \in \calA_2 \times \ldots \times \calA_k} 
		\Big|\sum_{n_1 \in \calA_1} e^{2 \pi i P(n_1; \tilde{n})} \ind{\Omega}(n_1, \tilde{n}) 
		\Big|.
	\end{equation}
	If $\ell = d$ the conclusion follows by Step 1. Suppose that $\ell < d$. In view of Step 2 and the inductive
	hypothesis, the estimate holds for any $\abs{\gamma_0} = j$, $\ell < j \leq d$. Let $\beta_1$ be the largest 
	value of $\beta_0$ among those that were determined in Step 2 and resulting from the inductive hypothesis.
	By Dirichlet's principle, for each $\ell < \abs{\gamma} \leq d$, we select coprime integers $a_\gamma$ and $q_\gamma$,
	such that $1 \leq a_\gamma \leq q_\gamma \leq N^{\abs{\gamma}} (\log N)^{-\beta_1}$, satisfying
	\[
		\bigg|\xi_\gamma - \frac{a_\gamma}{q_\gamma} \bigg| \leq 
		\frac{1}{q_\gamma} N^{-\abs{\gamma}} (\log N)^{\beta_1}.
	\]
	If for some $\gamma \in \Gamma$, $\ell < \abs{\gamma} \leq d$ we have $(\log N)^{\beta_1} \leq q_\gamma$, then
	the conclusion follows by the inductive hypothesis or Step 2. Otherwise, we set
	$\theta_\gamma = \xi_\gamma - a_\gamma/q_\gamma$ and $Q = \lcm\{q_\gamma : \ell < \abs{\gamma} \leq d \}$. We have
	\begin{equation}
		\label{eq:17}
		\abs{\theta_\gamma} \leq \frac{1}{q_\gamma} N^{-\abs{\gamma}} (\log N)^{\beta_1},
	\end{equation}
	and
	\[
		Q \leq (\log N)^{\alpha_1},
	\]
	where $\alpha_1 = \beta_1 \cdot \#\{\gamma \in \NN_0^k : \ell < \abs{\gamma} \leq d \}$. We have
	\begin{equation}
		\label{eq:22}
		\begin{aligned}
		\sum_{\tilde{n} \in \calA_2 \times \ldots \times \calA_k}
        \Big|\sum_{n_1 \in \calA_1} e^{2 \pi i P(n_1; \tilde{n})} \ind{\Omega}(n_1, \tilde{n})
        \Big|
		&\leq 
		\sum_{\tilde{n} \in \ZZ^{k-1}} \Big|\sum_{n_1 \in \calA_1} e^{2\pi i P(n_1; \tilde{n})}
		\ind{\Omega}(n_1, \tilde{n}) 
		\Big|\\
		&\leq
		\sum_{(r_1, \tilde{r}) \in \NN_Q^k} \sum_{\atop{\tilde{n} \in \ZZ^{k-1}}{\tilde{n} \equiv \tilde{r} \bmod Q}}
		\Big|\sum_{\atop{n_1 \in \calA_1}{n_1 \equiv r_1 \bmod Q}} e^{2\pi i P(n_1; \tilde{n})} 
		\ind{\Omega}(n_1, \tilde{n}) \Big|.
		\end{aligned}
	\end{equation}
	Setting
	\[
		P_0(x) = \sum_{0 < \abs{\gamma} \leq \ell} \xi_\gamma x^\gamma,
	\]
	we can write
	\begin{align*}
		P(Q m + r) 
		&\equiv
		\sum_{\ell < \abs{\gamma} \leq d} \xi_\gamma (Q m + r)^\gamma + P_0(Q m + r) \pmod 1 \\
		&\equiv 
		\sum_{\ell < \abs{\gamma} \leq d} \frac{a_\gamma}{q_\gamma} r^\gamma + 
		\sum_{\ell < \abs{\gamma} \leq d} \theta_\gamma (Q m + r)^\gamma + P_0(Q m + r) \pmod 1.
	\end{align*}
	Thus
	\begin{equation}
		\label{eq:18}
		\begin{aligned}
		&\sum_{(r_1, \tilde{r}) \in \NN_Q^k} \sum_{\atop{\tilde{n} \in \ZZ^{k-1}}{\tilde{n} \equiv \tilde{r} \bmod Q}}
        \Big|\sum_{\atop{n_1 \in \calA_1}{n_1 \equiv r_1 \bmod Q}} e^{2\pi i P(n_1; \tilde{n})}
        \ind{\Omega}(n_1, \tilde{n}) \Big| \\
		&\qquad\qquad
		=
		\sum_{(r_1, \tilde{r}) \in \NN_Q^k} \sum_{\atop{\tilde{n} \in \ZZ^{k-1}}{\tilde{n} \equiv \tilde{r} \bmod Q}}
        \Big|
		\sum_{n_1 \in \ZZ}
		A_{n_1, \tilde{n}} \big(S_{n_1, \tilde{n}}^{(r)} - S_{n_1-1, \tilde{n}}^{(r)}\big)
		\Big|,
		\end{aligned}
	\end{equation}
	where
	\[
		A_{n_1, \tilde{n}} = e^{2 \pi i \sum_{\ell < \abs{\gamma} \leq d} \theta_\gamma (n_1, \tilde{n})^\gamma},
	\]
	and
	\[
		S_{n_1, \tilde{n}} ^{(r)}=
		\sum_{\atop{n_1' \leq n_1}{n_1' \equiv r_1 \bmod Q}} e^{2\pi i P_0(n_1'; \tilde{n})}
		\ind{\Omega}(n_1', \tilde{n}) \ind{\calA_1}(n_1').
	\]
	To estimate the inner sum on the right-hand side of \eqref{eq:18}, we apply the partial summation. Setting
	\[
		(M_0, M_0+1, \ldots, M_1) = \big\{n_1 \in \ZZ : (n_1, \tilde{n}) \in \Omega\big\},
	\]
	we can write
	\[
		\Big|\sum_{n_1=M_0}^{M_1} A_{n_1, \tilde{n}}  \big(S_{n_1, \tilde{n}}^{(r)} - S_{n_1-1, \tilde{n}}^{(r)} \big)\Big|
		\leq
		\abs{S_{M_1, \tilde{n}}^{(r)}} + \sum_{n_1=M_0}^{M_1-1}
		\abs{S_{n_1, \tilde{n}}^{(r)}} \cdot \abs{A_{n_1, \tilde{n}} - A_{n_1+1, \tilde{n}}}.
	\]
	By \eqref{eq:17}, for $(n_1, \tilde{n}) \in \Omega$ we have
	\[
		\big|A_{n_1, \tilde{n}} - A_{n_1+1, \tilde{n}} \big| \lesssim
		\sum_{\ell < \abs{\gamma} \leq d} \abs{\theta_\gamma} N^{\abs{\gamma}-1}
		\lesssim N^{-1}(\log N)^{\beta_1}.
	\]
	Recall that $\gamma_0 = (\ell, 0, \ldots, 0)$ and
	\[
		\bigg|\xi_{\gamma_0} - \frac{a}{q}\bigg| \leq \frac{1}{q^2},
	\]
	thus, by Step 1 applied to $S_{n_1, \tilde{n}}^{(r)}$ we obtain
	\[
		\sup_{M_0 \leq n_1 \leq M_1} 
		\abs{S_{n_1, \tilde{n}}^{(r)}} \lesssim N Q^{-1} (\log N)^{-\alpha-\beta_1},
	\]
	whenever $\beta > \beta_2$, where $\beta_2$ is the value of $\beta_0$ determined in Step 1 for
    $\alpha+\beta_1$ and $\alpha_1$. Hence, 
	\[
		\Big|\sum_{n_1=M_0}^{M_1} A_{n_1, \tilde{n}}  \big(S_{n_1, \tilde{n}}^{(r)} - S_{n_1-1, \tilde{n}}^{(r)} \big)\Big|
        \lesssim
		Q^{-1} N(\log N)^{-\alpha}.
	\]
	Consequently, by \eqref{eq:19}, \eqref{eq:22} and \eqref{eq:18} we get
	\begin{align*}
		\Big|
		\sum_{n \in \calA_1 \times \ldots \times \calA_k} e^{2\pi i P(n)} \ind{\Omega}(n)
		\Big|
		&\lesssim
		N^k (\log N)^{-\alpha},
	\end{align*}
	provided that $\beta > \beta_0 = \max\{\beta_1, \beta_2\}$.
	
	Finally, we deal with a general $\phi$. Given $\alpha$, let $\beta_0$ be such that
	\begin{equation}
		\label{eq:30}
		\sup_{\atop{\Omega \subseteq [-N, N]^k}{\Omega \text{ convex}}}
		\left|\sum_{n \in \calA_1 \times \ldots \times \calA_k} e^{2\pi i P(n)} \ind{\Omega}(n) \right|
		\leq
		C N^k (\log N)^{-(k+1) \alpha - k}.
	\end{equation}
	We divide the cube $[-N, N]^k$ into $J$ closed cubes $(Q_j : 1 \leq j \leq J)$ with sides parallel to the axes
	and having side lengths $\calO\big(N (\log N)^{-\alpha-1} \big)$. Thus
	\begin{equation}
		\label{eq:37}
		J = \calO\big((\log N)^{k(\alpha+1)} \big).
	\end{equation}
	By $Q_j^\mathrm{o}$ we denote the interior of $Q_j$. We assume that $Q_j^\mathrm{o}$ are disjoint with the axes.
	Let $n_j$ be the vertex of $Q_j$ at the largest distance to the origin. Then by the mean value theorem and
	\eqref{eq:7}, we have
	\begin{align*}
		\left|
		\sum_{n \in \calA_1 \times \ldots \times \calA_k} e^{2\pi i P(n)} \ind{Q_j^\mathrm{o} \cap \Omega}(n) 
		\big(\phi(n) - \phi(n_j)\big)
		\right|
		&\lesssim
		\sum_{n \in Q_j} \sup_{t \in [0, 1]} \Big|\nabla \phi\big(t n + (1- t) n_j\big)\Big| \cdot \big|n - n_j\big| \\
		&\lesssim
		N (\log N)^{-\alpha-1} \sum_{n \in Q_j} \big(1 + \abs{n}\big)^{-1}, 
	\end{align*}
	thus
	\begin{equation}
		\label{eq:35}
		\left|
		\sum_{j = 1}^J
        \sum_{n \in \calA_1 \times \ldots \times \calA_k} e^{2\pi i P(n)} \ind{Q_j^\mathrm{o} \cap \Omega}(n) 
        \big(\phi(n) - \phi(n_j)\big)
        \right|
		\lesssim
		N^k (\log N)^{-\alpha}.
	\end{equation}
	On the other hand, in view of \eqref{eq:30}, we get
	\[
		\left|
		\phi(n_j) 
		\sum_{n \in \calA_1 \times \ldots \times \calA_k} e^{2\pi i P(n)} \ind{Q_j^\mathrm{o} \cap \Omega}(n)
		\right|
		\lesssim
		N^k (\log N)^{-(k+1)\alpha - k},
	\]
	hence, by \eqref{eq:37},
	\[
		\left|
		\sum_{j = 1}^J
        \phi(n_j)
        \sum_{n \in \calA_1 \times \ldots \times \calA_k} e^{2\pi i P(n)} \ind{Q_j^\mathrm{o} \cap \Omega}(n)
        \right|
		\lesssim
		N^k (\log N)^{-\alpha},
	\]
	which together with \eqref{eq:35} completes the proof.
\end{proof}
We next apply Theorem \ref{thm:1} to get the following variant of Weyl--Vinogradov's inequality. 
\begin{theorem}
	\label{thm:9}
	Let $\xi \in \TT^d$. Assume that there is a multi-index $\gamma_0 \in \Gamma$, such that
	\[
		\bigg|\xi_{\gamma_0} - \frac{a}{q} \bigg| \leq \frac{1}{q^2},
	\]
	for some coprime integers $a$ and $q$ such that $1 \leq a \leq q$. Then for all $\alpha > 0$, there is
	$\beta_\alpha > 0$, so that for any $\beta > \beta_\alpha$, if
	\[
		(\log N)^\beta \leq q \leq N^{\abs{\gamma_0}} (\log N)^{-\beta},
	\]
	then 
	\[
		\sup_{\atop{\Omega \subseteq [-N, N]^k}{\Omega \text{ convex}}}
		\bigg|
		\sum_{n \in \NN^{k'}} \sum_{p \in \PP^{k''}} e^{2\pi i \sprod{\xi}{\calQ(n, p)}}
		\ind{\Omega}(n, p) \phi(n, p)
		\bigg(\prod_{j = 1}^{k''} \log p_j \bigg) 
		\bigg|
		\leq C N^k (\log N)^{-\alpha}.
	\]
	The constant $C$ depends on $\alpha$, $d$ and a constant in \eqref{eq:7}.
\end{theorem}
\begin{proof}
	We claim that the following holds true.
	\begin{claim}
		\label{clm:4}
		For all $\alpha > 0$, there is $\beta_\alpha > 0$, such that for all $\beta > \beta_\alpha$,
		$N \in \NN$, and $r \in \{0, \ldots, k''\}$, if there is a multi-index
		$\gamma_0 \in \Gamma$, such that
		\[
			\bigg|\xi_{\gamma_0} -\frac{a}{q}\bigg| \leq \frac{1}{q^2},
		\]
		for some coprime integers $a$ and $q$, such that $1 \leq a \leq q$, and
		\[
			(\log N)^\beta \leq q \leq N^{\abs{\gamma_0}} (\log N)^{-\beta},
		\]
		then
		\[
			\sup_{\atop{\Omega \subseteq [-N, N]^k}{\Omega \text{ convex}}}
			\bigg|
			\sum_{n \in \NN^{k'}}
			\sum_{p \in \PP^{k''}}
        	e^{2\pi i \sprod{\xi}{\calQ(n, p)}} \ind{\Omega}(n, p) \phi(n, p)
        	\bigg(\prod_{j = r+1}^{k''} \log p_j\bigg)
			\bigg|
			\leq C N^k (\log N)^{-\alpha+k'' - r}.
		\]
	\end{claim}
	The proof is by a backward induction over $r$. For $r = k''$ the assertion follows by Theorem \ref{thm:1}. For 
	$r \in \{1, \ldots, k''\}$, $N \in \NN$ and $m \in \{1, \ldots, N\}$, we set
	\[
		S^{(r)}_{N, m}(\xi)
		= 
		\sum_{n \in \NN^{k'}}
		\sum_{\atop{p \in \PP^{k''}}{p_r \leq m}}
		e^{2\pi i \sprod{\xi}{\calQ(n, p)}} \ind{\Omega}(n, p) \phi(n, p)
		\bigg(\prod_{j = r+1}^{k''} \log p_j\bigg),
	\]
	and
	\[
		S^{(0)}_{N, N}(\xi)
		=
        \sum_{n \in \NN^{k'}}
        \sum_{p \in \PP^{k''}}
        e^{2\pi i \sprod{\xi}{\calQ(n, p)}} \ind{\Omega}(n, p) \phi(n, p)
        \bigg(\prod_{j = 1}^{k''} \log p_j\bigg),
	\]
	where $\Omega$ is a convex subset of $[-N, N]^k$. For $0 \leq r < k''$, by the partial summation, we can write
	\begin{align*}
		S_{N, N}^{(r)} &=
		\sum_{m = 1}^N
		\left(S_{N, m}^{(r+1)} - S_{N, m-1}^{(r+1)}\right) \log m\\
		&=
		S_{N, N}^{(r+1)} (\log N) + \sum_{m = 1}^{N-1} S_{N, m}^{(r+1)} \big(\log(m) - \log(m+1)\big).
	\end{align*}
	Hence, by the inductive hypothesis we get
	\begin{align*}
		\abs{S_{N, N}^{(r)}} 
		& \leq 
		\abs{S_{N, N}^{(r+1)}} (\log N) 
		+
		\sum_{m = 1}^{N-1} 
		\left| S_{N, m}^{(r+1)} \right| m^{-1} \\
		& \leq
		C' N^k (\log N)^{-\alpha + k'' - r },
	\end{align*}
	proving the claim. Now, the theorem follows by Claim \ref{clm:4} for $r = 0$.
\end{proof}

\subsection{Gaussian sums}
\label{sec:6}
Given $q \in \NN$ and $a \in \mathbf{A}_q$, the \emph{Gaussian sum} is
\[
	G(a/q) = 
	\frac{1}{q^{k'}} \cdot \frac{1}{\varphi(q)^{k''}} \sum_{x \in \NN^{k'}_q} \sum_{y \in A_q^{k''}} 
	e^{2\pi i \sprod{(a/q)}{\calQ(x, y)}},
\]
where $\varphi$ is Euler's totient function, i.e $\varphi(q)$ equals to the number of elements in $A_q$.
The following theorem provides a very useful estimate on the Gaussian sums.
\begin{theorem}
	\label{thm:2}
	There are $C > 0$ and $\delta > 0$ such that for all $q \in \NN$ and $a \in \bfA_q$,
	\[
		\big|G(a/q) \big| \leq C q^{-\delta}.
	\]
\end{theorem}
\begin{proof}
	Let us recall that for $a, q \in \NN$, (see e.g. \cite[Theorem 4.1]{mv})
	\begin{equation}
		\label{eq:128}
		\frac{1}{ \varphi(q)} 
		\sum_{x \in A_q} e^{2\pi i a x / q} = 
		\frac{\mu(q/\gcd(a, q))}{\varphi(q/\gcd(a, q))},
	\end{equation}
	wherein $\mu(q)$ is M\"obius function defined for $q = p_1^{j_1} \cdots p_m^{j_m}$, $p_j$ are distinct
	prime numbers, as
	\[
		\mu(q) = 
		\begin{cases}
			(-1)^m & \text{if } j_1 = \cdots = j_m = 1, \\
			0 & \text{otherwise.}
		\end{cases}
	\]
	For each $\epsilon > 0$ there is $C_\epsilon > 0$, such that (see e.g. \cite[Theorem 2.9]{mv})
	\begin{equation}
		\label{eq:129}
		\varphi(q) \geq C_\epsilon q^{1-\epsilon}.
	\end{equation}
	We start the proof of the theorem by considering $d = 1$. Then
	\[
		G(a/q) =
        \prod_{\gamma = (\gamma', 0) \in \Gamma}
        \bigg(\frac{1}{q} \sum_{x = 1}^q e^{2\pi i a_\gamma x /q} \bigg)
        \prod_{\gamma = (0, \gamma'') \in \Gamma}
        \bigg(\frac{1}{\varphi(q)} \sum_{x \in A_q} e^{2\pi i a_\gamma x /q} \bigg).
	\]
	Suppose that $k' \geq 1$. If $G(a/q) \neq 0$ then
	$q \mid a_\gamma$ for all $\gamma = (\gamma', 0) \in \Gamma$. Since $a \in \bfA_q$,
	we must have $k'' \geq 1$. For $\gamma = (0, \gamma'') \in \Gamma$, we set $b_\gamma/q_\gamma = a_\gamma/q$, where
	$(b_\gamma, q_\gamma) = 1$. By \eqref{eq:128}, $G(a/q) \neq 0$ entails that each $q_\gamma$ is square-free. Since
	for any $p$ prime factor $q$ there is $\gamma = (0, \gamma'') \in \Gamma$ such that $p \not\mid q/q_\gamma$,
	we conclude that $q$ is square-free. Because $q = \lcm\big(q_\gamma : \gamma = (0, \gamma'') \in \Gamma \big)$,
	\[
		\abs{G(a/q)}
		\leq 
		\prod_{\gamma = (0, \gamma'') \in \Gamma} 
		\frac{1}{\varphi(q_\gamma)}
		\leq
		\frac{1}{\varphi(q)},
	\]
	which together with \eqref{eq:129} gives
	\[
		\big| G(a/q) \big|
		\leq C_\epsilon q^{\epsilon-1}.
	\]
	Next, let us consider the case $d \geq 2$. For a given polynomial $P$ on $\RR^k$ with integral coefficients we 
	define
	\[
		S(q, P) = \sum_{x \in \NN_q^{k'}} \sum_{y \in A_q^{k''}} \exp\big(2\pi i P(x, y)/q\big).
	\]
	Let 
	\[
		P(x) = \sum_{0 < \abs{\gamma} \leq d} a_\gamma x^\gamma,
	\]
	where $a \in \bfA_q$. Our aim is to show that there are $C > 0$ and $\delta > 0$ such that for all $q \in \NN$ and 
	$a \in \bfA_q$,
	\begin{equation}
		\label{eq:127}
		\big| S(q, P) \big| \leq C q^{k-\delta}.
	\end{equation}
	First, observe that for $q = q_1 q_2$, $(q_1, q_2) = 1$, we have
	\[
		S(q, P) = S\big(q_1, q_2^{-1} P(q_2 {\: \cdot \:})\big) S\big(q_2, q_1^{-1} P(q_1 {\: \cdot \:})\big).
	\]
	Therefore, if $q = p_1^{j_1} \cdots p_m^{j_m}$ for some distinct prime numbers $p_j$, then
	\[
		S(q, P) = \prod_{s = 1}^m S\big(p_s^{j_s}, P_s\big),
	\]
	where
	\[
		P_s(x) = \frac{p_s^{j_s}}{q} P\bigg(\frac{q}{p_s^{j_s}} x \bigg).
	\]
	Since $\omega(q)$, the number of distinct prime factors of $q$, satisfies (see e.g. \cite[Theorem 2.10]{mv})
	\[
		\omega(q) \leq C \frac{\log q}{\log \log q},
	\]
	we have
	\begin{align*}
		2^{\omega(q)} &\leq C_\epsilon' q^{\epsilon}.
	\end{align*}
	Hence, it is enough to proof \eqref{eq:127} for $q = p^j$ with $p$ being a prime number and $j \geq 1$. Since for any
	arithmetic function $F$, we have
	\[
		\sum_{x \in A_{p^j}} F(x) = \sum_{x \in \NN_{p^j}} F(x) - \sum_{x \in \NN_{p^{j-1}}} F(p x),
	\]
	if $j \geq 2$ we write
	\[
		S(p^j, P) = 
		\sum_{\sigma \in \{0, 1\}^{k''}}
		(-1)^\sigma
		\sum_{(x', x'') \in \Omega^\sigma} \exp\big( 2\pi i P(x', p^\sigma x'')/p^j \big),
	\]
	where for $\sigma \in \{0, 1\}^{k''}$, we have set
	\[
		\Omega^\sigma = \NN_{p^j}^{k'} \times \NN_{p^{j-\sigma_1}} \times \ldots \times \NN_{p^{j-\sigma_{k''}}}.
	\]
	Fix $\sigma \in \{0, 1\}^{k''}$. For each $\gamma \in \Gamma$, we define 
	\[
		\frac{b_\gamma}{q_\gamma} = 
		\frac{a_\gamma p^{\sprod{\sigma}{\gamma''}}}{p^j},
	\]
	where $(b_\gamma, q_\gamma) = 1$. Let
	\[
		q = \lcm\big(q_\gamma : \gamma \in \Gamma, \abs{\gamma} \geq 2\big), \qquad\text{and}\qquad
		Q = \lcm\big(q_\gamma : \gamma \in \Gamma \big).
	\]
	Observe that
	\begin{equation}
		\label{eq:41}
		\sum_{(x', x'') \in \Omega^\sigma} \exp\big( 2 \pi i P(x', p^\sigma x'')/p^j\big) \neq 0
	\end{equation}
	entails that $q = Q$. To obtain a contradiction, suppose that $q < Q$. Let $\gamma_0 \in \Gamma$, $\abs{\gamma_0} = 1$
	be such that $q_{\gamma_0} = Q$. Thus $q \mid p^{j-\sigma_1}$. For any $r \in \NN_q^k$ we can wright
	\[
		\begin{aligned}
		&
		\sum_{\atop{(x', x'') \in \Omega^\sigma}{x \equiv r \bmod q}} 
		\exp\big(2\pi i P(x', p^\sigma x'')/p^j\big) \\
		&\qquad\qquad=
		\exp\big(2 \pi i \tilde{P}(r', p^\sigma r'')/p^j \big)
		\sum_{\atop{x \in \Omega^\sigma}{x \equiv r \bmod q}}
		\prod_{\atop{\gamma \in \Gamma}{\abs{\gamma} = 1}}
		\exp\big(2\pi i b_\gamma x^\gamma / q_\gamma\big),
		\end{aligned}
	\]
	where
	\[
        \tilde{P}(x) = \sum_{\atop{\gamma \in \Gamma}{\abs{\gamma} \geq 2}} a_\gamma x^\gamma.
    \]
	Thus \eqref{eq:41} implies that $q_{\gamma_0} \mid b_{\gamma_0} q$, which is	
	impossible. Hence, $q = Q$.

	Now, let $\gamma_0 \in \Gamma$, $\abs{\gamma_0} \geq 2$, be such that $q_{\gamma_0} = Q$. Then
	\[
		q_{\gamma_0} = Q = q \leq \prod_{\atop{\gamma \in \Gamma}{\abs{\gamma} \geq 2}} q_\gamma \leq q_{\gamma_0}^d,
	\]
	and thus
	\begin{equation}
		\label{eq:43}
		Q^{1/d} \leq q_{\gamma_0} = Q < Q^{\abs{\gamma_0} - 1/d}.
	\end{equation}
	Suppose that $Q < p^j$. Since $a \in \bfA_q$, we must have $\sigma \neq 0$. Then for
	$j \leq D = \max\{\abs{\gamma''} : \gamma \in \Gamma\}$, by a trivial estimate we have
	\[
		\Big|
		\sum_{(x', x'') \in \Omega^\sigma} \exp\big(2\pi i P(x', p^\sigma x'')/p^j\big)
		\Big|
		\leq
		p^{kj - \abs{\sigma}}
		\leq
		p^{kj(1 - \delta_1)},
	\]
	provided $0 < \delta_1 < (k D)^{-1}$. Since $Q \geq p^{j-D}$, for $j \geq D+1$ we have
	\[
		Q^{1/d} \geq p^{j \epsilon},
	\]
	whenever $0 < \epsilon < (d (D+1))^{-1}$. Hence, by \eqref{eq:43},
	\[
		p^{j\epsilon} \leq q_{\gamma_0} \leq p^{j(\abs{\gamma_0}-\epsilon)}.
	\]
	Obviously, the last estimate is also valid for $Q = p^j$. Since $\Omega^\sigma \subseteq \NN_{p^j}^k$, by
	\cite[Proposition 3]{SW0}, there are $C > 0$ and $\delta_2 > 0$ such that
    \[
        \Big|
        \sum_{(x', x'') \in \Omega^\sigma} 
		\exp\big( 2\pi i P(x', p^\sigma x'')/p^j \big)
        \Big|
        \leq
		C p^{j(k - \delta_2)},
	\]
	which finishes the proof of \eqref{eq:127} for $q=p^j$, and the theorem follows.
\end{proof}

\section{Multipliers}
In this section we develop some estimates on discrete Fourier multipliers corresponding to operators $M_N$ and $H_N$.
\subsection{Averaging operators}
\label{sec:8}
For a function $f: \ZZ^d \rightarrow \CC$ with a finite support we have
\[
	M_N f(x) = \calF^{-1}\big(\frkM_N \hat{f}\big)(x),
\]
where $\frkM_N$ is the discrete Fourier multiplier
\[
	\frkM_N(\xi) = \frac{1}{\vartheta_B(N)}
	\sum_{n \in \NN^{k'}} \sum_{p \in \PP^{k''}} 
	e^{2\pi i \sprod{\xi}{\calQ(n, p)}} 
	\ind{B_N}(n, p)
	\bigg(\prod_{j = 1}^{k''} \log p_j\bigg),
\]
wherein $\vartheta_B$ is the Chebyshev function
\[
	\vartheta_B(N) = \sum_{n \in \NN^{k'}} \sum_{p \in \PP^{k''}} \ind{B_N}(n, p) 
	\bigg(\prod_{j = 1}^{k''} \log p_j\bigg).
\]
By \eqref{eq:151} and the prime number theorem,
\begin{equation}
	\label{eq:149}
	\vartheta_B(N) \simeq N^k.
\end{equation}
Next, let us define
\[
	\Phi_N(\xi) = \frac{1}{|B|} \int_B e^{2\pi i \sprod{\xi}{\calQ(N x)}} {\: \rm d} x,
\]
where $\abs{B}$ denotes Euclidean measure of $B$. By a multi-dimensional version of van der Corput's lemma
(see \cite[Proposition 2.1]{sw}) we have
\[
	\abs{\Phi_N(\xi)} \lesssim \min\left\{1, \abs{N^A \xi}_\infty^{-1/d}\right\},
\]
where $A$ is the matrix defined in \eqref{eq:21}. Moreover,
\begin{equation}
	\label{eq:26}
	\abs{\Phi_N(\xi) - 1} \lesssim \min\left\{1, \abs{N^A \xi}_\infty \right\}.
\end{equation}
Therefore, for $N < N' \leq 2N$, we have
\begin{equation}
	\label{eq:27}
	\abs{\Phi_N(\xi) - \Phi_{N'}(\xi)} \lesssim \min\left\{\abs{N^A \xi}_\infty , \abs{N^A \xi}_\infty^{-1/d} \right\}.
\end{equation}
We start with the following proposition.
\begin{proposition}
	\label{prop:1}
	For each $\beta' > 0$ there are $C, c > 0$ such that for all $N \in \NN$, and $\xi \in \TT^d$ satisfying
	\begin{equation}
		\label{eq:2}
		\bigg|\xi_{\gamma} - \frac{a_{\gamma}}{q}\bigg| \leq N^{-\abs{\gamma}} L,
		\qquad\text{for all } \gamma \in \Gamma,
	\end{equation}
	where $1 \leq q \leq (\log N)^{\beta'}$, $a \in \bfA_q$, and $1 \leq L \leq \exp\big(c \sqrt{\log N}\big)$, we have
	\[
		\Big|\frkM_N(\xi) - G(a/q) \Phi_N(\xi - a/q) \Big| 
		\leq C L\exp\big(-c\sqrt{\log N}\big).
	\]
	The constant $c$ is absolute.
\end{proposition}
\begin{proof}
	Observe that for a prime number $p$, $p \mid q$ if and only if $(p \bmod q, q) > 1$. Hence, for each 
	$s \in \{1, \ldots, k''\}$, we have
	\[
		\bigg|
		\sum_{u \in \NN^{k'}}
		\sum_{\atop{r'' \in \NN_q^{k''}}{(r_s'', q) > 1}} 
		\sum_{\atop{p \in \PP^{k''}}{p \equiv r'' \bmod q}}
		e^{2\pi i \sprod{\xi}{\calQ(u, p)}}
		\ind{B_N}(u, p)
		\bigg(\prod_{j=1}^{k''} \log p_j\bigg)
		\bigg|
		\lesssim
		N^{k-1}
		\sum_{p \mid q} \log p
		\lesssim
		N^{k-1} \log q.
	\]
	Let $\theta = \xi - a/q$. Then by \eqref{eq:2},
	\begin{equation}
		\label{eq:5}
		\abs{\theta_\gamma} \leq N^{-\abs{\gamma}} L, \qquad\text{for all } \gamma \in \Gamma.
	\end{equation} 
	Since for $(u, p) \in \NN^{k'} \times \PP^{k''}$ such that $u \equiv r' \bmod q$, and 
	$p \equiv r'' \bmod q$,
	\[
		\begin{aligned}
		\xi_\gamma u^{\gamma'} p^{\gamma''} &\equiv \frac{a_\gamma}{q} u^{\gamma'} p^{\gamma''} 
		+ \theta_\gamma n^{\gamma'} p^{\gamma''} \pmod 1\\
		&\equiv \frac{a_\gamma}{q} (r')^{\gamma'} (r'')^{\gamma''} + \theta_\gamma u^{\gamma'} p^{\gamma''} \pmod 1,
		\end{aligned}
	\]
	we have
	\begin{equation}
		\label{eq:6}
		\begin{aligned}
		&
		\sum_{u \in \NN^{k'}} \sum_{p \in \PP^{k''}} e^{2\pi i \sprod{\xi}{\calQ(u, p)}} \ind{B_N}(u, p)
		\bigg(\prod_{j = 1}^{k''} \log p_j\bigg)
		\\
		&\qquad\qquad=
		\sum_{r' \in \NN^{k'}_q} 
		\sum_{r'' \in A_q^{k''}}
		e^{2\pi i \sprod{(a/q)}{\calQ(r', r'')}} 
		\sum_{\atop{u \in \NN^{k'}}{u \equiv r' \bmod q}}
		\sum_{\atop{p \in \PP^{k''}}{p \equiv r'' \bmod q}}
		e^{2\pi i \sprod{\theta}{\calQ(u, p)}}
		\ind{B_N}(u, p) 
		\bigg(\prod_{j = 1}^{k''} \log p_j\bigg)
		\\ 
		&\qquad\qquad\phantom{=}+
		\calO\Big(N^{k-1} \log \log N\Big).
		\end{aligned}
	\end{equation}
	Let us fix $u \in \NN^{k'}$, $\tilde{p} \in \PP^{k''-1}$ and $r_1'' \in A_q$. Then
	\[
		\left\{v \in \NN : (u, v, \tilde{p}) \in B_N\right\} = \left(V_0+1, \ldots, V_1\right),
	\]
	for some $0 \leq V_0 \leq V_1 \leq N$. Let $\tilde{V}_0 = \max\big\{N^{1/2}, V_0\big\}$ and
	$\tilde{V}_1 = \max\big\{N^{1/2}, V_1\big\}$. We have
	\[
		\sum_{\atop{p_1 \in \PP_{V_1} \setminus \PP_{V_0}}{p_1 \equiv r_1'' \bmod q}}
        e^{2\pi i \sprod{\theta}{\calQ(u, p_1, \tilde{p})}} \log p_1
		=
		\sum_{\atop{p_1 \in \PP_{\tilde{V}_1} \setminus \PP_{\tilde{V}_0}}{p_1 \equiv r_1'' \bmod q}}
        e^{2\pi i \sprod{\theta}{\calQ(u, p_1, \tilde{p})}} \log p_1
		+
		\calO\big(N^{1/2}\big).
	\]
	By the partial summation we obtain
	\begin{equation}
		\label{eq:3}
		\begin{aligned}
		\sum_{\atop{p_1 \in \PP_{\tilde{V}_1} \setminus \PP_{\tilde{V}_0}}{p_1 \equiv r_1'' \bmod q}}
		e^{2\pi i \sprod{\theta}{\calQ(n, p_1, \tilde{p})}} \log p_1 
		&=
		\sum_{\atop{\tilde{V}_0 < v_1 \leq \tilde{V}_1}{v_1 \equiv r_1'' \bmod q}}
		e^{2\pi i \sprod{\theta}{\calQ(u, v_1, \tilde{p})}} 
		\ind{\PP}(v_1) \log v_1 \\
		&=
		\vartheta(\tilde{V}_1; q, r''_1) e^{2\pi i \sprod{\theta}{\calQ(u, \tilde{V}_1, \tilde{p})}}
		-
		\vartheta(\tilde{V}_0; q, r''_1) e^{2\pi i \sprod{\theta}{\calQ(u, \tilde{V}_0, \tilde{p})}} \\
		&\phantom{=} 
		- \int_{\tilde{V}_0}^{\tilde{V}_1} \vartheta(t; q, r) 
		\frac{{\rm d}}{{\rm d} t} \Big(e^{2\pi i \sprod{\theta}{\calQ(u, t, \tilde{p})}} \Big) {\: \rm d} t,
		\end{aligned}
	\end{equation}
	where for $x \geq 2$, we have set
	\[
		\vartheta(x; q, r) = \sum_{\atop{p \in \PP_x}{p \equiv r \bmod q}} \log p.
	\]
	Analogously, we can write
	\begin{equation}
		\label{eq:4}
		\begin{aligned}
		\sum_{V_0 < v_1 \leq V_1} 
		e^{2\pi i \sprod{\theta}{\calQ(u, v_1, \tilde{p})}} 
		&=
		\tilde{V}_1 e^{2\pi i \sprod{\theta}{\calQ(u, \tilde{V}_1, \tilde{p})}} 
		- \tilde{V}_0 e^{2\pi i \sprod{\theta}{\calQ(u, \tilde{V}_0, \tilde{p}})} \\
		&\phantom{=}
		-\int_{\tilde{V}_0}^{\tilde{V}_1} 
		t \frac{{\rm d}}{{\rm d} t} \Big(e^{2\pi i \sprod{\theta}{\calQ(u, t, \tilde{p})}}\Big) {\: \rm d} t
		+ \calO\big(N^{1/2} \big).
		\end{aligned}
	\end{equation}
	Furthermore, in view of the Siegel--Walfisz theorem (\cite{sieg, wal}, see also \cite[Corollary 11.21]{mv}), 
	there are $C, c > 0$ such that for all $x \geq 2$, $(r, q) = 1$ and $1 \leq q \leq (\log x)^{2 \beta'}$,
	\begin{equation}
		\label{eq:1}
		\bigg|\vartheta(x; q, r) - \frac{x}{\varphi(q)} \bigg| \leq C x \exp\big(-c \sqrt{\log x}\big).
	\end{equation}
	Hence, by \eqref{eq:3}, \eqref{eq:4} and \eqref{eq:5}, we obtain
	\begin{align*}
		&\bigg| \sum_{\atop{p_1 \in \PP_{V_1}\setminus\PP_{V_0}}{p_1 \equiv r_1'' \bmod q}}
        e^{2\pi i \sprod{\theta}{\calQ(u, p_1, \tilde{p})}} \log p_1 -
		\frac{1}{\varphi(q)}
		\sum_{V_0 < v_1 \leq V_1} e^{2\pi i \sprod{\theta}{\calQ(u, v_1, \tilde{p})}}
		\bigg| \\
		&\qquad\qquad\lesssim
		N^{1/2}
		+
		\bigg|\vartheta(\tilde{V}_1; q, r''_1) - \frac{\tilde{V}_1}{\varphi(q)}\bigg|
		+
		\bigg|\vartheta(\tilde{V}_0; q, r''_1) - \frac{\tilde{V}_0}{\varphi(q)}\bigg| \\
		&\qquad\qquad\phantom{\lesssim}
		+
		\Big( \sum_{\gamma \in \Gamma} \abs{\theta_\gamma} N^{\abs{\gamma}-1}\Big)
		\int_{\tilde{V}_0}^{\tilde{V}_1}
		\bigg|\vartheta(t; q, r''_1) - \frac{t}{\varphi(q)} \bigg| {\: \rm d} t \\
		&\qquad\qquad\lesssim
		N \exp\big(- c \sqrt{\log N}\big)
		+
		L N^{-1}
		\int_{N^{1/2}}^{N} t \exp\big(-c \sqrt{\log t}\big){\: \rm d}t.
	\end{align*}
	Thus,
	\[
		\sum_{\atop{p_1 \in \PP_{V_1}\setminus \PP_{V_0}}{p_1 \equiv r_1'' \bmod q}}
        e^{2\pi i \sprod{\theta}{\calQ(u, p_1, \tilde{p})}} \log p_1
		=
		\frac{1}{\varphi(q)}
        \sum_{V_0 < v_1 \leq V_1}e^{2\pi i \sprod{\theta}{\calQ(u, v_1, \tilde{p})}}
		+
		\calO\Big(N L \exp\big(-c\sqrt{\log N}\big)\Big).
	\]
	In view of \eqref{eq:149}, similar arguments applied to the sums over $p_2, \ldots, p_{k''}$ lead to 
	\begin{align*}
		&\sum_{\atop{u \in \NN^{k'}}{u \equiv r' \bmod q}}
		\sum_{\atop{p \in \PP^{k''}}{p \equiv r'' \bmod q}}
		e^{2\pi i \sprod{\theta}{\calQ(u, p)}}
		\ind{B_N}(u, p)
		\bigg(\prod_{j = 1}^{k''} \log p_j \bigg) \\
		&\qquad\qquad=
		\frac{1}{\varphi(q)^{k''}}
		\sum_{u \in \NN_0^{k'}}
		\sum_{v \in \NN^{k''}}
		e^{2\pi i \sprod{\theta}{\calQ(q u+r', v)}} \ind{B_N}(qu+r', v) 
		+
		\calO\Big(N^k L \exp\big(-c\sqrt{\log N}\big)\Big).
	\end{align*}
	By \cite[Proposition 3.1]{mst1}, the number of lattice points in $B_N$ at the distance $<q$
	from the boundary of $B_N$ is $\calO(q N^{k-1})$. Moreover, for each $(x, y) \in [0, 1]^k$, and
	$(q u + q x, v + y) \in B_N$, we have
   	\begin{align*}
        \big|
        \sprod{\theta}{\calQ(q u + q x , v + y )} - \sprod{\theta}{\calQ(q u, v)}
        \big|
        \leq
        C q \sum_{\gamma \in \Gamma} \abs{\theta_\gamma} N^{\abs{\gamma}-1}
        \lesssim
		q N^{-1} L.
    \end{align*}
	Hence, by \eqref{eq:6} and \eqref{eq:149},
	\begin{align*}
		&\sum_{\atop{u \in \NN^{k'}}{u \equiv r' \bmod q}} 
		\sum_{\atop{p \in \PP^{k''}}{p \equiv r'' \bmod q}}
        e^{2\pi i \sprod{\theta}{\calQ(u, p)}} 
		\ind{B_N}(u, p)
		\bigg(\prod_{j = 1}^{k''} \log p_j \bigg) \\
		&\qquad\qquad=
		\frac{1}{\varphi(q)^{k''}}
        \sum_{u \in \NN_0^{k'}}
        \sum_{v \in \NN^{k''}}
        e^{2\pi i \sprod{\theta}{\calQ(q u, v)}} \ind{B_N}(qu+r', v)
		+
        \calO\Big(N^k L \exp\big(-c\sqrt{\log N}\big)\Big) \\
		&\qquad\qquad=
		\frac{1}{\varphi(q)^{k''}}
        \sum_{u \in \NN_0^{k'}}
        \sum_{v \in \NN^{k''}}
        e^{2\pi i \sprod{\theta}{\calQ(q u, v)}} \ind{B_N}(qu, v)
		+
        \calO\Big(N^{k} L \exp\big(-c\sqrt{\log N}\big)\Big).
	\end{align*}
	Finally, another application of the mean value theorem allows us to
	replace the sums by the corresponding integrals. Indeed, we have
	\begin{align*}
		&
		\left|
		\sum_{u \in \NN_0^{k'}}
        \sum_{v \in \NN^{k''}}
        e^{2\pi i \sprod{\theta}{\calQ(q u, v)}} \ind{B_N}(qu, v)
		-
		\iint_{\RR_+^k} e^{2 \pi i \sprod{\theta}{\calQ(qx, y)}} 
		\ind{B_N}(q x, y) {\: \rm d} x {\: \rm d}y \right|\\
		&\qquad\qquad=
		\left|
		\sum_{u \in \NN_0^{k''}} 
		\sum_{v \in \NN^{k'}}
		\int_{u + (0, 1]^{k'}} \int_{v+(0, 1]^{k''}}
		e^{2\pi i \sprod{\theta}{\calQ(q u, v)}} \ind{B_N}(qu, v)
		-
		e^{2\pi i \sprod{\theta}{\calQ(q x, y)}} \ind{B_N}(qx, y)
		{\: \rm d}y {\: \rm d}x
		\right| \\
		&\qquad\qquad\lesssim
		\sum_{u \in \NN_0^{k'}} \sum_{v \in \NN^{k''}}
		\iint_{(0, 1]^k}
		\left|
		e^{2\pi i \sprod{\theta}{\calQ(q u, v)}} - e^{2\pi i \sprod{\theta}{\calQ(q u+qx, v + y)}}
		\right|
		\ind{B_N}(qu, v)
		{\: \rm d}x{\: \rm d}y \\
		&\qquad\qquad\phantom{\lesssim}+
		\sum_{u \in \NN_0^{k'}} 
		\sum_{v \in \NN^{k''}}
		\int_{(0, 1]^k} \Big|\ind{B_N}(qu, v) - \ind{B_N}(q(u + x), v + y)\Big| {\: \rm d}x {\:\rm d}y,
	\end{align*}
	which is again bounded by $q N^{k-1} L$. Therefore,
	\[
		\Big|
		\vartheta_B(N) \frkM_N(\xi) - G(a/q) \abs{B} N^k \Phi_N(\xi - a/q)
		\Big|
		\leq
		C N^k L \exp\big(- c \sqrt{\log N}\big).
	\]
	In particular, taking $\xi = 0$, $a = 0$ and $L = 1$, we obtain
	\begin{equation}
		\label{eq:31}
		\vartheta_B(N) = \abs{B} N^k \Big(1 + \calO\Big(\exp\big(-c\sqrt{\log N}\big)\Big)\Big).
	\end{equation}
	This completes the proof.
\end{proof}

\begin{lemma}
	\label{lem:3}
	For each $\alpha > 0$ there is $C > 0$ such that for all $N \in \NN$, and $\xi \in \TT^d$ satisfying
	\begin{equation}
        \label{eq:173}
        \bigg|\xi_\gamma - \frac{a_\gamma}{q} \bigg| 
		\leq N^{-\abs{\gamma}} L,
        \qquad\text{for all } \gamma \in \Gamma,
    \end{equation}
	where $1 \leq q \leq L$, $a \in \bfA_q$, and $1 \leq L \leq \exp\big(c \sqrt{\log N}\big) (\log N)^{-\alpha}$, 
	we have
    \[
		\Big|
        \frkM_N(\xi) 
		- G(a/q) \Phi_N(\xi - a/q)
		\Big|
		\leq
		C 
		(\log N)^{-\alpha}.
    \]
\end{lemma}
\begin{proof}
	Given $\alpha > 0$, let $\beta' \geq d \beta_\alpha$, where $\beta_\alpha$ is the value determined in Theorem
	\ref{thm:9}. 

	Suppose that \eqref{eq:173} holds for some $(\log N)^{\beta'} < q \leq L$ and $a \in \bfA_q$. For each
	$\gamma \in \Gamma$, by Dirichlet's principle there are coprime integers $a'_\gamma$ and $q'_\gamma$ such that
	$1 \leq a_\gamma' \leq q_\gamma' \leq N^{\abs{\gamma}} L^{-1} (\log N)^{-\beta'/d}$,
	and satisfying
	\[
		\bigg|\xi_\gamma - \frac{a_\gamma'}{q_\gamma'}\bigg| 
		\leq \frac{1}{q_\gamma'} N^{-\abs{\gamma}} L (\log N)^{\beta'/d}.
	\]
	Assume that for some $\gamma \in \Gamma$, $(\log N)^{\beta'/d} \leq 
	q_\gamma' \leq N^{\abs{\gamma}} L^{-1} (\log N)^{-\beta'/d}$. Then, by Theorem \ref{thm:9}, we have
	\[
		\big|\frkM_N(\xi)\big| \leq C (\log N)^{-\alpha}.
	\]
	If for all $\gamma \in \Gamma$, $1 \leq q_\gamma' \leq (\log N)^{\beta'/d}$, then we set
	$q'' = \lcm\big(q_\gamma' : \gamma \in \Gamma\big)$ and $a''_\gamma = a_\gamma' q''/q_\gamma'$ getting
	$1 \leq q'' \leq (\log N)^{\beta'}$ and $a'' \in \bfA_{q''}$ with
	\[
		\bigg|\xi_\gamma - \frac{a_\gamma'}{q_\gamma'}\bigg|
		=
		\bigg|\xi_\gamma -\frac{a_\gamma''}{q''} \bigg| 
		\leq N^{-\abs{\gamma}} L (\log N)^{\beta'/d}.
	\]
	Since $a'/q' \neq a/q$,
	\begin{align*}
		(\log N)^{-\beta'} L^{-1}
		\leq 
		\frac{1}{q'' q} \leq \bigg|\frac{a_\gamma''}{q''} - \frac{a_\gamma}{q}\bigg|
		&\leq
		\bigg|\xi_\gamma -\frac{a_\gamma''}{q''} \bigg|
		+
		\bigg|\xi_\gamma -\frac{a_\gamma}{q} \bigg| \\
		&\leq
		N^{-\abs{\gamma}} L \big(1 + (\log N)^{\beta'/d}\big),
	\end{align*}
	which is possible only for finite number of $N$'s.

	Finally, in the case when $1 \leq q \leq (\log N)^{\beta'}$, by Proposition \ref{prop:1}, we obtain
	\[
		\frkM_N(\xi)= G(a/q) \Phi_N(\xi - a/q) 
		+ \calO\big((\log N)^{-\alpha}\big),
	\]
	which concludes the proof.
\end{proof}

\begin{lemma}
	\label{lem:2}
	For all $p \in [1, \infty)$, $N_1, N_2 \in \NN$, $N_1 < N_2$, and any $f \in \ell^p\big(\ZZ^d\big)$,
	\[
		\Big\|
		\sum_{n = N_1}^{N_2-1}
		\big|
		M_{n+1} f - M_n f
		\big|
		\Big\|_{\ell^p}
		\leq
		C_p
		N_1^{-k} \big(\vartheta_B(N_2) - \vartheta_B(N_1)\big) \|f\|_{\ell^p}.
	\]
\end{lemma}
\begin{proof}
	Let us denote by $m_n$ the convolution kernel corresponding to $M_n$. Consider $(x, y) \in \NN^{k'} \times
	\PP^{k''}$. If $(x, y) \in B_{N_1}$ then
	\[
		\sum_{n = N_1}^{N_2-1} \big| m_{n+1}(x, y) - m_n(x, y) \big|
		=
		\bigg(\frac{1}{\vartheta_B(N_1)} - \frac{1}{\vartheta_B(N_2)}\bigg) \prod_{j = 1}^{k''} \log y_j. 
	\]
	If $(x, y) \in B_{N_2} \setminus B_{N_1}$ then by setting
	\[
		n_0 = \min \big\{n \in \NN : x \in B_n \big\},
	\]
	we have 
	\begin{align*}
		\sum_{n = N_1}^{N_2-1} \big| m_{n+1}(x, y) - m_n(x, y) \big|
		&=
		\bigg(\frac{1}{\vartheta_B(n_0)} 
		+ \sum_{n = n_0}^{N_2-1} \frac{1}{\vartheta_B(n)} - \frac{1}{\vartheta_B(n+1)}
		\bigg) \prod_{j = 1}^{k''} \log y_j \\
		&=
		\bigg(\frac{2}{\vartheta_B(n_0)} - \frac{1}{\vartheta_B(N_2)}
        \bigg) \prod_{j = 1}^{k''} \log y_j.
	\end{align*}
	Therefore,
	\[
		\Big\|
		\sum_{n = N_1}^{N_2-1} \big| m_{n+1} - m_n \big|
		\Big\|_{\ell^1}
		\leq
		\bigg(\frac{1}{\vartheta_B(N_1)} - \frac{1}{\vartheta_B(N_2)}\bigg) \vartheta_B(N_1)
		+ \frac{2}{\vartheta_B(N_1)}\big(\vartheta_B(N_2) - \vartheta_B(N_1) \big),
	\]
	and hence, by Young's inequality,
	\begin{align*}
		\Big\|
		\sum_{n = N_1}^{N_2-1} \big| M_{n+1} f - M_n f \big|
		\Big\|_{\ell^p}
		&\leq
		\Big\|
        \sum_{n = N_1}^{N_2-1} \big| m_{n+1} - m_n \big|
        \Big\|_{\ell^1}
		\|f\|_{\ell^p}\\
		&\lesssim
		\frac{\vartheta_B(N_2) - \vartheta_B(N_1)}{\vartheta_N(N_1)} \|f\|_{\ell^p}. 
	\end{align*}
	which finishes the proof since $\vartheta_B(N_1) \simeq N_1^k$. 
\end{proof}

\subsection{Truncated discrete singular operators}
\label{sec:9}
In this section we investigate the asymptotic of Fourier multipliers corresponding to the truncated 
discrete singular operators $H_N$ with a kernel $K$ satisfying \eqref{eq:23} and \eqref{eq:24}. Let $\frkH_N$ be the 
Fourier multiplier corresponding to $H_N$, that is for a finitely supported function $f: \ZZ^d \rightarrow \CC$,
\[
	H_N f = \calF^{-1}\big(\frkH_N \hat{f}\big),
\]
where
\[
	\frkH_N(\xi) = \sum_{n \in \ZZ^{k'}} \sum_{p \in (\pm \PP)^{k''}} 
	e^{2\pi i \sprod{\xi}{\calQ(n, p)}} K(n, p) \ind{B_N}(n, p)
	\bigg(\prod_{j = 1}^{k''} \log \abs{p_j}\bigg).
\]
We also define
\[
	\Psi_N(\xi) = \text{p.v.} \iint_{B_N} e^{2\pi i \sprod{\xi}{Q(x, y)}} K(x, y) {\: \rm d}x{\: \rm d}y.
\]
In view of a multi-dimensional version of van der Corput's lemma (see \cite[Proposition 2.1]{sw}), for
$N < N' \leq 2N$,
\[
	\abs{\Psi_N(\xi) - \Psi_{N'}(\xi)} 
	\lesssim
	\min\left\{1, \norm{N^A \xi}_{\infty}^{-1/d} \right\}.
\]
Moreover, by \eqref{eq:24},
\[
	\abs{\Psi_N	(\xi) - \Psi_{N'}(\xi)} 
	\lesssim
	\min \left\{1, \norm{N^A \xi}_\infty \right\}.
\]
Hence,
\begin{equation}
	\label{eq:25}
	\abs{\Psi_N(\xi) - \Psi_{N'}(\xi)} \lesssim 
	\min\left\{\norm{N^A \xi}_\infty, \norm{N^A \xi}_\infty^{-1/d}\right\}.
\end{equation}
We start with a proposition analogous to Proposition \ref{prop:1}.
\begin{proposition}
	\label{prop:2}
	For each $\beta' > 0$ there $C, c > 0$ such that for all $N < N' \leq 2N$, and $\xi \in \TT^d$ satisfying
	\[
		\bigg|\xi_{\gamma} - \frac{a_{\gamma}}{q}\bigg| \leq N^{-\abs{\gamma}} L,
		\qquad\text{for all } \gamma \in \Gamma,
	\]
	where $1 \leq q \leq (\log N)^{\beta'}$, $a \in \bfA_q$, and $1 \leq L \leq \exp\big(c \sqrt{\log N}\big)$, we have
	\[
		\Big|\frkH_{N'}(\xi) - \frkH_N(\xi) 
		- G(a/q) \big(\Psi_{N'}(\xi - a/q) - \Psi_{N}(\xi - a/q) \big)
		\Big| \leq C L\exp\big(-c\sqrt{\log N}\big).
	\]
\end{proposition}
\begin{proof}
	For a prime number $p$, $p \mid q$ if and only if $\big(p \bmod q, q \big) > 1$. Therefore, by \eqref{eq:151},
	\eqref{eq:23}, and the prime number theorem, for any $s \in \{1, \ldots, k''\}$,
	\begin{align*}
		&\Bigg|
		\sum_{u \in \NN^{k'}}
		\sum_{\atop{r'' \in \NN_q^{k''}}{(r_s'', q) > 1}} 
		\sum_{\atop{p \in \PP^{k''}}{p \equiv r'' \bmod q}}
		e^{2 \pi i \sprod{\xi}{\calQ(n, p)}} K(u, p)
		\ind{B_{N'}\setminus B_N}(u, p)
		\bigg(\prod_{j = 1}^{k''} \log p_j\bigg)		
		\bigg| \\
		&\qquad\qquad\lesssim
		N^{-1} \sum_{p \mid q} \log p 
		\lesssim N^{-1} \log q.
	\end{align*}
	To simplify the notations, for $(x, y) \in \RR^k \setminus \{0\}$, we set
	\[
		F(x, y) = e^{2\pi i \sprod{\theta}{Q(x, y)}} K(x, y),
	\]
	where $\theta = \xi - a/q$. For any $(u, p) \in \NN^{k'} \times \PP^{k''}$ such that $u \equiv r' \bmod q$, and
	$p \equiv r'' \bmod q$, we have
	\[
		\xi_\gamma u^{\gamma'} p^{\gamma''} \equiv \frac{a_\gamma}{q} (r')^{\gamma'} (r'')^{\gamma''}  + 
		\theta_\gamma u^{\gamma'} p^{\gamma''} \pmod 1,
	\]
	thus
	\begin{align*}
		&
		\sum_{u \in \NN^{k'}} \sum_{p \in \PP^{k''}}
		e^{2\pi i \sprod{\xi}{\calQ(u, p)}}
		K(u, p)
		\ind{B_{N'} \setminus B_N}(u, p)
		\bigg(\prod_{j = 1}^{k''} \log p_j\bigg) \\
		&\qquad\qquad=
		\sum_{r' \in \NN_q^{k'}} \sum_{r '' \in A_q^{k''}} e^{2 \pi i \sprod{(a/q)}{\calQ(r', r'')}}
		\sum_{\atop{u \in \NN^{k'}}{u \equiv r' \bmod q}} 
		\sum_{\atop{p \in \PP^{k''}}{p \equiv r'' \bmod q}}
		F(u, p) \ind{B_{N'} \setminus B_N}(n, p)
		\bigg(\prod_{j = 1}^{k''} \log p_j\bigg) \\
		&\qquad\qquad\phantom{=}+
		\calO\left(N^{-1} \log \log N\right).
	\end{align*}
	Fix $u \in \NN^{k'}$, $\tilde{p} \in \PP^{k''-1}$ and $r_1'' \in A_q$. Then
	\[
		\left\{
		v \in \NN : (u, v, \tilde{p}) \in B_{N'} \setminus B_N
		\right\}
		=
		\left(
		V_0+1, \ldots, V_1
		\right),
	\]
	for some $1 \leq V_0 \leq V_1 \leq N' \leq 2 N$. Let $\tilde{V}_0 = \max\big\{N^{1/2}, V_0\big\}$ and
	$\tilde{V}_1 = \max\big\{N^{1/2}, V_1\big\}$. We have
	\[
		\sum_{\atop{p_1 \in \PP_{V_1} \setminus \PP_{V_0}}{p_1 \equiv r_1'' \bmod q}}
    	F(u, p_1, \tilde{p}) \log p_1
		=
		\sum_{\atop{p_1 \in \PP_{\tilde{V}_1} \setminus \PP_{\tilde{V}_0}}{p_1 \equiv r_1'' \bmod q}}
		F(u, p_1, \tilde{p}) \log p_1
		+
		\calO\big(N^{-k+1/2}\big).
	\]
	By the partial summation
	\begin{align*}
		\sum_{\atop{p_1 \in \PP_{\tilde{V}_1} \setminus \PP_{\tilde{V}_0} }{p_1 \equiv r_1'' \bmod q}}
		F(u, p_1, \tilde{p})
		\log p_1
		&=
		\sum_{\atop{\tilde{V}_0 < v_1 \leq \tilde{V}_1}{v_1 \equiv r_1'' \bmod q}}
		F(u, v_1, \tilde{p})
		\ind{\PP}(v_1) \log v_1 \\
		&=
		\vartheta(\tilde{V}_1; q, r_1'') F(u, \tilde{V}_1, \tilde{p})
		-
		\vartheta(\tilde{V}_0; q, r_1'') F(u, \tilde{V}_0, \tilde{p}) \\
		&\phantom{=}-
		\int_{\tilde{V}_0}^{\tilde{V}_1}
		\vartheta(t; q, r_1'') 
		\frac{{\rm d}}{{\rm d} t} F(u, t, \tilde{p}) {\: \rm d} t.
	\end{align*}
	Analogously, we have
	\begin{align*}
		\sum_{V_0 < v_1 \leq V_1} F(u, v_1, \tilde{p})
		=
		\tilde{V}_1 F(u, \tilde{V}_1, \tilde{p})
		-
		\tilde{V}_0 F(n, \tilde{V}_0, \tilde{p})
		-
		\int_{\tilde{V}_0}^{\tilde{V}_1}
		t
        \frac{{\rm d}}{{\rm d} t} F(u, t, \tilde{p})
		{\: \rm d} t + \calO\big(N^{-k+1/2}\big).
	\end{align*}
	Hence, by \eqref{eq:1} and \eqref{eq:23}, we obtain
	\begin{align*}
		&\bigg|
		\sum_{\atop{p_1 \in \PP_{V_1} \setminus \PP_{V_0} }{p_1 \equiv r_1'' \bmod q}}
		F(u, p_1, \tilde{p})
		\log p_1
		-
		\frac{1}{\varphi(q)}
		\sum_{V_0 < v_1 \leq V_1} F(u, v_1, \tilde{p})
		\bigg| \\
		&\qquad\qquad\lesssim
		N^{-k+1/2} +
		\bigg|\vartheta(\tilde{V}_1; q, r_1'') - \frac{\tilde{V}_1}{\varphi(q)} \bigg| N^{-k}
		+
		\bigg|\vartheta(\tilde{V}_0; q, r_1'') - \frac{\tilde{V}_0}{\varphi(q)} \bigg| N^{-k} \\
		&\qquad\qquad\phantom{\lesssim}+
		\Big(\sum_{\gamma \in \Gamma} \abs{\theta_\gamma} N^{-k-1+\abs{\gamma}} + N^{-k-1}\Big)
		\int_{\tilde{V}_0}^{\tilde{V}_1} \bigg| \vartheta(t; q, r_1'') - \frac{t}{\varphi(q)} \bigg| {\: \rm d} t \\
		&\qquad\qquad\lesssim
		N^{-k+1} \exp\big(-c \sqrt{\log N}\big) 
		+ L N^{-k-1} \int_{N^{1/2}}^N t \exp\big(-c \sqrt{\log t}\big) {\: \rm d}t.
	\end{align*}
	Therefore,
	\begin{align*}
		\sum_{\atop{p_1 \in \PP_{V_1} \setminus \PP_{V_0} }{p_1 \equiv r_1'' \bmod q}}
     	F(u, p_1, \tilde{p})
		\log p_1
		=
		\frac{1}{\varphi(q)}
        \sum_{V_0 < v_1 \leq V_1} 
		F(u, v_1, \tilde{p})
		+
		\calO\left(N^{-k+1} L \exp\big(-c\sqrt{\log N}\big)\right).
	\end{align*}
	By similar reasonings applied to the sums over $p_2, \ldots, p_{k''}$, one can show that
	\begin{align*}
		&\sum_{\atop{u \in \NN^{k'}}{n \equiv r' \bmod q}}
		\sum_{\atop{p \in \PP^{k''}}{p \equiv r'' \bmod q}}
        F(u, p) \ind{B_{N'}\setminus B_N}(u, p)
		\bigg(\prod_{j = 1}^{k''} \log p_j\bigg) \\
		&\qquad\qquad=
		\frac{1}{\varphi(q)^{k''}}
		\sum_{u \in \NN_0^{k'}}
		\sum_{v \in \NN^{k''}}
		F(qu+r', v)
		\ind{B_{N'} \setminus B_N}(q u + r', v)
		+
		\calO\left(L \exp\big(-c\sqrt{\log N}\big)\right).
	\end{align*}
	Since for each $(x, y) \in [0, 1]^k$ and $(qu+q x, v+y) \in B_{N'} \setminus B_N$, we have
	\[
		\big|
		\sprod{\theta}{Q(q u + qx, v+y)} - \sprod{\theta}{Q(qu, v)}
		\big|
		\lesssim
		q
		\sum_{\gamma \in \Gamma} \abs{\theta_\gamma} N^{\abs{\gamma}-1}
		\lesssim
		q N^{-1} L,
	\]
	and
	\[
		\big|
		K(qu + qx, v+y) - K(qu, v)
		\big|
		\lesssim
		q N^{-k-1},
	\]
	thus by the mean value theorem, we obtain
	\[
		\big|F(qu+qx, v+y) - F(qu, v) \big| \lesssim q N^{-k-1} L.
	\]
	Moreover, in view of \cite[Proposition 3.1]{mst1}, the number of lattice points in $B_N$ of distance $<q$ from
	the boundary of $B_N$ is $\calO(q N^{k-1})$. Therefore, 
	\begin{align*}
		&
		\sum_{\atop{u \in \NN^{k'}}{n \equiv r' \bmod q}} 
		\sum_{\atop{p \in \PP^{k''}}{p \equiv r'' \bmod q}}
		F(u, p) \ind{B_{N'} \setminus B_{N}}(u, p)
		\bigg(\prod_{j = 1}^{k''} \log p_j \bigg) \\
		&\qquad\qquad=
		\frac{1}{\varphi(q)^{k''}}
		\sum_{u \in \NN_0^{k'}} \sum_{v \in \NN^{k''}} F(qu, v) 
		\ind{B_{N'} \setminus B_N}(qu + r', v) +
		\calO\left(L \exp\big(-c\sqrt{\log N}\big)\right) \\
		&\qquad\qquad=
		\frac{1}{\varphi(q)^{k''}}
        \sum_{u \in \NN_0^{k'}} \sum_{v \in \NN^{k''}} 
		F(qu, v)
		\ind{B_{N'} \setminus B_N}(qu, v) +
        \calO\left(L \exp\big(-c\sqrt{\log N}\big)\right).
	\end{align*}
	Lastly, we can replace the sums by the corresponding integrals because
	\begin{align*}
		&\Bigg|
        \sum_{u \in \NN_0^{k'}} \sum_{v \in \NN^{k''}} F(qu, v)
		\ind{B_{N'} \setminus B_N}(qu, v)
		-
		\iint_{\RR_+^k} F(qx, y) \ind{B_{N'}\setminus B_N}(q x, y) 
		{\: \rm d}x {\: \rm d}y
		\Bigg| \\
		&\leq
        \sum_{u \in \NN_0^{k'}} \sum_{v \in \NN^{k''}} 
		\int_{(0, 1]^k}
		\big| F(qu, v) - F(qu+qx,v+y)\big| 
		\ind{B_{N'} \setminus B_N}(q u, v) 
		{\: \rm d} x {\: \rm d}y\\
		&\phantom{\leq}+
		\sum_{u \in \NN_0^{k'}} \sum_{v \in \NN^{k''}}
		\int_{(0, 1]^k}
		\Big|F(qu+qx, v+y) 
		\big( \ind{B_{N'} \setminus B_N}(qu, v) - \ind{B_{N'} \setminus B_N}(qu+qx, v+y)\big) \Big|
		{\: \rm d} x {\: \rm d}y,
	\end{align*}
	which is bounded by $q N^{-1} L$.
\end{proof}

Analogously to Lemma \ref{lem:3}, we can prove the following statement.
\begin{lemma}
	\label{lem:4}
	For each $\alpha > 0$ there is $C > 0$ such that for all $N \leq N' \leq 2N$, and $\xi \in \TT^d$ satisfying
	\[
		\bigg|\xi_\gamma - \frac{a_\gamma}{q} \bigg| 
		\leq N^{-\abs{\gamma}} L,
        \qquad\text{for all }\gamma \in \Gamma,
	\]
	where $1 \leq q \leq L$, $a \in \bfA_q$, and $1 \leq L \leq \exp\big(c \sqrt{\log N} \big)(\log N)^{-\alpha}$,
    \[
		\Big|
        \frkH_{N'}(\xi) - \frkH_N(\xi) - 
		G(a/q) 
		\big(\Psi_{N'}(\xi - a/q) - \Psi_N(\xi - a/q)\big)
		\Big|
		\leq
		C (\log N)^{-\alpha}.
    \]
\end{lemma}

\begin{lemma}
	\label{lem:5}
	For all $p \in [1, \infty)$, $N_1, N_2 \in \NN$, $N_1 < N_2$, and any $f \in \ell^p\big(\ZZ^d\big)$,
	\[
		\Big\|
		\sum_{n = N_1}^{N_2-1}
		\big|
		H_{n+1} f - H_n f
		\big|
		\Big\|_{\ell^p}
		\leq
		C_p
		N_1^{-k} \big(\vartheta_B(N_2) - \vartheta_B(N_1)\big) \|f\|_{\ell^p}.
	\]
\end{lemma}
\begin{proof}
	Let $h_n$ denote the convolution kernel corresponding to $H_n$. Observe that for 
	$(x, y) \in \ZZ^{k'} \times (\pm \PP)^{k''}$, if $(x, y) \in B_{N_2} \setminus B_{N_1}$ then
	\[
		\sum_{n = N_1}^{N_2-1}
		\big|
        h_{n+1}(x, y) - h_n(x, y)
        \big|
		=
		\abs{K(x, y)} \prod_{j = 1}^{k''} \log \abs{y_j},
	\]
	otherwise the sum equals zero. Thus, by \eqref{eq:23}, we obtain
	\[
		\Big\|
		\sum_{n = N_1}^{N_2-1}
        \big|
        h_{n+1} - h_n
        \big|
		\Big\|_{\ell^1}
		\lesssim
		N_1^{-k}
		\big(\vartheta_B(N_2) - \vartheta_B(N_1)\big),
	\]
	hence, by Young's inequality,
	\begin{align*}
		\Big\|
        \sum_{n = N_1}^{N_2-1}
        \big|
        H_{n+1} f - H_n f
        \big|
        \Big\|_{\ell^p}
		&\leq
		\Big\|
        \sum_{n = N_1}^{N_2-1}
        \big|
        h_{n+1} - h_n
        \big|
        \Big\|_{\ell^1} \|f\|_{\ell^p} \\
		&\lesssim
		 N_1^{-k}
        \big(\vartheta_B(N_2) - \vartheta_B(N_1)\big) \|f\|_{\ell^p},
	\end{align*}
	which completes the proof.
\end{proof}

\section{Variational estimates}
\label{sec:7}
In this section we present the estimates for $\ell^p\big(\ZZ^d\big)$ norm of the $r$-variational seminorm for the averaging
operators $(M_N : N \in \NN)$ and the truncated discrete singular operators $(H_N : N \in \NN)$. In order to give a
unified approach, we set $(Y_N : N \in \NN)$ to be any of them. By $(\frkY_N : N \in \NN)$ we denote the corresponding
discrete Fourier multipliers and by $(\Upsilon_N : N \in \NN)$ its continuous counterparts. We start by listing properties
that are sufficient to obtain $r$-variational estimates. Let $\rho \in (0, 1)$ and set
$N_n = \big\lfloor 2^{n^\rho} \big\rfloor$.
\begin{enumerate}[label={\bf Property \arabic*.}, ref=\arabic*, itemindent=*,leftmargin=0pt]
	\item 
	\label{pr:1}
	In view of \cite{mst2} (see also \cite{jsw}) for each $p \in (1, \infty)$ there is 
	$C_p > 0$ such that for all $r \in (2, \infty)$ and any function $f \in L^p\big(\RR^d \big) \cap L^2\big(\RR^d \big)$,
	\[
		\big\|
		V_r\big(\calF^{-1}\big(\Upsilon_N \calF f \big) : N \in \NN \big)
   		\big\|_{L^p}
    	\leq
    	C_p
    	\frac{r}{r-2}
    	\|f\|_{L^p},
	\]
	and
	\[
		\Big\|
		\Big(
		\sum_{n \geq 0}
		V_2\big(\calF^{-1}\big(\Upsilon_N \calF f \big) : N \in [2^n, 2^{n+1})\big)^2\Big)^{1/2}
        \Big\|_{L^p}
        \leq
        C_p
        \|f\|_{L^p}.
	\]
	\item
	\label{pr:2}
	By \eqref{eq:27} and \eqref{eq:25}, for each $n \in \NN$,
	\begin{equation}
		\label{eq:170}
		\left|
		\Upsilon_{N_n}(\xi) - \Upsilon_{N_{n+1}}(\xi)
		\right|
		\lesssim
		\min\left\{\norm{N_n^A \xi}_\infty, \norm{N_n^A \xi}_{\infty}^{-1/d} \right\},
	\end{equation}
	where $A$ is the matrix defined in \eqref{eq:21}.
	\item
	\label{pr:3}
	By Lemma \ref{lem:2} and Lemma \ref{lem:5} we deduce that for each $p \in (1, \infty)$ and any
	$f \in \ell^p\big(\ZZ^d\big)$,
	\begin{equation}
		\label{eq:28}
		\Big\|
        \sum_{N = N_n}^{N_{n+1}-1}
        \big|
        Y_{N+1} f - Y_N f
        \big|
        \Big\|_{\ell^p}
        \leq
        C_{p, \rho} n^{\rho -1 } \|f\|_{\ell^p},
	\end{equation}
	because by \eqref{eq:31},
	\[
        N_n^{-k} \big(
        \vartheta_B(N_{n+1}) - \vartheta_B(N_n) \big)
        \lesssim
        2^{k(n+1)^\rho-kn^\rho} - 1 + e^{-c n^{\rho/2}}
        \lesssim
        n^{\rho-1},
    \]
	In particular,
	\begin{equation}
        \label{eq:121}
        \big\|Y_{N_{n+1}} - Y_{N_n} \big\|_{\ell^p \rightarrow \ell^p} \leq C.
    \end{equation}
	\item 
	\label{pr:4}
	By Theorem \ref{thm:9} and partial summation for each $\alpha > 0$, there is $\beta_\alpha > 0$ so that for any
	$\beta > \beta_\alpha$, and $n \in \NN$, if there is $\gamma_0 \in \Gamma$, such that
	\[
    	\bigg|\xi_{\gamma_0} - \frac{a}{q} \bigg| \leq \frac{1}{q^2},
	\]
	for some coprime numbers $a$ and $q$ such that $1 \leq a \leq q$, and
	$(\log N_n)^\beta \leq q \leq N_n^{\abs{\gamma_0}} (\log N_n)^{-\beta}$, then
	\[
		\big|\frkY_{N_{n+1}}(\xi) - \frkY_{N_n}(\xi) \big| \leq C (\log N_n)^{-\alpha}.
	\]
	\item
	\label{pr:5}
	By Proposition \ref{prop:1} and Proposition \ref{prop:2}, for each $\beta' > 0$ there is $C > 0$ such that for
	all $n \in \NN$, and $\xi \in \TT^d$, satisfying
	\[
        \bigg|\xi_\gamma - \frac{a_\gamma}{q} \bigg| \leq N^{-\abs{\gamma}} L,
        \qquad\text{for all }\gamma \in \Gamma,
	\]
	where $1 \leq q \leq (\log N_n)^{\beta'}$, $a \in \bfA_q$, and $1 \leq L \leq \exp\big(c\sqrt{\log{N_n}}\big)$, 
	we have
    \begin{equation}
		\label{eq:33}
        \frkY_{N_{n+1}}(\xi) - \frkY_{N_n}(\xi) =
        G(a/q) 
		\big(\Upsilon_{N_{n+1}}(\xi - a/q) - \Upsilon_{N_n}(\xi - a/q)\big) 
		+ \calO\Big(L \exp\big(-c \sqrt{\log N_n}\big)\Big).
    \end{equation}
	\item 
	\label{pr:6}
	By Lemma \ref{lem:3} and Lemma \ref{lem:4}, for each $\alpha > 0$, all $n \in \NN$, and $\xi \in \TT^d$,
	satisfying
	\[
        \bigg|\xi_\gamma - \frac{a_\gamma}{q} \bigg| \leq N^{-\abs{\gamma}} L,
        \qquad\text{for all }\gamma \in \Gamma,
	\]
	where $1 \leq q \leq L$, $a \in \bfA_q$, and $1 \leq L \leq \exp\big(c\sqrt{\log{N_n}}\big) (\log N_n)^{-\alpha}$,
	we have
    \begin{equation}
		\label{eq:29}
        \frkY_{N_{n+1}}(\xi) - \frkY_{N_n}(\xi) =
        G(a/q) 
		\big(\Upsilon_{N_{n+1}}(\xi - a/q) - \Upsilon_{N_n}(\xi - a/q)\big) 
		+ \calO\big((\log N_n)^{-\alpha}\big).
    \end{equation}
\end{enumerate}

Before we embark on proving variational estimates, we show the following auxiliary result.
\begin{proposition}
	\label{prop:4}
	For each $p \in (1, \infty)$ there is $C > 0$, such that for all increasing sequences of integers
	$(n_j : j \in \NN)$ and any function $f \in L^p\big(\RR^d\big) \cap L^2\big(\RR^d\big)$,
	\[
		\Big\|
		\Big(
		\sum_{j = 1}^\infty
		\big|
		\calF^{-1} \big((\Upsilon_{n_j} - \Upsilon_{n_{j-1}}) \calF f\big)
		\big|^2
		\Big)^{1/2}
		\Big\|_{L^p}
		\leq
		C \|f\|_{L^p}. 
	\]
\end{proposition}
\begin{proof}
	For each $j \in \NN$, such that
	\[
		2^{n-1} \leq n_j < 2^n \leq 2^m \leq n_{j+1} < 2^{m+1},
	\]
	we write
	\begin{align*}
		\big|
		\calF^{-1} \big((\Upsilon_{n_{j+1}} - \Upsilon_{n_j} ) \calF f\big)
		\big|
        &\leq
		\big|\calF^{-1} \big((\Upsilon_{n_{j+1}} - \Upsilon_{2^m} ) \calF f\big)\big|
		+
		\big|\calF^{-1} \big((\Upsilon_{2^m} - \Upsilon_{2^{n-1}} ) \calF f\big)\big|\\
		&\phantom{\leq}
		+
		\big|\calF^{-1} \big((\Upsilon_{2^{n-1}} - \Upsilon_{n_j} ) \calF f\big)\big| \\
		&\leq
		V_2\big(\calF^{-1} \big(\Upsilon_N \calF f\big) : N \in [2^m, 2^{m+1}) \big)
		+
		\big|
		\calF^{-1} \big((\Upsilon_{2^m} - \Upsilon_{2^{n-1}} ) \calF f\big)
		\big| \\
		&\phantom{\leq}+
		V_2\big(\calF^{-1} \big(\Upsilon_N \calF f\big) : N \in [2^{n-1}, 2^n)\big).
	\end{align*}
	For every $j_1, j_2 \in \NN$, $j_1 < j_2$ such that
	\[
		n_{j_1-1} < 2^n \leq n_{j_1} < n_{j_2} < 2^{n+1} \leq n_{j_2+1},
	\]
	we estimate
	\[
		\sum_{j = j_1+1}^{j_2} \big|\calF^{-1} \big((\Upsilon_{n_j} - \Upsilon_{n_{j-1}} ) \calF f\big) \big|^2
		\leq
		V_2\big(\calF^{-1} \big(\Upsilon_N \calF f\big) : N \in [2^n, 2^{n+1})\big)^2.
	\]
	Hence, for some increasing sequence of integers $(m_j : j \in \NN)$, we have
	\begin{align*}
		\sum_{j = 1}^\infty 
		\big|\calF^{-1} \big((\Upsilon_{n_{j+1}} - \Upsilon_{n_j}) \calF f\big)\big|^2
		&\lesssim
		\sum_{j = 1}^\infty V_2\big(\calF^{-1} \big(\Upsilon_N \calF f\big) : N \in [2^j, 2^{j-1})\big)^2  \\
		&\phantom{\leq}+
		\sum_{j = 1}^\infty \big|\calF^{-1}\big( (\Upsilon_{2^{m_j}} - \Upsilon_{2^{m_{j+1}}} ) \calF f\big) 
		\big|^2.
	\end{align*}
	The conclusion now follows by \cite{DuoRdF} and Property \ref{pr:1}.
\end{proof}

The aim of this section is to prove the following theorem.
\begin{theorem}
	\label{thm:0}
	For each $p \in (1, \infty)$ and $r \in (2, \infty)$ there is $C > 0$ such that for any finitely supported
	function $f: \ZZ^d \rightarrow \CC$,
	\[
		\big\|
		V_r\big( Y_N f : N \in \NN\big)
		\big\|_{\ell^p}
		\leq
		C \frac{r}{r-2} \|f\|_{\ell^p}.
	\]
\end{theorem}
We split a variational seminorm into two parts \emph{long variations} $V_r^L$, and \emph{short variations} $V_r^S$,
where
\[
	V_r^L\big(Y_N f : N \in \NN\big) = V_r\big(Y_{N_n} f : n \in \NN_0\big),
\]
and
\[
	V_r^S\big(Y_N f : N \in \NN\big) = \Big(\sum_{n \geq 0}
	V_r\Big( Y_N f : N \in [N_n, N_{n+1})\Big)^r \Big)^{\frac{1}{r}},
\]
respectively. Then
\begin{equation}
	\label{eq:90}
	V_r\big(Y_N f : N \in \NN\big) \lesssim V^L_r\big(Y_N f : N \in \NN\big)+ V_r^S\big(Y_N f : N \in \NN\big).
\end{equation}
We first estimate $\ell^p$-norm of long variations.

\subsection{Long variations}
\label{sec:10}
Let $\beta \in \NN$ which value will be determined later. Take $\rho \in (0, 1)$ and
$0 < \chi < \tfrac{1}{10}\min\{1, c\}$ where $c$ is the constant from Lemma \ref{lem:3}. For each $n \in \NN$,
we define the multiplier
\[
	\Xi^\beta_n(\xi) = \sum_{a/q \in \mathscr{U}_{\lfloor n^\rho \rfloor}^\beta} \eta_n(\xi - a/q),
\]
where the sets $\scrU_{\lfloor n^\rho \rfloor}^\beta$ are given by \eqref{eq:91} and
\[
	\eta_n(\xi) = \eta\Big(2^{-\chi \sqrt{\log N_n}} N_n^A \xi \Big).
\]
We write
\begin{equation}
	\label{eq:54}
	\begin{aligned}
	\big\|V_r\big(Y_{N_n} f : n \in \NN\big) \big\|_{\ell^p}
	&=
	\Big\|V_r\Big(\sum_{j = 1}^n \calF^{-1}\big( (\frkY_{N_j} - \frkY_{N_{j-1}}) \hat{f} \big) : 
	n \in \NN
	\Big)
	\Big\|_{\ell^p} \\
	&\leq
	\Big\|V_r\Big(
	\sum_{j = 1}^n 
	\calF^{-1}\big((\frkY_{N_j}-\frkY_{N_{j-1}}) \Xi_j^\beta \hat{f} \big)
	: n \in \NN
	\Big)
	\Big\|_{\ell^p} \\
	&\phantom{\leq}+
	\Big\|V_r\Big(
	\sum_{j = 1}^n
	\calF^{-1}\big((\frkY_{N_j}-\frkY_{N_{j-1}}) (1- \Xi_j^\beta)\hat{f} \big) 
	: n \in \NN \Big)
	\Big\|_{\ell^p}.
	\end{aligned}
\end{equation}
We now separately estimate each term on the right-hand side of \eqref{eq:54}. We notice that in view of \eqref{eq:121}
and \eqref{eq:10}, we have
\begin{equation}
	\label{eq:72}
	\begin{aligned}
	\big\|
	\calF^{-1}\big((\frkY_{N_{n+1}}-\frkY_{N_n}) (1 - \Xi_{n+1}^{\beta}) \hat{f}\big) 
	\Big\|_{\ell^p}
	&\lesssim
	\|f\|_{\ell^p} + 
	\big\|\calF^{-1}\big(\Xi_{n+1}^{\beta} \hat{f}\big) \big\|_{\ell^p} \\
	&\lesssim
	\log (n+1) \|f\|_{\ell^p}.
	\end{aligned}
\end{equation}
In fact, for $p = 2$, we can gain some decay in $n$. Given $\alpha > 0$, we select $\beta_\alpha$ to be determined by
Property \ref{pr:4}. Let $\beta > d \beta_\alpha$. Take any $\xi \in \TT^d$. By Dirichlet's principle, for each
$\gamma \in \Gamma$, there are coprime integers $a_\gamma$ and $q_\gamma$, such that
$1 \leq a_\gamma \leq q_\gamma \leq N_n^{\abs{\gamma}} (\log N_n)^{-\beta/d}$, and
\[
	\bigg|
	\xi_\gamma - \frac{a_\gamma}{q_\gamma}
	\bigg|
	\leq
	\frac{1}{q_\gamma} N_n^{-\abs{\gamma}} (\log N_n)^{\beta/d}.
\]
Suppose that $1 \leq q_\gamma \leq (\log N_n)^{\beta/d}$, for all $\gamma \in \Gamma$. We set
$q' = \lcm\big(q_\gamma : \gamma \in \Gamma \big)$ and $a'_\gamma = a_\gamma q'/q_\gamma$. Observe that for all
$\gamma \in \Gamma$, we have
\[
	\bigg|
	\xi_\gamma - \frac{a_\gamma}{q_\gamma}
  	\bigg|
	=
	\bigg|\xi_\gamma - \frac{a'_\gamma}{q'} \bigg| \leq N_n^{-\abs{\gamma}} (\log N_n)^{\beta/d}
	\leq
	\frac{1}{32 d} N_n^{-\abs{\gamma}} \cdot 2^{\chi \sqrt{\log N_n} },
\]
provided that
\[
	32 d (\log N_n)^{\beta/d} \leq 2^{\chi \sqrt{\log N_n}},
\]
which excludes only a finite number of $n$'s depending on $\beta$ and $\rho$. In particular, 
$\eta_{n}(\xi - a'/q') = 1$. Since $1 \leq q' \leq (\log N_n)^\beta$, $a' \in \bfA_{q'}$, we conclude that
$\Xi_{n+1}^{\beta}(\xi) = 1$. Hence, the condition $\Xi_{n+1}^{\beta}(\xi) < 1$ implies that
$(\log N_n)^{\beta/d} \leq q_\gamma \leq N_n^{\abs{\gamma}} (\log N_n)^{-\beta/d}$ for some $\gamma \in \Gamma$. Now,
by Property \ref{pr:4}, we obtain
\[
	\big| \frkY_{N_{n+1}}(\xi) - \frkY_{N_n}(\xi) \big| 
	\lesssim (\log N_n)^{-\alpha},
\]
which entails that
\begin{equation}
	\label{eq:73}
	\big\|
	\calF^{-1}\big((\frkY_{N_{n+1}}-\frkY_{N_n}) (1 - \Xi_{n+1}^{\beta}) \hat{f}\big) 
	\big\|_{\ell^2}
	\lesssim
	(\log N_n)^{-\alpha} \|f\|_{\ell^2}.
\end{equation}
Interpolation between \eqref{eq:72} and \eqref{eq:73}, shows that for each $p \in (1, \infty)$ and $\alpha > 0$
there is $\beta_{p, \alpha} > 0$ such that for all $\beta > \beta_{p, \alpha}$ and $n \in \NN$, we have
\[
	\big\|
	\calF^{-1}\big((\frkY_{N_{n+1}}-\frkY_{N_n}) (1 - \Xi_{n+1}^{\beta}) \hat{f}\big)
   	\big\|_{\ell^p}
	\leq
	C_{\alpha} (\log N_n)^{-\alpha} \|f\|_{\ell^p}.
\]
Taking $\beta > \beta_{p, 2 \rho^{-1}}$, we get
\begin{equation}
	\label{eq:56}
	\begin{aligned}
	\Big\|V_r\Big(
	\sum_{j = 1}^n
	\calF^{-1}\big((\frkY_{N_j} - \frkY_{N_{j-1}}) (1-\Xi_j^{\beta}) \hat{f} \big) 
	: n \in \NN\Big)\Big\|_{\ell^p}
	&\lesssim
	\sum_{n \geq 1} \big\|\calF^{-1}
	\big((\frkY_{N_n}-\frkY_{N_{n-1}})(1 - \Xi_n^{\beta}) \hat{f}\big) \big\|_{\ell^p} \\
	&\lesssim
	\Big(\sum_{n \geq 1} n^{-2}\Big)
	\|f\|_{\ell^p}.
	\end{aligned}
\end{equation}
We now turn to bounding the first term on the right-hand side of \eqref{eq:54}. For each $n \in \NN$ and
$s \in \{0, \ldots, n-1\}$ let us define the multiplier
\[
	\Xi_{n, s}^{\beta}(\xi) = \sum_{a/q \in \mathscr{R}_s^\beta} \eta_{n}(\xi - a/q),
\]
where $\mathscr{R}_s^\beta =\mathscr{U}_{\lfloor (s+1)^\rho \rfloor}^\beta \setminus 
\mathscr{U}_{\lfloor s^\rho \rfloor}^\beta$. By the triangle inequality we can write
\begin{equation}
	\label{eq:45}
	\begin{aligned}
	&\Big\|V_r\Big(
	\sum_{j = 1}^n
	\sum_{s = 0}^{j-1} 
	\calF^{-1}\big((\frkY_{N_j} - \frkY_{N_{j-1}}) \Xi_{j, s}^{\beta} \hat{f} \big) 
	: n \in \NN \Big)\Big\|_{\ell^p} \\
	&\qquad\qquad\leq
	\sum_{s = 0}^\infty
	\Big\|V_r\Big(\sum_{j = s+1}^n 
	\calF^{-1}\big((\frkY_{N_j} - \frkY_{N_{j-1}}) \Xi_{j, s}^{\beta} \hat{f} \big) : 
	s < n \Big)\Big\|_{\ell^p}.
	\end{aligned}
\end{equation}
Thus, the aim is to show that for each $\beta \in \NN$, $p \in (1, \infty)$, $s \in \NN_0$, and $r \in (2, \infty)$,
\[
	\Big\|V_r\Big(
	\sum_{j = s+1}^n
	\calF^{-1}\big((\frkY_{N_j}-\frkY_{N_{j-1}}) \Xi_{j, s}^{\beta} \hat{f} \big)
	: s < n \Big)\Big\|_{\ell^p}
	\lesssim
	(s+1)^{-2} \|f\|_{\ell^p}.
\]
We split the variational seminorm into two parts: $s < n \leq 2^{\kappa_s}$ and $2^{\kappa_s} < n$, where
\[
	\kappa_s = 20 d \big(\lfloor \rho^{-1} (s+1)^{\rho/10} \rfloor+1\big).
\]
We begin with $p = 2$ and $s < n \leq 2^{\kappa_s}$.
\begin{theorem}
	\label{thm:3}
	For each $\beta \in \NN$ there is $C > 0$ such that for all $s \in \NN_0$, $r \in (2, \infty)$ and any finitely
	supported function $f: \ZZ^d \rightarrow \CC$, we have
	\[
		\Big\|
		V_r\Big( 
		\sum_{j = s+1}^n
		\calF^{-1}\big((\frkY_{N_j} - \frkY_{N_{j-1}}) \Xi_{j, s}^{\beta} \hat{f} \big) 
		: s < n \leq 2^{\kappa_s}\Big)
		\Big\|_{\ell^2}
		\leq
		C
		(s+1)^{-\delta \beta \rho +2}
		\|f\|_{\ell^2},
	\]
	where $\delta$ is determined in Theorem \ref{thm:2}.
\end{theorem}
\begin{proof}
	First, let us see that for each $m > s$, supports of functions $\eta_{m}(\cdot - a/q)$ are disjoint while
	$a/q$ varies over $\mathscr{R}_s^\beta$. Indeed, otherwise there would be $a/q, a'/q' \in \mathscr{R}_s^\beta$,
	$a'/q' \neq a/q$ and $\xi \in \TT^d$, such that $\eta_{m}(\xi - a/q) > 0$ and $\eta_{m}(\xi - a'/q') > 0$. Hence,
	\begin{align*}
		e^{-2 (s+1)^{\rho/10}}
		\leq 
		\frac{1}{q q'} \leq \bigg|\frac{a}{q} - \frac{a'}{q'} \bigg| \leq \bigg|\xi_\gamma - \frac{a}{q} \bigg|
		+
		\bigg|\xi_\gamma - \frac{a_\gamma'}{q'}\bigg| \leq 2^{-m^\rho \abs{\gamma}} 2^{\chi m^{\rho/2}},
	\end{align*}
	which is impossible. 

	Next, we consider the following multiplier
	\[
		\Lambda_{n, s}^\beta(\xi) 
		= \sum_{a/q \in \mathscr{R}_s^\beta} G(a/q) \big(\Upsilon_{N_n}(\xi-a/q) - \Upsilon_{N_{n-1}}(\xi-a/q)\big)
		\eta_n(\xi - a/q).
	\]
	Let us see that $\Lambda_{n, s}^\beta$ is sufficiently close to $(\frkY_{N_n} - \frkY_{N_{n-1}}) \Xi_{j, s}^{\beta}$.
	For each $a/q \in \mathscr{R}^\beta_s$, we have $q \leq \exp\big(\tfrac{c}{2} \sqrt{\log N_n} \big)$, thus by 
	\eqref{eq:29}, on the support of $\eta_n(\cdot - a/q)$ we can write
    \begin{align*}
        (\frkY_{N_n}(\xi) - \frkY_{N_{n-1}} (\xi) )
        &=
        G(a/q) \big(\Upsilon_{N_n}(\xi-a/q) - \Upsilon_{N_{n-1}}(\xi-a/q)\big) \\
        &\phantom{=}+
        \calO\big((\log N_n)^{-\rho^{-1} - \beta\delta}\big).
    \end{align*}
	Therefore, 
	\[
		\Big\|
		\calF^{-1}\Big(\big((\frkY_{N_n} - \frkY_{N_{n-1}}) \Xi_{n, s}^\beta - \Lambda_{n, s}^\beta 
		\big)\hat{f} \Big)
		\Big\|_{\ell^2}
		\leq
		C n^{-1-\beta\delta\rho} \|f\|_{\ell^2},
	\]
	and hence,
	\begin{equation}
		\label{eq:47}
		\begin{aligned}
		&
		\Big\|
		V_r\Big(\sum_{j = s+1}^n 
		\calF^{-1}\Big(\big((\frkY_{N_j} - \frkY_{N_{j-1}}) \Xi_{j, s}^\beta
        \big)\hat{f}\Big) 
		: s < n \leq 2^{\kappa_s}
		\Big)
        \Big\|_{\ell^2} \\
        &\qquad\qquad\leq
		\Big\|
		V_r\Big(\sum_{j = s+1}^n
        \calF^{-1}\big(\Lambda_{j, s}^\beta \hat{f} \big) : s < n \leq 2^{\kappa_s} \Big)
        \Big\|_{\ell^2}
		+
		\Big(\sum_{n = s+1}^\infty n^{-1 - \beta \delta \rho}\Big) \|f\|_{\ell^2}.
		\end{aligned}
	\end{equation}
	Therefore, our task is reduced to showing boundedness of the first term on the right-hand side of \eqref{eq:47}.
	Observe that for $n > s$, $\eta_n = \eta_n \eta_s$, thus we can write
	\begin{align*}
		\Lambda_{n, s}^\beta \hat{f}
		=
		\Theta_{n, s}^\beta \hat{F},
	\end{align*}
	where
	\[
		\Theta_{n, s}^\beta(\xi) = \sum_{a/q \in \mathscr{R}_s^\beta} 
		\big(\Upsilon_{N_n}(\xi-a/q) - \Upsilon_{N_{n-1}}(\xi-a/q)\big) \eta_n(\xi - a/q),
	\]
	and
	\[
		\hat{F}(\xi) = \sum_{a/q \in \mathscr{R}_s^\beta} G(a/q) \eta_s(\xi -a/q) \hat{f}(\xi).
	\]
	Now, in view of Lemma \ref{lem:1},
	\begin{equation}
		\label{eq:36}
		\Big\| V_r\Big(
		\sum_{j = s+1}^n
		\calF^{-1}\big( \Theta_{j, s}^\beta \hat{f} \big)
		: s < n \leq 2^{\kappa_s} \Big)
		\Big\|_{\ell^2}
		\leq
		\sqrt{2}
		\sum_{i = 0}^{\kappa_s}
		\Big\|\Big(
		\sum_{j=0}^{2^{\kappa_s-i}-1}
		\big|
		\sum_{m \in I_j^i+s+1}
		\calF^{-1}\big(\Theta_{m, s}^\beta \hat{F}\big)
		\big|^2
		\Big)^{1/2}
		\Big\|_{\ell^2},
	\end{equation}
	where $I^i_j = \big\{j 2^i, j2^i+1, \ldots, (j+1)2^i-1\big\}$. Let us consider a fixed $i \in \{0, \ldots, \kappa_s\}$.
	To bound the norm of the square function on the right-hand side of \eqref{eq:36}, we first study its continuous
	counterpart, that is
	\[
		\Big(
		\sum_{j=0}^{2^{\kappa_s-i}-1}
        \big|
        \sum_{m \in I_j^i+s+1}
        \calF^{-1}\big((\Upsilon_{N_m} - \Upsilon_{N_{m-1}}) \eta_m \hat{f}\big)
        \big|^2
		\Big)^{1/2}.
	\]
	If $\eta_m(\xi) < 1$ then
	\[
		\big|\xi_\gamma\big| \geq \frac{1}{32d} N_m^{-\abs{\gamma}} 2^{\chi \sqrt{\log N_m}},
		\qquad\text{for some }\gamma \in \Gamma,
	\]
	thus by Property \ref{pr:2},
	\[
		\big|\Upsilon_{N_m}(\xi) - \Upsilon_{N_{m-1}}(\xi) \big| \lesssim \abs{N_m^A \xi}^{-1/d} 
		\lesssim 2^{-\chi\sqrt{\log N_m}/d}.
	\]
	Therefore,
	\begin{align*}
		&
		\Big\|
		\Big(
        \sum_{j=0}^{2^{\kappa_s-i}-1}
        \Big|
        \sum_{m \in I_j^i+s+1}
        \calF^{-1}\big((\Upsilon_{N_m} - \Upsilon_{N_{m-1}}) (\eta_m - 1) \calF f\big)
        \Big|^2
        \Big)^{1/2}
		\Big\|_{L^2}\\
		&\qquad\qquad\leq
		\sum_{m = s}^{2^{\kappa_s}} 
		\Big\|
		\calF^{-1}\big((\Upsilon_{N_m} - \Upsilon_{N_{m-1}}) (\eta_m - 1) \calF f\big) 
		\Big\|_{L^2}
		\\
		&\qquad\qquad\lesssim
		\Big(\sum_{m \geq 1} 2^{-\chi\sqrt{\log N_m}/d} \Big) \|f\|_{L^2}.
	\end{align*}
	Now, by Proposition \ref{prop:4}, we have
	\[
		\Big\|
        \Big(
        \sum_{j=0}^{2^{\kappa_s-i}-1}
        \Big|
        \sum_{m \in I_j^i+s+1}
        \calF^{-1}\big((\Upsilon_{N_m} - \Upsilon_{N_{m-1}}) \calF f\big)
        \Big|^2
        \Big)^{1/2}
		\Big\|_{L^2}
		\lesssim
		\|f\|_{L^2},
	\]
	thus, in view of \eqref{eq:10}, we conclude that
	\[
		\bigg\|\bigg(
        \sum_{j=0}^{2^{\kappa_s-i}-1}
        \Big|
        \sum_{m \in I_j^i+s+1}
        \calF^{-1}\big(\Theta_{m, s}^\beta \hat{F}\big)
        \Big|^2
        \bigg)^{1/2}
        \bigg\|_{\ell^2}
		\lesssim
		\log (s+2) \|F\|_{\ell^2}.
	\]
	Therefore, by \eqref{eq:36}, we arrive at the 
	\begin{equation}
		\label{eq:44}
		\Big\| V_r\Big(
        \sum_{j = s+1}^n
        \calF^{-1}\big( \Theta_{j, s}^\beta \hat{f} \big)
        : s < n \leq 2^{\kappa_s} \Big)
        \Big\|_{\ell^2}
		\lesssim
		\kappa_s \log (s+2) \|F\|_{\ell^2}.
	\end{equation}
	Finally, by Plancherel's theorem
	\[
		\big\|F\big\|_{\ell^2}^2 = 
		\sum_{a/q \in \mathscr{R}_s^\beta} \abs{G(a/q)}^2
		\int_{\TT^d} \eta_s(\xi - a/q)^2 \big|\hat{f}(\xi)\big|^2{\: \rm d} \xi,
	\]
	and hence, by Theorem \ref{thm:2},
	\[
		\big\|F\big\|_{\ell^2} \lesssim (s+1)^{-\beta \delta \rho} \|f\|_{\ell^2},
	\]
	which together with \eqref{eq:44} and \eqref{eq:47} concludes the proof.
\end{proof}
\begin{theorem}
	\label{thm:6}
	For each $\beta \in \NN$ and $p \in (1, \infty)$ there is $C > 0$, such that for all $s \in \NN_0$,
	$r \in (2, \infty)$, and any finitely supported function $f: \ZZ^d \rightarrow \CC$, we have
	\[
		\Big\|
		V_r\Big(
		\sum_{j = s+1}^n 
		\calF^{-1}\big((\frkY_{N_j} - \frkY_{N_{j-1}}) \Xi^{\beta}_{j, s} \hat{f} \big) 
		: s < n \leq 2^{\kappa_s}
		\Big)
		\Big\|_{\ell^p}
		\leq
		C
		(s+1) \log (s+2)
		\|f \|_{\ell^p}.
	\]
\end{theorem}
\begin{proof}
	For the proof, let us consider the following multiplier
	\[
		\Pi^{\beta}_{n, s} (\xi) = 
		\sum_{a/q \in \mathscr{R}_s^\beta} \big(\Upsilon_{N_n}(\xi - a/q) - \Upsilon_{N_{n-1}}(\xi-a/q)\big) 
		\eta_{s}(\xi - a/q).
	\]
	Fix $s < n_1 < n_2 \leq \min\big\{2^{\kappa_s}, 2 n_1\big\}$. Let $J_{n_1} = N_{n_1} 2^{-3 \chi\sqrt{\log N_{n_1}}}$. 
	We claim the following holds true.
	\begin{claim}
		\label{clm:5}
		For each $\beta \in \NN$ and $p \in (1, \infty)$ there is $C > 0$, such that for all 
		$n_1 \leq n \leq n_2 \leq 2n_1$,
		\begin{equation}
			\label{eq:53}
			\Big\|\calF^{-1}\Big(\big((\frkY_{N_n}-\frkY_{N_{n-1}}) \Xi^{\beta}_{n, s} 
			- \frkM_{J_{n_1}} \Pi^{\beta}_{n, s}\big) \hat{f} \Big)\Big\|_{\ell^p}
			\leq
			C	
			n^{-2}
			\|f\|_{\ell^p}.
		\end{equation}
		The constant $C$ is independent of $n_1$ and $n_2$.
	\end{claim}
	Let us first observe that, by \eqref{eq:121}, we can write
	\begin{align*}
		&\Big\|\calF^{-1}\Big(\big((\frkY_{N_n} - \frkY_{N_{n-1}}) \Xi^{\beta}_{n, s} 
		- \frkM_{J_{n_1}} \Pi^{\beta}_{n, s}\big) \hat{f}\Big) \Big\|_{\ell^p} \\
		&\qquad\qquad\leq
		\big\|
		\calF^{-1}\big((\frkY_{N_n} - \frkY_{N_{n-1}}) \Xi^{\beta}_{n, s} \hat{f}\big)
		\big\|_{\ell^p}
		+
		\big\|\calF^{-1}\big(\frkM_{J_{n_1}} \Pi^{\beta}_{n, s} \hat{f} \big)\big\|_{\ell^p}\\
		&\qquad\qquad\lesssim
		\big\|\calF^{-1}\big(\Xi^{\beta}_{n, s} \hat{f}\big)\big\|_{\ell^p}
		+
		\big\|\calF^{-1}\big(\Pi^{\beta}_{n, s} \hat{f} \big)\big\|_{\ell^p},
	\end{align*}
	thus, by \eqref{eq:10},
	\begin{equation}
		\label{eq:50}
		\Big\|\calF^{-1}\Big(\big((\frkY_{N_n} - \frkY_{N_{n-1}}) \Xi^{\beta}_{n, s} 
		- \frkM_{J_{n_1}} \Pi^{\beta}_{n, s}\big) \hat{f} \Big)\Big\|_{\ell^p}
		\lesssim
		\log(n+1) \|f\|_{\ell^p}.
	\end{equation}
	We can improve the estimate for $p = 2$. Namely, we are going to show that for each $\alpha > 0$, and
	$n_1 \leq n \leq n_2 \leq 2n_1$,
	\begin{equation}
		\label{eq:51}
		\Big\|\calF^{-1}\Big(\big((\frkY_{N_n} - \frkY_{N_{n-1}}) \Xi^{\beta}_{n, s} 
		- \frkM_{J_{n_1}} \Pi^{\beta}_{n, s}\big) \hat{f} \Big)\Big\|_{\ell^2}
        \lesssim
		n^{-\alpha \rho} \|f\|_{\ell^2}.
	\end{equation}
	Given $\alpha > 0$, let $c$ be the minimal value among those determined in Lemma \ref{lem:3} and Lemma \ref{lem:4}.
	Then for each $a/q \in \mathscr{R}_s^\beta$, 
	\begin{align*}
		&
		\big(\frkY_{N_n}(\xi) - \frkY_{N_{n-1}}(\xi)\big) \eta_{n}(\xi - a/q) 
		- \frkM_{J_{n_1}}(\xi) \big(\Upsilon_{N_n}(\xi - a/q) - \Upsilon_{N_{n-1}}(\xi - a/q)\big)
		\eta_{s}(\xi - a/q) \\
		&\qquad\qquad=
		G(a/q) \big(\Upsilon_{N_n}(\xi - a/q) - \Upsilon_{N_{n-1}}(\xi - a/q)\big) 
		\big(1-\Phi_{J_{n_1}}(\xi -a/q)\big) \eta_{n}(\xi - a/q) \\
		&\qquad\qquad\phantom{=}+
		\frkM_{J_{n_1}}(\xi) \big(\Upsilon_{N_n}(\xi-a/q) - \Upsilon_{N_{n-1}}(\xi-a/q)\big) 
		\big(\eta_{s}(\xi - a/q) - \eta_{n}(\xi - a/q)\big) \\
		&\qquad\qquad\phantom{=}+
		\calO\big((\log N_n)^{-\alpha}\big) \eta_{s}(\xi - a/q).
	\end{align*}
	If $\eta_{s}(\xi - a/q) - \eta_{n}(\xi - a/q) \neq 0$, then
	\[
		\bigg|\xi_\gamma - \frac{a_\gamma}{q} \bigg| \geq \frac{1}{32d} 
		N_n^{-\abs{\gamma}} 2^{\chi \sqrt{\log N_n}},
		\qquad\text{for some }\gamma \in \Gamma,
	\]
	thus, by Property \ref{pr:2},
	\[
		\big| \Upsilon_{N_n}(\xi - a/q) - \Upsilon_{N_{n-1}}(\xi-a/q)\big| 
		\lesssim \big|N_n^A (\xi - a/q) \big|_\infty^{-1/d}
		\lesssim
		2^{-\chi \sqrt{\log N_n}/d}.
	\]
	Moreover, if $\eta_{n}(\xi - a/q) > 0$ then
	\[
		\bigg|\xi_\gamma - \frac{a_\gamma}{q}\bigg| \leq \frac{1}{16 d} N_n^{-\abs{\gamma}} 2^{\chi \sqrt{\log N_n}}
		\leq J_{n_1}^{-\abs{\gamma}} 2^{-\chi\sqrt{\log N_{n_1}}},
		\qquad\text{for all }\gamma \in \Gamma,
	\]
	hence, by \eqref{eq:26}, we obtain
	\[
		\big|1 - \Phi_{J_{n_1}}(\xi - a/q)\big| \lesssim 
		\big| J_{n_1}^A(\xi - a/q) \big|_\infty 
		\lesssim 2^{-\chi \sqrt{\log N_{n_1}}}.
	\]
	Therefore,
	\begin{equation}
		\label{eq:48}
		\begin{aligned}
		(\frkY_{N_n}(\xi) - \frkY_{N_{n-1}}(\xi)\big) \eta_{n}(\xi - a/q)
		&=\frkM_{J_{n_1}}(\xi) \big(\Upsilon_{N_n}(\xi - a/q) - \Upsilon_{N_{n-1}}(\xi - a/q)\big)
		\eta_{s}(\xi - a/q) \\
		&\phantom{=}+\calO\big((\log N_n)^{-\alpha}\big) \eta_{s}(\xi - a/q).
		\end{aligned}
	\end{equation}
	Since the functions $\eta_{s}(\cdot - a/q)$ have disjoint supports provided that $a \in \bfA_q$ and
	$1 \leq q \leq e^{(s+1)^{\rho/10}}$, by \eqref{eq:48} and Plancherel's theorem we conclude \eqref{eq:51}.
	Now, by interpolation between \eqref{eq:51} and \eqref{eq:50} we arrive at \eqref{eq:53}. 
	
	With a help of Claim \ref{clm:5}, we obtain
	\[
		\begin{aligned}
		&
		\Big\|V_r\Big(\sum_{j = n_1}^n \calF^{-1}\Big(\big((\frkY_{N_j}-\frkY_{N_{j-1}}) \Xi^{\beta}_{j, s} 
		- \frkM_{J_{n_1}} \Pi^{\beta}_{j, s}\big) \hat{f} \Big):
		n_1 \leq n \leq n_2
		\Big)
	 	\Big\|_{\ell^p}\\
		&\qquad\qquad\lesssim
		\sum_{n = n_1}^{n_2}
		\Big\|
        \calF^{-1}\Big(\big((\frkY_{N_n} - \frkY_{N_{n-1}}) \Xi^{\beta}_{n, s} 
		- \frkM_{J_{n_1}} \Pi^{\beta}_{n, s}\big) \hat{f} \Big)
        \Big\|_{\ell^p} \\
		&\qquad\qquad\lesssim
		\Big(\sum_{n = 1}^\infty
		n^{-2}\Big) 
		\|f\|_{\ell^p}.
		\end{aligned}
	\]
	Hence,
	\begin{equation}
		\label{eq:46}
		\begin{aligned}
		&
		\Big\|
		V_r\Big(
		\sum_{j = n_1}^n 
       	\calF^{-1}\big((\frkY_{N_j}-\frkY_{N_{j-1}}) \Xi^{\beta}_{j, s} \hat{f} \big)
		:
		n_1 \leq n \leq n_2
		\Big)
        \Big\|_{\ell^p} \\
		&\qquad\qquad\lesssim
		\|f\|_{\ell^p} +
 		\Big\|V_r\Big(
		\sum_{j = n_1}^n
		\calF^{-1}\big(\frkM_{J_{n_1}} \Pi^{\beta}_{j, s} \hat{f} \big) : n_1 \leq n \leq 2n_1 \Big)
        \Big\|_{\ell^p},
		\end{aligned}
	\end{equation}
	with an implied constant independent of $n_1$. We next claim that the following holds true.
	\begin{claim}
		\label{clm:1}
		For each $\beta \in \NN$ and $p \in (1, \infty)$ there is $C > 0$, such that for all $s \in \NN_0$, we have
		\begin{equation}
			\label{eq:61}
			\Big\|
			V_r\Big(
			\sum_{j = 0}^n
			\calF^{-1}\big(\Pi^{\beta}_{j, s} \hat{f} \big) :
			0 \leq n \leq 2^{\kappa_s}
			\Big)
	        \Big\|_{\ell^p}
			\leq
			C
			\kappa_s \log (s+2)
			\|f\|_{\ell^p}.
		\end{equation}
	\end{claim}
	Let us see that \eqref{eq:61} suffices to finish the proof of the theorem. Indeed, \eqref{eq:46} together 
	with \eqref{eq:61} imply that
	\[
		\Big\|
        V_r\Big(
		\sum_{j = n_1}^n
        \calF^{-1}\Big((\frkY_{N_j} - \frkY_{N_{j-1}}) \Xi^{\beta}_{j, s} \hat{f} \big)
        :
        n_1 \leq n \leq n_2
        \Big)
        \Big\|_{\ell^p}
        \lesssim
		\kappa_s \log (s+2)
        \|f\|_{\ell^p}.
	\]
	Therefore, by \eqref{eq:162} and Minkowski's inequality
	\begin{align*}
		&\Big\|V_r\Big(\sum_{j = s+1}^n
		\calF^{-1}\big((\frkY_{N_j} - \frkY_{N_{j-1}}) \Xi^{\beta}_{j, s} \hat{f} \Big)
		 : s < n \leq 2^{\kappa_s} \big)
        \Big\|_{\ell^p} \\
		&\qquad\qquad\lesssim
		\kappa_s^{1-\frac{1}{r}}
		\Big(
		\sum_{\log_2 s \leq m < \kappa_s}
		\Big\|
		V_r\Big(
		\sum_{j = 2^m}^n
        \calF^{-1}\big((\frkY_{N_j} - \frkY_{N_{j-1}}) \Xi^{\beta}_{j, s} \hat{f} \big) :
        2^m \leq n \leq 2^{m+1}
		\Big)
		\Big\|_{\ell^p}^r
		\Big)^{1/r}\\
		&\qquad\qquad\lesssim
		\kappa_s
		\max_{\atop{s < n_1 < n_2 \leq 2 n_1}{n_2 \leq 2^{\kappa_s}}}
		\Big\|
		V_r\Big(
		\sum_{j = n_1}^n
        \calF^{-1}\big((\frkY_{N_j} - \frkY_{N_{j-1}}) \Xi^{\beta}_{j, s} \hat{f} \big) :
		n_1 \leq n \leq n_2
		\Big)
        \Big\|_{\ell^p} \\
		&\qquad\qquad\lesssim
		\kappa_s^2 \log (s+2) \|f\|_{\ell^p}.
	\end{align*}
	It remains to prove Claim \ref{clm:1}. By Lemma \ref{lem:1}, we can write
	\begin{equation}
		\label{eq:38}
		\begin{aligned}
		&\Big\|V_r\Big(
		\sum_{j = 0}^n
		\calF^{-1}\big(\Pi^{\beta}_{j, s} \hat{f} \big)
		: 0 \leq n \leq 2^{\kappa_s}
		\Big)
		\Big\|_{\ell^p} \\
		&\qquad\qquad
		\leq
		\sqrt{2}
		\sum_{i = 0}^{\kappa_s}
		\Big\|
		\Big(
		\sum_{j = 0}^{2^{\kappa_s-i}-1} 
		\Big|
		\sum_{m \in I_j^i} \calF^{-1}\Big(\big(\Pi^{\beta}_{m+1, s} - \Pi^{\beta}_{m, s}\big) \hat{f} \Big)
		\Big|^2
		\Big)^{1/2}
		\Big\|_{\ell^p},
		\end{aligned}
	\end{equation}
	where $I_j^i = \big\{j 2^i, j 2^i+1, \ldots, (j+1) 2^i-1\big\}$. Let us fix $i \in \{0, 1, \ldots \kappa_s\}$.
	In view of Proposition \ref{prop:4},
	\[
		\Big\|
		\Big(
		\sum_{j = 0}^{2^{\kappa_s-i}-1}
		\Big|
		\sum_{m \in I_j^i} \calF^{-1}\Big(\big(\Upsilon_{N_{m+1}}  - \Upsilon_{N_m}\big) \hat{f} \Big)
		\Big|^2
		\Big)^{1/2}	
		\Big\|_{L^p} \lesssim \|f\|_{\ell^p},
	\]
	where the implied constant is independent of $i$. Hence, by \eqref{eq:10}, we obtain
	\[
		\Big\|
        \Big(
        \sum_{j = 0}^{2^{\kappa_s-i}-1}
        \Big|
        \sum_{m \in I_j^i} \calF^{-1}\Big(\big(\Pi^{\beta}_{m+1, s} - \Pi^{\beta}_{m, s}\big) \hat{f} \Big)
        \Big|^2
        \Big)^{1/2}
        \Big\|_{\ell^p}
		\lesssim
		\log (s+2) \|f\|_{\ell^p},
	\]
	which together with \eqref{eq:38} implies \eqref{eq:61}.
\end{proof}

We now turn to studying the part of the variational seminorm where $2^{\kappa_s} < n$. For $s \in \NN_0$ we set
\[
	Q_s = \big(\big\lfloor e^{(s+1)^{\rho/10}} \big\rfloor\big)!
\]
\begin{theorem}
	\label{thm:4}
	For each $\beta \in \NN$ there is $C > 0$, such that for all $r \in (2, \infty)$, $s \in \NN_0$, and any finitely
	supported function $f: \ZZ^d \rightarrow \CC$, we have
	\[
		\Big\|
		V_r\Big(
		\sum_{j = 2^{\kappa_s}+1}^n
		\calF^{-1}\big((\frkY_{N_j} - \frkY_{N_{j-1}}) \Xi_{j, s}^{\beta} \hat{f} \big) :
		2^{\kappa_s} < n
		\Big)
		\Big\|_{\ell^2}
		\leq
		C
		\frac{r}{r-2}
		(s+1)^{-\delta \beta \rho}
		\|f\|_{\ell^2},
	\]
	where $\delta$ is determined in Theorem \ref{thm:2}.
\end{theorem}
\begin{proof}
	Let us define
	\[
		\Omega_{n, s}^{\beta} = 
		\sum_{a/q \in \mathscr{R}_s^\beta} G(a/q) \big( \Upsilon_{N_n}(\xi - a/q) - \Upsilon_{N_{n-1}}(\xi - a/q)\big) 
		\varrho_s(\xi - a/q),
	\]
	where
	\[
		\varrho_s(\xi) = \eta\big(Q_{s+1}^{3d A} \xi \big).
	\]
	Our first goal is to show that the multipliers $\Omega_{n, s}^\beta$ approximate 
	$(\frkY_{N_n} - \frkY_{N_{n-1}}) \Xi_{n, s}^{\beta}$ well.
	\begin{claim}
		\label{clm:2}
		For each $\beta \in \NN$ there is $C > 0$, such that for all $s \in \NN_0$, and $n > 2^{\kappa_s}$,
		\begin{equation}
			\label{eq:49}
			\Big\|
			\calF^{-1}\Big(\big((\frkY_{N_n} - \frkY_{N_{n-1}}) \Xi_{n, s}^{\beta} 
			- \Omega_{n, s}^\beta \big) \hat{f} \Big)
			\Big\|_{\ell^2}
			\leq
			C \cdot 2^{-\chi \sqrt{\log N_n}/d} \|f\|_{\ell^2}.
		\end{equation}
	\end{claim}
	Since $n > 2^{\kappa_s}$, for each $a/q \in \mathscr{R}_s^\beta$ we have $q \leq \log N_n$. Therefore, by
	\eqref{eq:33}, we obtain
	\begin{align*}
		&\big(\frkY_{N_n}(\xi) - \frkY_{N_{n-1}}(\xi)\big)\eta_{n}(\xi - a/q) 
		- G(a/q) \big(\Upsilon_{N_n}(\xi-  a/q) - \Upsilon_{N_{n-1}}(\xi - a/q) \big)\varrho_s(\xi - a/q) \\
		&\qquad=
		G(a/q) \big(\Upsilon_{N_n}(\xi-a/q) - \Upsilon_{N_{n-1}}(\xi-a/q)\big) 
		\big(\eta_{n}(\xi - a/q) - \varrho_s(\xi - a/q)\big)\\
		&\qquad\phantom{=}+
		\calO\Big(\exp\big((\chi \log 2- c)\sqrt{\log N_n}\big)\Big).
	\end{align*}
	Next, if $\varrho_s(\xi - a/q) - \eta_{n}(\xi-a/q) \neq 0$, then
	\[
		\bigg|\xi_\gamma - \frac{a_\gamma}{q} \bigg| 
		\geq 
		\frac{1}{32d} N_n^{-\abs{\gamma}} 2^{\chi \sqrt{\log N_n}},
		\qquad\text{for some }\gamma \in \Gamma,
	\]
	and thus, by \eqref{eq:170}, we have
	\[
		\big|\Upsilon_{N_n}(\xi - a/q) - \Upsilon_{N_{n-1}}(\xi -a/q)\big| 
		\lesssim \big|N_n^A (\xi - a/q) \big|^{-1/d}_\infty
		\lesssim
		2^{-\chi \sqrt{\log N_n}/d}.
	\]
	Hence,
	\begin{align*}
		\big(\frkY_{N_n}(\xi) - \frkY_{N_{n-1}}(\xi)\big)\eta_{n}(\xi - a/q)
        &= G(a/q) \big(\Upsilon_{N_n}(\xi-  a/q) - \Upsilon_{N_{n-1}}(\xi - a/q) \big)\varrho_s(\xi - a/q) \\
		&\phantom{=}+
		\calO\Big(2^{-\chi\sqrt{\log N_n}/d}\Big).
	\end{align*}
	Since the functions $\eta_{s}(\cdot - a/q)$ have disjoint supports while $a/q$ varies over
	$\mathscr{R}_s^\beta$, by Plancherel's theorem we obtain \eqref{eq:49}.

	Now, by applying Claim \ref{clm:2}, 
	\begin{align*}
		&
		\Big\|
        V_r\Big(
        \sum_{j = 2^{\kappa_s}+1}^n
        \calF^{-1}\Big(\big((\frkY_{N_j} - \frkY_{N_{j-1}}) \Xi_{j, s}^{\beta}
        - \Omega_{j, s}^\beta \big)\hat{f} \Big)
        :
        2^{\kappa_s} < n
        \Big)
        \Big\|_{\ell^2} \\
		&\qquad\qquad
		\lesssim
        (s+1)^{-\delta \beta \rho}
        \Big(\sum_{n = 2^{\kappa_s}+1}^\infty 2^{-\chi \sqrt{\log N_n}/(2d)} \Big)\|f\|_{\ell^2},
	\end{align*}
	thus
	\begin{align*}
		&\Big\|
		V_r\Big(
		\sum_{j = 2^{\kappa_s}+1}^n
		\calF^{-1}\big((\frkY_{N_j} - \frkY_{N_{j-1}}) \Xi_{j, s}^{\beta} \hat{f} \big) :
		2^{\kappa_s} < n
		\Big)
        \Big\|_{\ell^2} \\
		&\qquad\leq
		(s+1)^{-\delta \beta \rho} \|f\|_{\ell^p} 
		+
		\Big\|
		V_r\Big(
		\sum_{j = 2^{\kappa_s}+1}^n
		\calF^{-1}\big(\Omega_{j, s}^\beta \hat{f} \big)
		: 
		2^{\kappa_s} < n
		\Big)
        \Big\|_{\ell^2}.
	\end{align*}
	Our next task is to show that there is $C > 0$ such that
	\begin{equation}
		\label{eq:75}
		\Big\|
		V_r\Big(
		\sum_{j = 2^{\kappa_s}+1}^n
		\calF^{-1}\big(\Omega_{j, s}^\beta \hat{f} \big) : 
		2^{\kappa_s} < n\Big)
        \Big\|_{\ell^2}
		\leq
		C (s+1)^{-\delta \beta\rho} \|f\|_{\ell^2}.
	\end{equation}
	For the proof, let us define
	\[
		I(x, y) = V_r\Big(
		\sum_{a/q \in \mathscr{R}_s^\beta} G(a/q) e^{-2\pi i \sprod{(a/q)}{x}} 
		\sum_{j = 2^{\kappa_s}+1}^n 
		\calF^{-1} \Big(\big(\Upsilon_{N_j} - \Upsilon_{N_{j-1}} \big)\varrho_s \hat{f}\big(\cdot + a/q)\big)\Big)(y)
		:
		2^{\kappa_s} < n 
		\Big),
	\]
	and
	\[
		J(x, y) = \sum_{a/q \in \mathscr{R}_s^\beta}
		G(a/q) e^{-2 \pi i \sprod{(a/q)}{x}} \calF^{-1}\Big(\varrho_s \hat{f}\big(\cdot + a/q)\big) \Big)(y).
	\]
	By Plancherel's theorem, for any $u \in \NN^d_{Q_s}$ and $a/q \in \mathscr{R}_s^\beta$, we have
	\begin{align*}
		&
		\Big\|
		\calF^{-1}\Big(\big(\Upsilon_{N_j} - \Upsilon_{N_{j-1}}\big) \varrho_s \hat{f}(\cdot + a/q)\Big)(x+u) 
		-
		\calF^{-1}\Big(\big(\Upsilon_{N_j} - \Upsilon_{N_{j-1}}\big) \varrho_s \hat{f}(\cdot + a/q)\Big)(x)
		\Big\|_{\ell^2(x)} \\
		&\qquad\qquad\qquad\qquad=
		\Big\|
		\big(1 - e^{- 2\pi i \sprod{\xi}{u}}\big) 
		\big(\Upsilon_{N_j}(\xi) - \Upsilon_{N_{j-1}}(\xi)\big) \varrho_s(\xi) \hat{f}(\xi + a/q)
		\Big\|_{L^2({\rm d} \xi)} \\
		&\qquad\qquad\qquad\qquad\lesssim
		N_j^{-1/d} \norm{u} \cdot \|\varrho_s(\cdot + a/q) \hat{f} \|_{L^2},
	\end{align*}
	because in view of \eqref{eq:170}, for each $\xi \in \TT^d$,
	\[
		\norm{\xi} \cdot \big|\Upsilon_{N_j}(\xi) - \Upsilon_{N_{j-1}}(\xi)\big|
		\lesssim
		N_j^{-1/d} \norm{\xi}_\infty^{1-1/d}
		\lesssim
		N_j^{-1/d}.
	\]
	Therefore,
	\[
		\Big|
		\big\|I(x, x+u)\big\|_{\ell^2(x)} - \big\|I(x, x)\big\|_{\ell^2(x)}
		\Big|
		\lesssim
		\norm{u} \Big( \sum_{j = 2^{\kappa_s}+1}^\infty N_j^{-1/d} \Big)
		\sum_{a/q \in \mathscr{R}_s^\beta} 
		\big\|\varrho_s(\cdot - a/q) \hat{f}\big\|_{L^2}.
	\]
	Since the set $\mathscr{R}_s^\beta$ has at most $e^{(d+1)(s+1)^{\rho/10}}$ elements, and
	\[
		(d+1) (s+1)^{\rho/10} + (s+1)^{\rho/10} e^{(s+1)^{\rho/10}} - \frac{\log 2}{2d} 2^{\kappa_s} \leq
		-(s+1)^\rho,
	\]
	we obtain
	\[
		\big\|
		I(x, x)
		\big\|_{\ell^2(x)} \lesssim
		\big\|
		I(x, x+u)
		\big\|_{\ell^2(x)} + 2^{-(s+1)^\rho} \|f\|_{\ell^2}.
	\]
	Hence,
	\begin{equation}
		\label{eq:74}
		\Big\|
		V_r\Big(
		\sum_{j = 2^{\kappa_s}+1}^n \calF^{-1}\big(\Omega_{j, s}^\beta \hat{f} \big)
		: 2^{\kappa_s} < n 
		\Big)
		\Big\|_{\ell^2}^2
		\lesssim
		\frac{1}{Q_s^d} \sum_{u \in \NN_{Q_s}^d} \big\|I(x, x+u)\big\|_{\ell^2(x)}^2 + 2^{-2(s+1)^\rho} \|f\|_{\ell^2}^2.
	\end{equation}
	Let us observe that the functions $x \mapsto I(x, y)$ and $x \mapsto J(x, y)$ are $Q_s\ZZ^d$-periodic. Therefore,
	by repeated change of variables, we get
	\begin{align*}
		\sum_{u \in \NN_{Q_s}^d} 
		\big\|I(x, x+u)\big\|_{\ell^2(x)}^2 = \sum_{x \in \ZZ^d} \sum_{u \in \NN_{Q_s}^d} I(x - u, x)^2
		=
		\sum_{x \in \ZZ^d} \sum_{u \in \NN_{Q_s}^d} I(u, x)^2 = \sum_{u \in \NN_{Q_s}^d} 
		\big\|I(u, x)\big\|_{\ell^2(x)}^2.
	\end{align*}
	By \cite[Proposition 4.1]{mst1} (see also \cite[Proposition 3.2]{mt3}), Property \ref{pr:1} entails
	that for each $u \in \NN_{Q_s}^d$, we have
	\begin{align*}
		\big\|I(u, x)\big\|_{\ell^2(x)}
		&=
		\Big\|
		V_r\Big(
		\sum_{j = 2^{\kappa_s}+1}^n
		\calF^{-1}\Big( \big( \Upsilon_{N_j} - \Upsilon_{N_{j-1}}\big) J(u, \hat{\cdot}) \Big)
		: 2^{\kappa_s} < n \Big) 
        \Big\|_{\ell^2(x)} \\
		&
		\leq C \frac{r}{r-2}
		\big\|J(u, x) \big\|_{\ell^2(x)}.
	\end{align*}
	Observe that
	\[
        \sum_{u \in \NN_{Q_s}^d} \big\|J(u, x) \big\|_{\ell^2(x)}^2 =
        \sum_{u \in \NN_{Q_s}^d}
        \big\|J(x, x+u)\big\|_{\ell^2(x)}^2.
    \]
	Since by Theorem \ref{thm:2} and disjointness of supports of $\varrho_s(\cdot - a/q)$ while $a/q$ varies over
    $\mathscr{R}_s^\beta$, we get
	\begin{align*}
		\big\|J(x, x+u)\big\|_{\ell^2(x)}^2 
		&= \int_{\TT^d} 
		\Big|
		\sum_{a/q \in \mathscr{R}_s^\beta} G(a/q) e^{2\pi i \sprod{(a/q)}{u}} \varrho_s(\xi- a/q)
		\Big|^2
		|\hat{f}(\xi)|^2
		{\: \rm d} \xi \\
		&\lesssim
		(s+1)^{-2\delta\beta\rho} \|f\|_{\ell^2}^2,
	\end{align*}
	we obtain
	\[
		\sum_{u \in \NN_{Q_s}^d}
        \big\|I(x, x+u)\big\|_{\ell^2(x)}^2
		\lesssim
		\bigg(\frac{r}{r-2}\bigg)^2
		(s+1)^{-2\delta \beta\rho} Q_s^d \|f\|_{\ell^2}^2,
	\]
	which together with \eqref{eq:74} implies \eqref{eq:75} and the proof of theorem is completed.
\end{proof}

\begin{theorem}
	\label{thm:7}
	For each $\beta \in \NN$ and $p \in (1, \infty)$ there is $C > 0$, such that for all $s \in \NN_0$,
	$r \in (2, \infty)$, and any finitely supported function
	$f: \ZZ^d \rightarrow \CC$, 
	\[
		\Big\|
		V_r\Big(
		\sum_{j = 2^{\kappa_s}+1}^n 
		\calF^{-1}\big((\frkY_{N_j} - \frkY_{N_{j-1}}) \Xi_{j, s}^{\beta} \hat{f} \big)
		: 2^{\kappa_s} < n
		\Big)
		\Big\|_{\ell^p}
		\leq
		C
		\frac{r}{r-2}
		\log (s+2)
		\|f\|_{\ell^p}.
	\]
\end{theorem}
\begin{proof}
	First, we are going to refine Claim \ref{clm:2}.
	\begin{claim}
		\label{clm:3}
		For each $\beta \in \NN$ and $p \in (1, \infty)$ there is $c_p > 0$ such that for all $s \in \NN_0$, and
		$n > 2^{\kappa_s}$,
		\begin{equation}
			\label{eq:63}
			\Big\|
			\calF^{-1}\Big(\big((\frkY_{N_n} - \frkY_{N_{n-1}}) \Xi_{n, s}^{\beta} 
			- \Omega_{n, s}^\beta \big) \hat{f} \Big)
			\Big\|_{\ell^p}
			\leq
			C
			\cdot 2^{-\chi c_p \sqrt{\log N_n}}
			\|f\|_{\ell^p}.
		\end{equation}
	\end{claim}
	We notice the following trivial bound 
	\[
		\big\|
        \calF^{-1}\big(\Omega_{n, s}^\beta \hat{f} \big)
        \big\|_{\ell^p}
		\lesssim
		e^{(d+1)(s+1)^{\rho/10}} \|f\|_{\ell^p} 
		\leq
		(\log N_n) \|f\|_{\ell^p},
	\]
	thus, by \eqref{eq:121} and \eqref{eq:10}, we also have
	\begin{align}
		\nonumber
		\Big\|
		\calF^{-1}\Big(\big((\frkY_{N_n} - \frkY_{N_{n-1}}) \Xi_{n, s}^\beta - \Omega_{n, s}^\beta \big) \hat{f} \Big)
        \Big\|_{\ell^p}
		&\leq
		\big\|
		\calF^{-1}\big((\frkY_{N_n} - \frkY_{N_{n-1}}) \Xi_{n, s}^\beta \hat{f} \big)
        \big\|_{\ell^p}
		+
		\big\|
		\calF^{-1}\big(\Omega_{n, s}^\beta \hat{f} \big)
        \big\|_{\ell^p} \\
		\label{eq:60}
		&\lesssim
		(\log N_n)
		\|f\|_{\ell^p}.
	\end{align}
	Now, interpolation between \eqref{eq:60} and \eqref{eq:49} leads to \eqref{eq:63}.

	Next, using Claim \ref{clm:3}, we obtain
	\begin{align*}
		&
		\Big\|
		V_r\Big(
		\sum_{j = 2^{\kappa_s}+1}^n
        \calF^{-1}\Big(\big((\frkY_{N_j} - \frkY_{N_{j-1}}) \Xi_{j, s}^{\beta} 
		- \Omega_{j, s}^\beta \big) \hat{f} \Big)
        : 2^{\kappa_s} < n 
		\Big)
        \Big\|_{\ell^p} \\
		&\qquad\qquad\lesssim
		\sum_{n = 2^{\kappa_s}+1}^\infty
		\Big\|
        \calF^{-1}\Big(\big((\frkY_{N_n} -\frkY_{N_{n-1}}) \Xi_{n, s}^{\beta} 
		- \Omega_{n, s}^\beta \big) \hat{f} \Big)
        \Big\|_{\ell^p} \\ 
		&\qquad\qquad\lesssim
		\Big(\sum_{n = 2^{\kappa_s}+1}^\infty
		2^{-\chi c_p \sqrt{\log N_n}} \Big)\|f\|_{\ell^p}.
	\end{align*}
	Hence,
	\begin{align*}
		&
		\Big\|
		V_r\Big(
		\sum_{j = 2^{\kappa_s}+1}^n
        \calF^{-1}\big((\frkY_{N_j} - \frkY_{N_{j-1}}) \Xi_{j, s}^{\beta} \hat{f} \big)
		: 2^{\kappa_s} < n
		\Big)
        \Big\|_{\ell^p} \\
		&\qquad\qquad\lesssim
		\|f\|_{\ell^p}
		+
		\Big\|
		V_r\Big(
		\sum_{j = 2^{\kappa_s}+1}^n
        \calF^{-1}\big(\Omega_{j, s}^\beta \hat{f} \big)
        : 2^{\kappa_s} < n
		\Big)
        \Big\|_{\ell^p},
	\end{align*}	
	and the proof is reduced to showing the following claim.
	\begin{claim}
		For each $\beta \in \NN$ and $p \in (1, \infty)$ there is $C > 0$ such that for all $r \in (2, \infty)$,
		and $s \in \NN_0$,
		\[
			\Big\|
			V_r\Big(
			\sum_{j = 2^{\kappa_s}+1}^n
			\calF^{-1}\big(\Omega_{j, s}^\beta \hat{f} \big)
			: 
			2^{\kappa_s} < n
			\Big)
			\Big\|_{\ell^p}
			\leq
			C \frac{r}{r-2} \log (s+2)
			\|f\|_{\ell^p}.
		\]
	\end{claim}
	For any $a/q \in \mathscr{R}_s^\beta$, $x \in \ZZ^k$ and $m \in \NN_{Q_s}^k$, we have
	\begin{align*}
		&\calF^{-1}\Big(\big(\Upsilon_{N_n}(\cdot -a/q) - \Upsilon_{N_{n-1}}(\cdot - a/q) \big) 
		\varrho_s(\cdot - a/q) \hat{f}\Big)(Q_s x + m) \\
		&\qquad\qquad=
		\calF^{-1}\Big(\big(\Upsilon_{N_n} - \Upsilon_{N_{n-1}}\big) 
		\varrho_s \hat{f}(\cdot + a/q)\Big)(Q_s x+m) e^{-2\pi i \sprod{(a/q)}{m}}. 
	\end{align*}
	Therefore,
	\begin{align*}
		&\Big\|
		V_r\Big(
		\sum_{j = 2^{\kappa_s}+1}^n
		\calF^{-1}\big(
		\Omega_{j, s}^\beta \hat{f}
		\big)
		: 2^{\kappa_s} < n
		\Big)
		\Big\|_{\ell^p}^p \\
		&\qquad\qquad=
		\sum_{m \in \NN_{Q_s}^k}
		\Big\|
		V_r\Big(
		\sum_{j = 2^{\kappa_s}+1}^n
        \calF^{-1}\Big(\big(\Upsilon_{N_j} - \Upsilon_{N_{j-1}}\big) \varrho_s F(\cdot; m)\Big) (Q_s x + m) 
		: 2^{\kappa_s} < n
		\Big)	
		\Big\|_{\ell^p(x)}^p,
	\end{align*}
	where
	\[
		F(\xi; m) = \sum_{a/q \in \mathscr{R}_s^\beta} G(a/q) \hat{f}(\xi + a/q) e^{-2\pi i \sprod{(a/q)}{m}}.
	\]
	By \cite[Proposition 4.2]{mst1} (see also \cite[Proposition 3.2]{mt3}), we can write
	\begin{align*}
		&\sum_{m \in \NN_{Q_s}^k}
        \Big\|
        V_r\Big(
		\sum_{j = 2^{\kappa_s}+1}^n
        \calF^{-1}\Big(\big(\Upsilon_{N_j} - \Upsilon_{N_{j-1}}\big) 
		\varrho_s F(\cdot; m)\Big) (Q_s x + m) : 2^{\kappa_s} < n \Big)
        \Big\|_{\ell^p(x)}^p\\
		&\qquad\qquad\leq
		\bigg(\frac{C r}{r-2}\bigg)^p
		\Big\|
		\sum_{a/q \in \mathscr{R}_s^\beta} G(a/q) \calF^{-1}\big(
		\varrho_s(\cdot - a/q) \hat{f}\big)
		\Big\|_{\ell^p}^p.
	\end{align*}
	Therefore, the problem is reduced to showing
	\begin{equation}
		\label{eq:70}
		\Big\|\sum_{a/q \in \mathscr{R}_s^\beta} G(a/q) \calF^{-1} \big(\varrho_s(\cdot - a/q) \hat{f} \big)
		\Big\|_{\ell^p}
		\leq
		C \log (s+2) \|f\|_{\ell^p}.
	\end{equation}
	For the proof, let $N = \lfloor e^{(s+1)^{\rho/10}} \rfloor+1$ and $J = 2^N$. We write
	\[
		\begin{aligned}
		\Big\|\sum_{a/q \in \mathscr{R}_s^\beta} G(a/q) \calF^{-1} \big(\varrho_s(\cdot - a/q) \hat{f} \big)
        \Big\|_{\ell^p}
        &\leq
		\Big\|\sum_{a/q \in \mathscr{R}_s^\beta} 
		\calF^{-1} \Big( \big( \frkM_J - G(a/q)\big) \varrho_s(\cdot - a/q) \hat{f} \Big)
        \Big\|_{\ell^p} \\
		&\phantom{\leq}+
		\Big\|\sum_{a/q \in \mathscr{R}_s^\beta} 
		\calF^{-1} \big( \frkM_J \varrho_s(\cdot - a/q) \hat{f} \big)
        \Big\|_{\ell^p}.
		\end{aligned}
	\]
	In view of \eqref{eq:10}, we have
	\begin{align}
		\nonumber
		\Big\|\sum_{a/q \in \mathscr{R}_s^\beta}
        \calF^{-1} \big( \frkM_J \varrho_s(\cdot - a/q) \hat{f} \big)
        \Big\|_{\ell^p}
		&\leq
		\Big\|\sum_{a/q \in \mathscr{R}_s^\beta}
        \calF^{-1} \big(\varrho_s(\cdot - a/q) \hat{f} \big)
        \Big\|_{\ell^p} \\
		\label{eq:69}
		&\lesssim \log (s+2) \|f\|_{\ell^p}.
	\end{align}
	Next, we have the following trivial bound
	\begin{equation}
		\label{eq:67}
		\begin{aligned}
		\Big\|\sum_{a/q \in \mathscr{R}_s^\beta}
        \calF^{-1} \Big( \big( \frkM_J - G(a/q)\big) \varrho_s(\cdot - a/q) \hat{f} \Big)
        \Big\|_{\ell^p}
		&\leq
		e^{(d+1)(s+1)^{\rho/10}} \|f\|_{\ell^p} \\
		&\lesssim
		(\log J)^{d+1} \|f\|_{\ell^p}.
		\end{aligned}
	\end{equation}
	We want to improve the above estimate for $p =2$. We have
	\begin{equation}
		\label{eq:93}
		\begin{aligned}
		&
		\Big\|\sum_{a/q \in \mathscr{R}_s^\beta}
        \calF^{-1} \Big( \big( \frkM_J - G(a/q)\big) \varrho_s(\cdot - a/q) \hat{f} \Big)
        \Big\|_{\ell^2}^2 \\
		&\qquad\qquad=
		\sum_{a/q \in \mathscr{R}_s^\beta}
		\int_{\TT^d}
		\Big|\frkM_J(\xi) - G(a/q) \Big|^2 \varrho_s(\xi - a/q)^2 |\hat{f}(\xi)|^2
		{\: \rm d}\xi.
		\end{aligned}
	\end{equation}
	Since each fraction $a/q$ belonging to $\mathscr{R}_s^\beta$ has its denominator
	$q \leq e^{(s+1)^{\rho/10}} \leq \log J$, by Proposition \ref{prop:1},
	\begin{equation}
		\label{eq:66}
		\begin{aligned}
		\big(\frkM_J(\xi) - G(a/q)\big) \varrho_s(\xi - a/q) 
		&=
		G(a/q) \big(\Phi_J(\xi - a/q) - 1\big) \varrho_s(\xi - a/q) \\
		&\phantom{=}+
		\calO\Big(\exp\big(-c\sqrt{\log J}\big)\Big).
		\end{aligned}
	\end{equation}
	If $\varrho_s(\xi - a/q) > 0$ then
	\[
		\bigg|\xi_\gamma - \frac{a_\gamma}{q} \bigg| \leq Q_{s+1}^{-3d \abs{\gamma}}
		\leq
		J^{-2 \abs{\gamma}},
		\qquad\text{for all }\gamma \in \Gamma, 
	\]
	thus, by \eqref{eq:26}, we get
	\[
		\big|\Phi_J(\xi - a/q) - 1\big| \lesssim \big|J^A(\xi - a/q) \big|_{\infty}
		\lesssim J^{-1}.
	\]
	Hence, \eqref{eq:66} takes the following form
	\[
		\big(\frkM_J(\xi) - G(a/q)\big) \varrho_s(\xi - a/q)
		=
		\calO\Big(\exp\big(-c\sqrt{\log J}\big)\Big).
	\]
	Therefore, by \eqref{eq:93}, we get 
	\begin{equation}
		\label{eq:68}
		\Big\|\sum_{a/q \in \mathscr{R}_s^\beta}
        \calF^{-1} \Big( \big( \frkM_J - G(a/q)\big) \varrho_s(\cdot - a/q) \hat{f} \Big)
        \Big\|_{\ell^2}
		\lesssim
		e^{-c \sqrt{\log J}} 
		\|f\|_{\ell^2}.
	\end{equation}
	Now, interpolating \eqref{eq:67} with \eqref{eq:68}, we obtain
	\[
		\Big\|\sum_{a/q \in \mathscr{R}_s^\beta}
        \calF^{-1} \Big( \big( \frkM_J - G(a/q)\big) \varrho_s(\cdot - a/q) \hat{f} \Big)
        \Big\|_{\ell^p}
		\lesssim
		\|f\|_{\ell^p},
	\]
	which together with \eqref{eq:69} implies \eqref{eq:70}, and the proof of the theorem is completed.
\end{proof}

\begin{theorem}
	\label{thm:5}
	For each $p \in (1, \infty)$ and $\rho \in (0, 1)$, there is $C > 0$ such that for all $r \in (2, \infty)$
	and any finitely supported function $f: \ZZ^d \rightarrow \CC$,
	\[
		\big\|
		V_r\big(Y_{N_n} f : n \in \NN_0 \big) \big\|_{\ell^p}
		\leq
		C \frac{r}{r-2} 
		\|f\|_{\ell^p},
    \]
	where $N_n = \big\lfloor 2^{n^\rho} \big\rfloor$.
\end{theorem}
\begin{proof}
	In view of \eqref{eq:54}, \eqref{eq:56} and \eqref{eq:45}, we have
	\[
		\begin{aligned}
		&
		\Big\|
		V_r\Big(\sum_{j = 1}^n \calF^{-1}\Big(\big(\frkY_{N_j}-\frkY_{N_{j-1}}\big) \hat{f} \Big): n \in \NN\big)
		\Big\|_{\ell^p} \\
		&\qquad\qquad\leq
		C_{p, \beta} \|f\|_{\ell^p} +
		\sum_{s = 0}^\infty 
		\Big\|
		V_r\Big(\sum_{j = s+1}^n \calF^{-1}\Big(\big(\frkY_{N_j}-\frkY_{N_{j-1}}\big) \Xi^{\beta}_{j, s} 
		\hat{f} \Big)
		: s < n
		\Big)
        \Big\|_{\ell^p},
		\end{aligned}
	\]
	provided $\beta > \beta_{p, 2 \rho^{-1}}$. Next, we split the index set
	\begin{align*}
		&
		\Big\|V_r\Big( \sum_{j = s}^n \calF^{-1}\Big(\big(\frkY_{N_j} - \frkY_{N_{j-1}}\big)
		\Xi^{\beta}_{j, s} \hat{f} \big)
		: s < n
		\Big)
        \Big\|_{\ell^p} \\
		&\qquad\qquad\lesssim
		\Big\|V_r\Big(
		\sum_{j = s+1}^n
		\calF^{-1}\Big(\big(\frkY_{N_j} - \frkY_{N_{j-1}}\big) \Xi^{\beta}_{j, s} \hat{f} \Big)
		: s < n \leq 2^{\kappa_s}
		\Big)
        \Big\|_{\ell^p} \\
		&\qquad\qquad\phantom{\lesssim}+
		\Big\|
		V_r\Big(
		\sum_{j = 2^{\kappa_s}+1}^n
		\calF^{-1}\Big(
		\big(\frkY_{N_j} - \frkY_{N_{j-1}}\big) \Xi^{\beta}_{j, s} \hat{f} \Big)
		: 2^{\kappa_s} < n 
		\Big)
        \Big\|_{\ell^p}.
	\end{align*}
	By interpolation between Theorem \ref{thm:3} and Theorem \ref{thm:6}, and between Theorem \ref{thm:4} and
	Theorem \ref{thm:7}, for $\beta$ sufficiently larger we get
	\[
		\Big\|
		V_r\Big(
		\sum_{j = s+1}^n
		\calF^{-1}\Big(\big(\frkY_{N_j} - \frkY_{N_{j-1}}\big) \Xi^{\beta}_{j, s} \hat{f} \Big)
		:
		s < n \leq 2^{\kappa_s}
		\Big)
		\Big\|_{\ell^p}
		\leq
		C_p
		(s+1)^{-2} \|f\|_{\ell^p},
	\]
	and
	\[
		\Big\|
		V_r\Big(
		\sum_{j = 2^{\kappa_s}+1}^n
		\calF^{-1}\Big(\big(\frkY_{N_j} - \frkY_{N_{j-1}}\big) \Xi^{\beta}_{j, s} \hat{f} \Big)
		: 2^{\kappa_s} < n
		\Big)
		\Big\|_{\ell^p}
		\leq
		C_p \frac{r}{r-2}
		(s+1)^{-2} \|f\|_{\ell^p},
	\]
	and the theorem follows.
\end{proof}

\subsection{Short variations} 
\label{sec:5}

\begin{theorem}
	\label{thm:8}
	For each $p \in (1, \infty)$ there are $\rho \in (0, 1)$ and  $C > 0$ such that for all $r \in (2, \infty)$
	and any finitely supported function $f : \ZZ^d \rightarrow \CC$, we have
	\[
		\Big\|
		\Big(
		\sum_{n \geq 0}
		V_r\big( \big(Y_N - Y_{N_n}\big)f : N \in [N_n, N_{n+1})\big)^r
		\Big)^{1/r}
		\Big\|_{\ell^p}
		\leq
		C
		\|f\|_{\ell^p}.
	\]
\end{theorem}
\begin{proof}
	Let $u = \min\{2, p\}$. By monotonicity and Minkowski's inequality, we get
	\begin{align*}
		\Big\|
        \Big(\sum_{n \geq 0} V_r\big(Y_N f : N \in [N_n, N_{n+1}) \big)^r
        \Big)^{\frac{1}{r}}
        \Big\|_{\ell^p}
        \leq
        \Big\|
        \Big(
        \sum_{n \geq 0}
        \Big(
        \sum_{N = N_n}^{N_{n+1}-1} 
		\big|Y_{N+1} f - Y_N f \big|
		\Big)^u
		\Big)^{\frac{1}{u}}
		\Big\|_{\ell^p} \\
		\leq
        \Big(\sum_{n \geq 0}
		\Big\|
		\sum_{N = N_n}^{N_{n+1}-1} 
		\big|Y_{N+1}f - Y_N f \big|
		\Big\|_{\ell^p}^u
        \Big)^{\frac{1}{u}},
        \end{align*}
        which together with \eqref{eq:28} gives
        \[
        	\Big\|
        	\Big(\sum_{n \geq 0} V_r\Big(Y_N f : N \in [N_n, N_{n+1}) \Big)^r
        	\Big)^{\frac{1}{r}}
        	\Big\|_{\ell^p}
        	\lesssim
            \Big(\sum_{n \geq 1} n^{-u(1-\rho)} \Big)^{\frac{1}{u}} 
			\|f\|_{\ell^p}.
        \]
		which is bounded whenever $0 < \rho < \frac{u-1}{u}$. 
\end{proof}

\begin{bibliography}{discrete}
	\bibliographystyle{amsplain}
\end{bibliography}

\end{document}